\documentclass{amsart}
\usepackage{amsfonts,amscd,amsthm,amsgen,amsmath,amssymb}
\usepackage[all]{xy}
\usepackage[vcentermath]{youngtab}
\newtheorem{theorem}{Theorem}[section]
\newtheorem{lemma}[theorem]{Lemma}

\newtheorem{proposition}[theorem]{Proposition}

\theoremstyle{definition}
\newtheorem{definition}[theorem]{Definition}

\theoremstyle{remark}
\newtheorem{remark}[theorem]{Remark}

\numberwithin{equation}{section}

\begin{document}

\title{Transitive Courant algebroids, string structures and T-duality}
\author{David Baraglia}
\author{Pedram Hekmati}

\address{School of Mathematical Sciences, The University of Adelaide, Adelaide SA 5005, Australia}

\email{david.baraglia@adelaide.edu.au}
\email{pedram.hekmati@adelaide.edu.au}

\begin{abstract}In this paper, we use reduction by extended actions to give a construction of transitive Courant algebroids from string classes. We prove that T-duality commutes with the reductions and thereby determine global conditions for the existence of T-duals in heterotic string theory. In particular we find that  T-duality exchanges string structures and gives an isomorphism of transitive Courant algebroids. Consequently we derive the T-duality transformation for generalised metrics and show that the heterotic Einstein equations are preserved. The presence of string structures significantly extends the domain of applicability of T-duality and this is illustrated by several classes of examples.
\end{abstract}
\thanks{This work is supported by the Australian Research Council Discovery Project DP110103745 and DE12010265.}





\maketitle


\section{Introduction}

The bosonic fields in the low energy limit of type II string theories consist of a metric $g$, closed $3$-form $H$ and dilaton function $\varphi$ satisfying a modification of the Einstein equations. A surprising feature of these equations is that they possess a symmetry not found in the ordinary Einstein equations, namely T-duality. The presence of this symmetry reflects the string-theoretic origins of these equations. T-duality relates spaces $X$, $\hat{X}$ which are torus bundles over a common base space and is characterised by an interchange between Chern classes of the torus bundle with topological data associated to the closed $3$-form $H$.\\

It is possible to give a geometric meaning to T-duality using the language of generalised geometry \cite{guacav}. From this point of view T-duality is seen as an isomorphism of Courant algebroids $E,\hat{E}$ associated to the spaces $X,\hat{X}$. Leaving aside the dilaton $\varphi$, the field content $(g,H)$ defines a generalised metric on the Courant algebroid $E$. It then becomes possible to understand the T-duality symmetry of the type II string theory equations as an isomorphism of generalised metrics.\\

The aim of this paper is to address the problem of T-duality in heterotic string theory. The situation differs from the ordinary case in two significant ways. Firstly, heterotic theories require the presence of a gauge bundle with connection. This leads to more complicated equations of motion and modifies the topological conditions for the existence of T-duals as we will see. Secondly, the 3-form $H$ is no longer closed and globally does not correspond to a gerbe. Instead it is subject to the Green-Schwarz anomaly cancellation condition \cite{grsc}
\begin{equation}\label{gresch}
dH = \alpha' \left( Tr( R \wedge R ) - Tr(F \wedge F) \right),
\end{equation}
where $F$ is the curvature of a gauge connection $\nabla$ and $R$ the curvature of an affine connection $\nabla^T$ (often taken to be the Levi-Civita connection). 

To tackle these new challenges, we propose a class of transitive Courant algebroids, which we call {\em heterotic Courant algebroids} (Definition \ref{defhetca}), as the appropriate structure to consider. These are constructed by a reduction procedure and are characterised up to isomorphism by (\ref{gresch}). The latter also clarifies the connection to transitive Courant algebroids because, as established by Bressler \cite{bre}, a solution to (\ref{gresch}) determines a transitive Courant algebroid. From this vantage point we are able to obtain several new insights into heterotic T-duality. In particular, in Section \ref{sectdsc} we determine the topological conditions for a heterotic T-dual to exist and show that for any T-dual, there is an isomorphism of transitive Courant algebroids (Proposition \ref{proptopcond}). Using generalised metrics on transitive Courant algebroids we prove in Section \ref{sechetd} that the heterotic Einstein equations are preserved by T-duality. A further  consequence of (\ref{gresch}) is that it allows for more flexbility in the possible changes in topology and thus gives rise to a host of new examples of T-dual pairs, some of which are documented in Section \ref{secex}. \\

In more detail, to each transitive Courant algebroid $\mathcal{H}$ on $X$ is an associated transitive Lie algebroid $\mathcal{A} = \mathcal{H}/T^*X$. We say that $\mathcal{H}$ is a {\em heterotic Courant algebroid} if $\mathcal{A}$ is the Atiyah algebroid of a principal $G$-bundle $\sigma \colon P \to X$. Throughout we take $G$ to be a compact semisimple Lie group with Lie algebra $\mathfrak{g}$. The Courant algebroid $\mathcal{H}$ has a pairing $\langle \, \, , \, \, \rangle$ which in turn determines a non-degenerate invariant pairing $c = \langle \, \, , \, \, \rangle \in S^2(\mathfrak{g}^*)$ on $\mathfrak{g}$. We show in Proposition \ref{reduct} that every heterotic Courant algebroid is obtained by reduction of an exact Courant algebroid $E$ on $P$. To describe the reduction procedure, recall that exact Courant algebroids are classified by degree $3$ cohomology. Let $H$ be a closed $G$-invariant $3$-form on $P$ representing a class $h = [H] \in H^3(P,\mathbb{R})$ and let $E$ be the corresponding Courant 
algebroid on $P$. To reduce $E$ we require an extended action in the sense of \cite{bcg}. This is an equivariant $\mathfrak{g}$-valued $1$-form $\xi$ satisfying the equation 
\begin{equation}\label{equextended}
d_G( H + \xi ) = c,
\end{equation}
where $d_G$ is the differential for the Cartan complex. In Proposition \ref{extact} we prove that up to equivalence all solutions to (\ref{equextended}) are of the form 
\begin{equation*}
\xi = -\langle A, \, \, \, \rangle, \; \; \; \; H = \sigma^*(H^0) - CS_3(A),
\end{equation*}
where $A$ is a connection on $P$ with curvature $F$, $CS_3(A)$ is the Chern-Simons $3$-form of $A$ and $H^0$ is a $3$-form on $X$ satisfying $dH^0 = \langle F \wedge F \rangle$. In this way we see that extended actions lead naturally to the anomaly cancellation condition (\ref{gresch}). Furthermore it follows that such an extended action exists if and only if the class $h \in H^3(P,\mathbb{R})$ is a {\em string class} in real cohomology, that is the restriction of $h$ to the fibres of $P$ coincides with the Cartan $3$-form $\omega_3 \in H^3(G , \mathbb{R})$ determined by the pairing $\langle \, \, , \, \, \rangle$. This construction can be reversed so that to every string class on $P$ we obtain by reduction a heterotic Courant algebroid on $X$ depending only on the string class.\\

To explain our formulation of heterotic T-duality we first review the Courant algebroid approach to ordinary T-duality \cite{guacav}. Let $T^n,\hat{T}^n$ be dual tori with Lie algebras $\mathfrak{t},\hat{\mathfrak{t}}$. Duality of $T^n,\hat{T}^n$ means there is a natural pairing $\langle \, \, , \, \, \rangle \colon \mathfrak{t} \otimes \hat{\mathfrak{t}} \to \mathbb{R}$. Let $\pi \colon X \to M$ be a principal $T^n$-bundle and $H$ a closed $T^n$-invariant $3$-form on $X$. The cohomology class $[H] \in H^3(M,\mathbb{R})$ determines an exact Courant algebroid $E$ on $X$ and the $T^n$-action lifts to $E$. Similarly let $\hat{\pi} \colon \hat{X} \to M$ be a principal $\hat{T}^n$-bundle, $\hat{H}$ a closed invariant $3$-form and $\hat{E}$ the associated exact Courant algebroid. T-duality of the pairs $(X,H),(\hat{X},\hat{H})$ is captured by the following identity on the fibre product $X \times_M \hat{X}$
\begin{equation}\label{equhhhat1}
\hat{H} - H = d \langle \theta \wedge \hat{\theta} \rangle,
\end{equation}
where $\theta,\hat{\theta}$ are connections on $X,\hat{X}$. This equation captures the exchange of Chern classes of $X,\hat{X}$ and cohomology of $H,\hat{H}$ along the fibres, central to T-duality. Using (\ref{equhhhat1}) one can construct an isomorphism of Courant algebroids $E/T^n \simeq \hat{E}/ \hat{T}^n$. \\

In the heterotic setting we have torus bundles $\pi_0 \colon X \to M$, $\hat{\pi}_0 \colon \hat{X} \to M$ equipped with principal $G$-bundles $\sigma \colon P \to X$, $\hat{\sigma} \colon P \to \hat{X}$. Now $H,\hat{H}$ are invariant $3$-forms on $P,\hat{P}$ such that together with extended actions $\xi,\hat{\xi}$ we have $d_G(H + \xi) = c = d_G(\hat{H} + \hat{\xi})$. The extended actions determine heterotic Courant algebroids $\mathcal{H},\hat{\mathcal{H}}$ on $X,\hat{X}$ which carry torus actions. For heterotic 
$T$-duality we require the existence of a principal $G$-bundle $\sigma_0 \colon P_0 \to M$ such that $P,\hat{P}$ are obtained from $P_0$ by pullback. Observe that the projections $\pi \colon P \to P_0$, $\hat{\pi} \colon \hat{P} \to P_0$ are principal torus bundles over a common base. Moreover, they are equipped with exact Courant algebroids $E, \hat E$ associated with the invariant closed 3-forms $H, \hat H$. Taking advantage of this fact, we show in Section \ref{sectdflux} that for heterotic T-duality Equation (\ref{equhhhat1}) is replaced by the following equation on the fibre product $P \times_{P_0} \hat{P}$
\begin{equation}\label{equhhhat2}
( \hat{H} + \hat{\xi} ) - (H + \xi) = d_G \langle \theta \wedge \hat{\theta} \rangle,
\end{equation}
where $\theta,\hat{\theta}$ are connections for the torus bundles $P \to P_0$, $\hat{P} \to P_0$. Notice that if $H + \xi$ is a solution to (\ref{equextended}) then so is $\hat{H} + \hat{\xi}$. In addition to the usual exchange of Chern classes and $H$-classes, Equation (\ref{equhhhat2}) captures some features unique to heterotic T-duality, namely the change in extended action from $\xi$ to $\hat{\xi}$. Since an extended action $\xi$ corresponds to a gauge connection $A$, Equation (\ref{equhhhat2}) describes the change in connection brought on by heterotic T-duality. Passing from equivariant forms to the usual de Rham complex (\ref{equhhhat2}) reduces to (\ref{equhhhat1}), hence we obtain an isomorphism of exact Courant algebroids $\phi \colon E/T^n \to \hat{E}/\hat{T}^n$. Using (\ref{equhhhat2}) we show that $\phi$ exchanges extended actions $\xi,\hat{\xi}$ (Proposition \ref{propisohet}), hence on reduction we obtain an isomorphism $\mathcal{H}/T^n \simeq \hat{\mathcal{H}}/\hat{T}^n$ of transitive Courant 
algebroids. We have thus obtained a T-duality isomorphism in the heterotic setting by reduction from the exact case.\\

In Section \ref{secgmr} we consider generalised metrics on heterotic Courant algebroids. Proposition \ref{genmetclass} gives an equivalence between admissible generalised metrics and triples $(g,A,H)$, consisting of a Riemannian metric $g$, a $G$-connection $A$ with curvature $F$ and a $3$-form $H$ such that $dH = \langle F \wedge F \rangle$. This is the bosonic field content of the low energy limit of heterotic string theory. More accurately, one should in addition include a metric connection $\nabla^{T}$ with curvature $R$ and consider the gauge group $G \times O(m)$, for $m$ the dimension of $X$. Then the $3$-form $H$ will satisfy the anomaly cancellation condition (\ref{gresch}), as required for heterotic string theory. Using an isomorphism of heterotic Courant algebroids $\mathcal{H}/T^n \simeq \hat{\mathcal{H}}/\hat{T}^n$ we may send a generalised metric $(g,A,H)$ on $X$ to a corresponding generalised metric $(\hat{g},\hat{A},\hat{H})$ on $\hat{X}$. We determine in detail the relation between 
these generalised metrics (Proposition \ref{globalhb}).\\

In Section \ref{sechetd} we show that $(\hat{g},\hat{A},\hat{H})$ is a solution of the heterotic equations of motion if and only if $(g,A,H)$ is a solution (Proposition \ref{prophetequpres}), establishing that heterotic T-duality is a symmetry of these equations. To do this we first prove that the heterotic equations of motion may be obtained by reduction from the type II equations (Proposition \ref{einstlift}) and then verify that T-duality preserves the latter equations (Proposition \ref{eetd}). Given the complexity of the transformation laws for generalised metrics, it is a remarkable fact that the equations of motion can be matched up. Moreover the manner in which the equations are matched up involves a series of unexpected cancellations, brought to light through the use of Courant algebroids and generalised metrics.\\

We give a brief description of each section of the paper. Section \ref{secrca} reviews Courant algebroids and their reduction by group actions. We distinguish two notions of reduction, simple reduction and reduction by extended action and give conditions for these reduction procedures to commute. Section \ref{sechetca} introduces heterotic Courant algebroids, gives their classification and their construction from string classes and extended actions. Section \ref{sechettd} develops the theory of heterotic T-duality by reduction, culminating in an isomorphism of heterotic Courant algebroids. In Section \ref{secgmr} we consider generalised metrics on heterotic Courant algebroids and their transformation under T-duality. Finally in Section \ref{sechetd} we consider the heterotic equations of motion and show they are preserved by T-duality.


\section{Reduction of Courant algebroids}\label{secrca}


\subsection{Courant algebroids and symmetries}

Let $(E , [ \, \, , \, \, ], \rho , \langle \; , \; \rangle)$ consist of a vector bundle $E \to M$ over a smooth manifold $M$, a bilinear operator $[ \, \, , \, \, ] \colon \Gamma(E) \otimes \Gamma(E) \to \Gamma(E)$ on the space of sections of $E$, a bundle map $\rho \colon E \to TM$ and $\langle \, \, , \, \, \rangle$ a non-degenerate bilinear form on $E$. Since $\langle \, \, , \, \, \rangle$ is non-degenerate we may use it together with $\rho$ to define a natural map from $T^*M$ to $E$. Explicitly a $1$-form $\xi \in \Gamma(T^*M)$ is mapped to a section $\xi' = \frac{1}{2}\rho^*(\xi)$ of $E$ such that for all $e \in \Gamma(E)$ we have $\langle \xi' , e \rangle = \frac{1}{2} \xi(\rho(e))$.

\begin{definition}[\cite{lwx}]\label{defca}
The data $(E , [ \, \, , \, \, ], \rho , \langle \, \, , \, \, \rangle)$ is a {\em Courant algebroid} if the following conditions hold for all $a,b,c \in \Gamma(E)$:
\begin{itemize}
\item{$[ a , [b,c] ] = [ [a,b] , c ]+ [ b , [a,c] ]$}
\item{$[a,b] + [b,a] = 2 \, d\langle a , b \rangle$}
\item{$\rho(a)\langle b , c \rangle = \langle [ a,b] , c \rangle + \langle b , [ a,c ] \rangle$.}
\end{itemize}
We call $[ \, \, , \, \, ]$ the {\em Dorfman bracket}, $\rho$ the {\em anchor} and $\langle \, \, , \, \, \rangle$ the {\em pairing} of $E$.
\end{definition}

Two properties of Courant algebroids which follow from the above axioms but are often included as separate axioms are:
\begin{itemize}
\item{$ [a , fb ] = f[a,b] + \rho(a)(f)b $}
\item{$\rho [a,b] = [\rho(a) , \rho(b)]$}
\end{itemize}
for every $f \in \mathcal{C}^\infty(M)$ and $a,b \in \Gamma(E)$.\\

We say that a Courant algebroid $E$ is {\em transitive} if the anchor $\rho$ is surjective and that $E$ is {\em exact} if $E$ is transitive and the kernel of the anchor coincides with the image of the map $\frac{1}{2}\rho^* \colon T^*M \to E$. When $E$ is transitive the map $\frac{1}{2}\rho^*$ is injective, so we can identify $T^*M$ as a subbundle of $E$. Since we will be concerned exclusively with transitive Courant algebroids we will identify $1$-forms with their image in $E$ without further mention.\\

Let $E$ be an exact Courant algebroid on $M$, so we have an exact sequence $0 \to T^*M \to E \buildrel \rho \over \to TM \to 0$. An {\em isotropic splitting} for $E$, or {\em splitting} for short, is a section $s \colon TM \to E$ of $\rho$ such that the image $s(TM) \subseteq E$ is isotropic with respect to the pairing on $E$. Given a splitting $s \colon TM \to E$ we obtain an identification $E = TM \oplus T^*M$ with anchor given by projection to $TM$ and pairing $\langle X + \xi , Y + \eta \rangle = \frac{1}{2}( i_X \eta + i_Y \xi )$, for tangent vectors $X,Y$ and $1$-forms $\xi,\eta$. There exists a closed $3$-form $H$ on $M$ such that the bracket on $E$ is as follows \cite{sw}:
\begin{equation}\label{htwist}
[ X + \xi , Y + \eta ]_H = [X,Y] + \mathcal{L}_X \eta - i_Y d\xi + i_Y i_X H.
\end{equation}
Choosing a different splitting for $E$ has the effect of changing $H$ by an exact term. The cohomology class $[H] \in H^3(M,\mathbb{R})$, which is independent of the choice of splitting, is called the {\em \v{S}evera class} of the exact Courant algebroid \cite{sev}. Conversely given a closed $3$-form $H$, (\ref{htwist}) defines an exact Courant algebroid structure on $E = TM \oplus T^*M$. We call the bracket $[ \, \, , \, \, ]_H$ the {\em $H$-twisted Dorfman bracket} on $E$. This construction determines a bijection between isomorphism classes of exact Courant algebroids and $H^3(M , \mathbb{R})$.


\begin{definition}
Let $(E , [ \, \, , \, \, ], \rho , \langle \, \, , \, \, \rangle)$, $(E' , [ \, \, , \, \, ]', \rho' , \langle \, \, , \, \, \rangle')$ be Courant algebroids on $M$. An {\em isomorphism} $\phi \colon (E , [ \, \, , \, \, ], \rho , \langle \, \, , \, \, \rangle) \to (E' , [ \, \, , \, \, ]', \rho' , \langle \, \, , \, \, \rangle')$, or $\phi \colon E \to E'$ for short consists of a diffeomorphism $f \colon M \to M$ and a bundle isomorphism $\phi \colon E \to E'$ covering $f$ such that the induced map of sections $\phi_* \colon \Gamma(E) \to \Gamma(E')$ given by $\phi_*(a) = \phi \circ a \circ f^{-1} $ interchanges the Courant algebroid structures:
\begin{itemize}
\item{$ [ \phi_* a , \phi_* b ]' = \phi_* [ a , b ]$}
\item{$ \rho' ( \phi_* a) = f_* \rho(a) $}
\item{$\langle \phi_* a , \phi_* b \rangle' =  (f^{-1})^* \langle a , b \rangle$.}
\end{itemize}
In this case we say that the isomorphism $\phi$ {\em covers} $f$ or that $\phi$ is a {\em lift} of $f$.
\end{definition}
When $E = E'$ we have Courant algebroid automorphisms. The corresponding infinitesimal notion is given by derivations:
\begin{definition}
Let $(E , [ \, \, , \, \, ], \rho , \langle \, \, , \, \, \rangle)$ be a Courant algebroid on $M$. A {\em derivation} $D \colon (E , [ \, \, , \, \, ], \rho , \langle \, \, , \, \, \rangle) \to (E , [ \, \, , \, \, ], \rho , \langle \, \, , \, \, \rangle)$, or $D \colon E \to E$ for short consists of an endomorphism $D \colon\Gamma(E) \to \Gamma(E)$ and a vector field $X$ such that:
\begin{itemize}
\item{$D [a,b] = [ Da , b ] + [ a , Db ]$}
\item{$ X \langle a , b \rangle = \langle Da , b \rangle + \langle a , Db \rangle$.}
\end{itemize}
From these properties it follows that the vector field is uniquely determined from $D$, so we denote $X$ by $X = \rho(D)$. The space ${\rm Der}(E)$ of derivations of $E$ forms a Lie algebra. 
\end{definition}

From the definition we obtain the identities $D(fa) = \rho(D)(f)a + fD(a)$ and $\rho(Da) = [\rho(D) , \rho(a)]$, where $D$ is a derivation, $a \in \Gamma(E)$, $f \in \mathcal{C}^\infty(M)$. In particular the natural map $\rho \colon {\rm Der}(E) \to {\rm Vect}(M)$ sending a derivation to the corresponding vector field on $M$ is a Lie algebra homomorphism. Observe that for any section $a \in \Gamma(E)$ the adjoint action ${\rm ad}_a(b) = [a , b ]$ is a derivation. We say that a derivation is {\em inner} if it is given by the adjoint action of some section of $E$. Next we are interested in reducing a Courant algebroid by a group of symmetries and observe two such notions of reduction.


\subsection{Simple reduction}\label{secsr}

Let $G$ be a Lie group and suppose that $G$ acts on $M$ on the right by diffeomorphisms. Let $\mathfrak{g}$ denote the Lie algebra of $G$, defined using left invariant vector fields on $G$. By differentiation a right action of $G$ determines a Lie algebra homomorphism\footnote{For a right action taking left invariant vector fields on $\mathfrak{g}$ ensures that $\psi$ is a Lie algebra homomorphism. For a left action on $M$ we should define $\mathfrak{g}$ using right invariant vector fields on $G$ instead.} $\psi \colon \mathfrak{g} \to {\rm Vect}(M)$ to the Lie algebra of vector fields on $M$.
\begin{definition}
A {\em lifted action} of $G$ on $E$ is a right action of $G$ on $E$ by automorphisms covering the right action of $G$ on $M$ by diffeomorphisms. A {\em lifted infinitesimal action} is a Lie algebra homomorphism $\tilde{\psi} \colon \mathfrak{g} \to {\rm Der}(E)$ covering a homomorphism $\psi \colon \mathfrak{g} \to {\rm Vect}(M)$.
\end{definition}

Let $G$ act on $M$ on the right and let $\psi \colon \mathfrak{g} \to {\rm Vect}(M)$ be the corresponding infinitesimal action. Suppose we have a lifted action of $G$ to automorphisms of a Courant algebroid $E$. By differentiating we obtain a lifted infinitesimal action $\tilde{\psi} \colon\mathfrak{g} \to {\rm Der}(E)$ covering $\psi$. Conversely suppose we have a lifted infinitesimal action $\tilde{\psi}$ covering the infinitesimal action $\psi \colon \mathfrak{g} \to {\rm Vec}(M)$. If $\tilde{\psi}$ is obtained by differentiation of a lifted action of $G$ then we say that $\tilde{\psi}$ {\em integrates} to a lifted action of $G$ on $E$.\\

Suppose we are given a lifted action of $G$ on $M$ to a Courant algebroid $E$. We would like to define the quotient of $E$ by the lifted action of $G$. To do this we will assume that $M$ is a principal $G$-bundle so that the quotient space $M/G$ is a well behaved manifold. Since $G$ acts on $E$ by bundle isomorphisms covering the action on $M$ it is immediate that the quotient space $E/G$ is a vector bundle on $M/G$. Moreover there is a canonical identification $\Gamma( E/G ) = \Gamma(E)^G$ between sections of the quotient bundle $E/G$ and $G$-invariant sections of $E$, the space of which we denote by $\Gamma(E)^G$. We give $E/G$ the structure of a Courant algebroid. First observe that if $a,b$ are $G$-invariant sections of $E$ then $[ a, b ]$ and $\langle a , b \rangle$ are $G$-invariant, so $E/G$ naturally inherits a bracket $[ \, \, , \, \, ]$ and pairing $\langle \, \, , \, \, \rangle$. We define an anchor $\rho^G$ on $E/G$ by the composition $\Gamma(E)^G \buildrel \rho \over \to \Gamma(TM)^G \to \Gamma(
T(M/G))$, where the last map is induced by the natural projection $TM/G \to T(M/G)$.

\begin{proposition}
The data $(E/G , [ \, \, , \, \, ] , \rho^G , \langle \, \, , \, \, \rangle)$ obtained from a lifted action defines a Courant algebroid on $M/G$.
\end{proposition}

We call $E/G$ the {\em simple reduction} of $E$ by the lifted action of $G$.


\subsection{Extended actions}\label{secexa}

Let us recall the notion of extended actions and their reductions as introduced in \cite{bcg}. It will be necessary to adapt their framework to a class of Courant algebroids which are transitive but not exact. We focus on the case which is most relevant to us, the case of trivially extended actions.\\

Let $G$ act on a manifold $M$ and suppose we have a lifted action $G \to {\rm Aut}(E)$ of $G$ to an action on a Courant algebroid $E$ over $M$. By differentiation we have an infinitesimal lifted action $\mu \colon \mathfrak{g} \to {\rm Der}(E)$. Let ${\rm ad} \colon \Gamma(E) \to {\rm Der}(E)$ denote the adjoint map taking sections of $E$ to inner derivations. If the image of $\mu$ lies in the subspace of inner derivations then we can consider lifting $\mu$ to a map from $\mathfrak{g}$ to $\Gamma(E)$. However, we would like to lift $\mu$ in a way that respects the algebra structures on $\mathfrak{g}$ and $\Gamma(E)$. In general there are obstructions to doing this, so one is lead to consider extensions of $\mathfrak{g}$ \cite{bcg}. In the situation which concerns us however the obstruction vanishes, thus we will concern ourselves only with the special case of trivially extended actions.
\begin{definition}[\cite{bcg}]
Let $G$ be a connected Lie group acting on $M$ with infinitesimal action $\psi \colon \mathfrak{g} \to \Gamma(TM)$. Let $E$ be a transitive Courant algebroid on $M$, not necessarily exact. A {\em trivially extended action} is a map $\alpha \colon \mathfrak{g} \to \Gamma(E)$ such that $\alpha$ is a homomorphism of algebras, $\rho \circ \alpha = \psi$ and the induced adjoint action of $\mathfrak{g}$ on $E$ integrates to an action of $G$ on $E$.
\end{definition}

In the case that $G$ is compact and $E$ is an exact Courant algebroid there is a characterisation of trivially extended actions in terms of the Cartan complex for equivariant cohomology. Suppose that we have an extended action of $G$ on $E$. Since $G$ is compact find a $G$-invariant splitting of $E$. Thus we obtain a closed $3$-form $H$ such that $E = TM \oplus T^*M$ with $H$-twisted Dorfman bracket. With respect to this splitting the infinitesimal action of $G$ on $E$ necessarily has the form $e( Y + \eta ) = \mathcal{L}_{\psi(e)}(Y + \eta)$ for all $e \in \mathfrak{g}$. In particular, it follows that $H$ is $G$-invariant. Since this is an extended action we have that $e(Y+\eta) = [ \alpha(e) , Y+\eta] = [\psi(e)+\xi(e) , Y + \eta] = \mathcal{L}_{\psi(e)}(Y + \eta) + i_Y( i_{\psi(e)} H - d\xi(e) )$. We therefore must have
\begin{equation}\label{exacond}
d\xi(e) = i_{\psi(e)} H
\end{equation}
for all $e \in \mathfrak{g}$. The condition that $\alpha$ is an algebra homomorphism is that $\psi([a,b]) + \xi([a,b]) = [ \psi(a)+\xi(a),\psi(b)+\xi(b)] = \mathcal{L}_{\psi(a)}( \psi(b) + \xi(b))$, or simply 
\begin{equation}\label{exacond2}
\mathcal{L}_{\psi(a)}(\xi(b)) = \xi([a,b]).
\end{equation}
This is exactly the condition that $\xi$ thought of as a $\mathfrak{g}^*$-valued $1$-form is equivariant. The conditions (\ref{exacond}),(\ref{exacond2}) can be expressed using the Cartan complex
\begin{equation*}
\Omega^k_G(M) = \bigoplus_{2p+q = k} \Omega^q(M , S^p \mathfrak{g}^* )^G
\end{equation*}
with differential $d_G = d - \iota$ given by
\begin{equation*}
d_G( \omega ) = d \omega  - e^a i_{\psi(e_a)} \omega
\end{equation*}
where $e_a$ denotes a basis for $\mathfrak{g}$ and $e^a$ the corresponding dual basis. Following \cite{bcg} we set $\Phi = H + \xi \in \Omega^3_G(M)$. Then $\xi$ determines an extended action if and only if
\begin{equation*}
d_G( \Phi ) = -\langle \alpha(e_i) , \alpha(e_j) \rangle e^{ij}.
\end{equation*}
We write this as $d_G(\Phi) = c$, where $c = -\langle \alpha(e_i) , \alpha(e_j) \rangle e^{ij} \in \Omega^0(M,S^2 \mathfrak{g}^*)^G$ thought of as a bilinear form on $\mathfrak{g}$ is $-1$ times the pullback under $\alpha \colon \mathfrak{g} \to \Gamma(E)$ of the pairing $\langle \, \, , \, \, \rangle$ on $E$. Note that since an extended action of this form simply acts on $E = TM \oplus T^*M$ by Lie derivative, there is no obstruction to integrating this to a $G$-action on $E$ provided that $G$ acts on $M$.\\

Two extended actions $\xi,\xi'$ are considered equivalent if there is an equivariant function $f \colon M \to \mathfrak{g}^*$ such that for each $a \in \mathfrak{g}$ we have $\xi'(a) = \xi(a) + df(a)$. In terms of the corresponding equivariant $3$-forms $\Phi = H + \xi$, $\Phi' = H + \xi'$ we have $\Phi'  = \Phi + d_G(f)$, so that $\Phi,\Phi'$ differ by a $d_G$-exact term. In addition we observe that a change in choice of invariant splitting for $E$ corresponds to a change $\Phi \mapsto \Phi + d_G(\beta)$, where $\beta \in \Omega^2(M)^G$ is the invariant $2$-form relating the splittings. In this way we obtain a correspondence between solutions of the equation $d_G(\Phi) = c$ modulo $d_G$-exact terms and extended actions up to equivalence.
\begin{proposition}[\cite{bcg}, Theorem 2.13]\label{extact1}
Let $G$ be a compact group. Then trivially extended $G$-actions $\alpha \colon \mathfrak{g} \to \Gamma(E)$ on a fixed exact Courant algebroid $E$ with prescribed quadratic form $c(a) = -\langle \alpha(a) , \alpha(a) \rangle$ are up to equivalence in bijection with solutions to $d_G \Phi = c$ modulo $d_G$-exact forms, where $\Phi = H + \xi$ is an equivariant $3$-form and $H$ represents the \v{S}evera class of $E$.
\end{proposition}

Our objective now is to extend Proposition \ref{extact1} to a more general class of Courant algebroids. We will consider Courant algebroids of the following form: let $L$ be a Lie algebroid on $M$, which for simplicity we will assume to be transitive. Let $\Omega^k(L) = \Gamma( M , \wedge^k L^*)$ be the space of degree $k$ Lie algebroid forms and $d_L : \Omega^k(L) \to \Omega^{k+1}(L)$ the Lie algebroid differential. Let $H \in \Omega^3(L)$ be a $d_L$-closed $3$-form. Set $E = L \oplus L^*$ with the following bracket
\begin{equation*}
[ X + \xi , Y + \eta] = [X,Y]_L + \mathcal{L}_X \eta - i_Y d_L \xi + i_Y i_X H
\end{equation*}
where $X,Y \in \Gamma(L)$, $\xi,\eta \in \Gamma(L^*)$ and $\mathcal{L}_X \eta = i_X d_L \eta + d_L i_X \eta$. It is straightforward to see that this bracket makes $E$ into a Courant algebroid \cite{lb}. The pairing on $E$ is just the dual pairing of $L$ and $L^*$:
\begin{equation*}
\langle X + \xi , Y + \eta \rangle = \frac{1}{2}( i_X \eta + i_Y \xi )
\end{equation*}
and the anchor is the projection to $L$ followed by the anchor of $L$. Note also that the isomorphism class of $E$ depends only on the cohomology class $[H] \in H^3(L)$ in Lie algebroid cohomology. We call $[H]$ the \v{S}evera class of $E$ and $E$ the Courant algebroid associated to $[H]$. We note that $E$ naturally fits into an short exact sequence $0 \to L ^* \to E \to L \to 0$. An isotropic splitting $s \colon L \to E$ of this sequence will be called a {\em splitting} of $E$. It is immediate that splittings of $E$ exist. A choice of splitting $s \colon L \to E$ determines a $3$-form $H$ representing the \v{S}evera class of $E$, namely $H(a,b,c) = 2\langle [s(a) , s(b) ] , s(c) \rangle $. It is clear also that different choices of splitting give rise to different representatives for the \v{S}evera class.\\

As a prerequisite to obtaining an extended action on $E$ we first need to assume the existence of an action of $G$ on $L$ lifting an action of $G$ on $M$. Therefore we will assume that there is a Lie algebra homomorphism $\psi \colon \mathfrak{g} \to \Gamma(L)$ with the property that the adjoint action of $\psi$ integrates to an action of $G$ on $L$ covering an action of $G$ on $M$. If $\rho \colon L \to TM$ is the Lie algebroid anchor then we obtain a commutative diagram of Lie algebras
\begin{equation*}\xymatrix{
\mathfrak{g} \ar[r]^{\psi} \ar[dr] & \Gamma(L) \ar[d]^{\rho} \\
& \Gamma(TM)
}
\end{equation*}
For each $a \in \mathfrak{g}$ we obtain a contraction operator $i_{\psi(a)} \colon \Omega^k(L) \to \Omega^{k-1}(L)$. There is a natural infinitesimal action of $\mathfrak{g}$ on $\Omega^*(L)$ by Lie derivative $\mathcal{L}_{\psi(a)} = i_{\psi(a)} d_L + d_L i_{\psi(a)}$ which gives $\Omega^*(L)$ the structure of a $\mathfrak{g}$-differential complex in the language of \cite{gin}. Following \cite{gin} we may now define the equivariant cohomology for $\Omega^*(L)$ from a Cartan model construction (see also \cite{bcrr}). Let $\Omega^q( L , S^p \mathfrak{g})^{\mathfrak{g}} \subset \Omega^q(L , S^p \mathfrak{g}^*)$ denote the subspace of equivariant Lie algebroid forms valued in $S^* \mathfrak{g}^*$. The Cartan model $( \Omega^k_G(L) , d_G )$ is given by the complex
\begin{equation*}
\Omega^k_G(L) = \bigoplus_{2p+q = k} \Omega^q( L , S^p \mathfrak{g}^* )^{\mathfrak{g}}
\end{equation*}
with equivariant differential $d_G$ given by
\begin{equation*}
d_G(\omega) = d_L \omega - e^a i_{\psi(e_a)} \omega.
\end{equation*}

Now consider extended actions $\alpha \colon \mathfrak{g} \to \Gamma(E)$ of the form $\alpha = (\psi , \xi)$ for some map $\xi \colon \mathfrak{g} \to \Gamma(L^*)$. The conditions for $\alpha$ to be an extended action are a direct generalisation of (\ref{exacond}), (\ref{exacond2}) to the Lie algebroid setting. Therefore if we set $\Phi = H + \xi \in \Omega^3_G(L)$ and $c(a) = - \langle \alpha(a) , \alpha(a) \rangle$ then the condition for $\alpha$ to be an extended action is $d_G \Phi = c$, the same condition as in the exact case. As with the exact case we may say that extended actions corresponding to $\xi,\xi'$ are equivalent if there exists $f \in \Omega^0(L,\mathfrak{g}^*)^{\mathfrak{g}}$ such that $\xi' = \xi + d_L(f)$. Then modulo this equivalence relation we have a classification of extended actions parallel to Proposition \ref{extact1}. For convenience let us summarise this:
\begin{proposition}\label{extact2}
Let $G$ be a compact group, $L$ a transitive Lie algebroid and $E$ the Courant algebroid associated to a class in $H^3(L)$. Let $G$ act on $M$ and suppose that there is an infinitesimal action $\psi \colon \mathfrak{g} \to \Gamma(L)$ which integrates to an action of $G$ on $L$ covering the action on $M$. Fix a quadratic form $c \in \Omega^0(L,S^2\mathfrak{g}^*)^{\mathfrak{g}}$. There is a bijection between trivially extended $G$-actions $\alpha \colon \mathfrak{g} \to \Gamma(E)$ satisfying:
\begin{itemize}
\item{there exists an isotropic splitting $E = L \oplus L^*$ such that $\alpha$ has the form $\alpha = (\psi , \xi)$ for some $\xi : \mathfrak{g} \to \Gamma(L^*)$}
\item{$c(a) = -\langle \alpha(a) , \alpha(a) \rangle$ for all $a \in \mathfrak{g}$}
\end{itemize}
and solutions to $d_G \Phi = c$ modulo $d_G$-exact forms, where $\Phi = H + \xi$ is an equivariant $3$-form and $H$ represents the \v{S}evera class of $E$.
\end{proposition}

As with the exact case, once we assume that the action $\psi \colon \mathfrak{g} \to \Gamma(L)$ can be integrated to an action of $G$ on $L$ there is no obstruction to integrating $\alpha \colon \mathfrak{g} \to \Gamma(E)$ to an action of $G$ on $E$.


\subsection{Reduction by extended actions}\label{secrbea}

Let $G$ act on $M$ and suppose that $E$ is a Courant algebroid with a trivially extended action $\alpha \colon \mathfrak{g} \to \Gamma(E)$. In \cite{bcg} a theory of reduction for exact Courant algebroids with extended actions is developed. We will extend this to the more general class of Courant algebroids considered in Section \ref{secexa}, but for simplicity we will do so only for a restricted class of extended action.
\begin{definition}
An extended action $\alpha : \mathfrak{g} \to \Gamma(E)$ is called {\em non-degenerate} if for each point $x \in M$ the bilinear form $c_x \in S^2 \mathfrak{g}^*$ given by $c_x(a,b) = -\langle \alpha(a)(x) , \alpha(b)(x) \rangle$ is non-degenerate as a bilinear form on $\mathfrak{g}$.
\end{definition}
In the case of a non-degenerate extended action the map $\alpha \colon \mathfrak{g} \to \Gamma(E)$ is injective on each fibre of $E$ so that the image $K$ is a sub-bundle of $E$ which is non-degenerate with respect to the pairing on $E$. Let $K^\perp$ be the annihilator of $K$ with respect to the pairing on $E$.\\

Before considering the more general case let us suppose first that $E$ is an exact Courant algebroid. We will assume also that the action of $G$ on $M$ is free and proper so that $M$ is a principal $G$-bundle over $M/G$. Recall from \cite{bcg} that the big distribution $\Delta_b$ is given by $\Delta_b = \rho( K + K^\perp) \subset TM$. In our case we know that $K$ is non-degenerate with respect to the pairing on $E$ so that $K$ and $K^\perp$ are complementary and thus $\Delta_b = TM$. By \cite[Theorem 3.3]{bcg} we obtain a Courant algebroid $E_{{\rm red}}$ on $M/G$, the reduction of $E$. As a vector bundle $E_{{\rm red}}$ is given by
\begin{equation*}
E_{{\rm red}} = \frac{K^\perp}{K \cap K^{\perp}}  / G.
\end{equation*}
In our case we know that $K \cap K^\perp = 0$ so that $E_{{\rm red}} = K^\perp / G$. Sections of $E_{{\rm red}}$ can be identified with $G$-invariant sections of $K$ and in this way $E_{{\rm red}}$ inherits the structure of a Courant algebroid on $M$.\\

Now we consider the more general case where $E$ is the Courant algebroid associated to a transitive Lie algebroid $L$ and a class in $H^3(L)$. Assume that we have a non-degenerate extended action $\alpha \colon \mathfrak{g} \to \Gamma(E)$ of the form described in Proposition \ref{extact2}, associated to an infinitesimal action $\psi \colon \mathfrak{g} \to \Gamma(L)$. Let us again assume that $M \to M/G$ is a principal $G$-bundle. Then it is natural to define the reduction $E_{{\rm red}}$ of $E$ to be the bundle $E_{{\rm red}} = K^{\perp}/G$.
\begin{proposition}\label{redisca}
The bundle $E_{{\rm red}}$ naturally has the structure of a Courant algebroid on $M/G$. The Dorfman bracket and pairing on $E_{{\rm red}}$ are inherited from $E$ by identifying sections of $E_{{\rm red}}$ with $G$-invariant sections of $K^\perp \subseteq E$. The anchor on $E_{{\rm red}}$ is the composition
\begin{equation*}
K^\perp/G \to E/G \buildrel \rho \over \longrightarrow TM/G \to T(M/G).
\end{equation*}
\end{proposition}
\begin{proof}
We first need to show that the space $\Gamma(K^\perp)^G$ of $G$-invariant sections of $K^\perp$ is closed under the Dorfman bracket. For this let $v,w \in \Gamma(K^\perp)^G$. It is clear that $[v,w]$ is $G$-invariant so we need only show that $[v,w]$ is orthogonal to $\alpha(a)$ for any $a \in \mathfrak{g}$. We find
\begin{equation*}
\begin{aligned}
\langle \alpha(a) , [v,w] \rangle &= \rho(v) \langle a , w \rangle - \langle [v , \alpha(a) ] , w \rangle \\
& = \langle [\alpha(a) , v] , w \rangle - 2\langle d\langle v , \alpha(a) \rangle , w \rangle \\
& = 0
\end{aligned}
\end{equation*}
since $v,w$ are orthogonal to $\alpha(a)$ and $[\alpha(a),v] = 0$ since $v$ is $G$-invariant. This shows that $K^\perp / G$ naturally inherits a bracket from $E$. The bundle $E_{{\rm red}}$ inherits a pairing and anchor as described above. One verifies that these satisfy the axioms for a Courant algebroid on $M/G$ given in Definition \ref{defca}.
\end{proof}

We now show that equivalent extended actions produce isomorphic reductions. Although this is true more generally we restrict to the special case where the action of $G$ is free and proper. Consider changing an extended action $\Phi$ to an equivalent extended action $\Phi' = \Phi + d_G( f)$, where $f \in \Omega^0(L,\mathfrak{g}^* )^{ \mathfrak{g}}$ is an equivariant $\mathfrak{g}^*$-valued function. Let $A \in \Omega^1(M,\mathfrak{g})$ be a connection for the principal $G$-bundle $M \to M/G$. Then by pullback under the anchor $\rho \colon L \to TM$ we obtain a $\mathfrak{g}$-valued Lie algebroid $1$-form $B = \rho^*(A) \in \Omega^1(L,\mathfrak{g})$ with the property that $i_{\psi(e)} B = e$. Let $e_1, \dots , e_m$ be a basis for $\mathfrak{g}$ and $e^1, \dots , e^m$ the dual basis. If $f = f_i e^i$ then $f = -\iota( -f_iB^i)$, where $B^i = e^i(B)$. Thus $f_ie^i = d_G( -f_iB^i) + d_L(f_iB^i)$. Applying $d_G$ we obtain $d_G(f) = d_G( d_L(f_iB^i))$. Thus $\Phi' = \Phi + d_G(\beta)$ where $\beta = d_L(f_iB^i) \in 
\Omega^2(L)^{\mathfrak{g}}$, which is a closed, equivariant Lie algebroid $2$-form. We know that such $2$-forms act as automorphisms of $E$ changing the choice of splitting. It follows that the reduced Courant algebroids obtained from the extended actions $\Phi$ and $\Phi'$ are isomorphic.


\subsection{Commuting reductions}\label{seccomred}

We have introduced two types of reduction, simple reductions in Section \ref{secsr} and reduction by extended action in Section \ref{secrbea}. We now show that for commuting actions these reductions commute. Let $\sigma \colon P \to X$ be a principal $G$-bundle and suppose that the space $X$ is itself the total space of a principal $H$-bundle $\pi_0 \colon X \to M$. We assume here that $G,H$ are compact connected Lie groups with Lie algebras $\mathfrak{g},\mathfrak{h}$. When we consider T-duality we will take $H$ to be a torus, but it is not necessary to assume this yet.\\

Given a Courant algebroid on $P$ we are interested in reducing to a Courant algebroid on $X$ and then reducing again to $M$. In some situations it will be possible to reverse the order of the reductions. For this we assume that the $H$-action on $X$ lifts to a $H$-action on $P$ commuting with the $G$-action. Thus $P$ can be considered as a principal $G \times H$-bundle over $M$. We set $P_0 = P/H$, so that $P_0$ is a principal $G$-bundle over $M$. Define $\pi,\sigma_0$ to be the projections $\pi \colon P \to P_0$ and $\sigma_0 \colon P_0 \to M$. Then we have a commutative diagram
\begin{equation}\label{comdiag}
\xymatrix{
P \ar[r]^{\pi} \ar[d]^{\sigma} & P_0 \ar[d]^{\sigma_0} \\
X \ar[r]^{\pi_0} & M
}
\end{equation}
Let $E$ be a Courant algebroid on $P$. We assume that $E$ is the Courant algebroid associated to a transitive Lie algebroid $L$ and a cohomology class in $H^3(L)$. We assume that there is an infinitesimal action $\psi_{\mathfrak{g}} + \psi_{\mathfrak{h}} \colon \mathfrak{g} \oplus \mathfrak{h} \to \Gamma(L)$ which integrates to an action of $G \times H$ on $L$ covering the action on $P$. By averaging we can find a $G \times H$-invariant representative $h \in \Omega^3(L)$ for the \v{S}evera class of $E$ and thus we obtain an action of $G \times H$ on $E$. Note that the action of $G \times H$ on $L$ determines the action of $G \times H$ on $E$ up to isomorphism.

We suppose in addition that the $G$-action on $E$ comes from a non-degenerate extended action $\alpha_{\mathfrak{g}} \colon\mathfrak{g} \to \Gamma(E)$, where $\alpha_{\mathfrak{g}} = (\psi_{\mathfrak{g}} , \xi_{\mathfrak{g}})$ for some $\xi_{\mathfrak{g}} \in \Omega^1(L , \mathfrak{g}^*)^{\mathfrak{g}}$. Finally, in order for it to be possible to reduce $E$ twice we need to assume that the extended action $\alpha_{\mathfrak{g}}$ is $H$-invariant in the sense that for every $e \in \mathfrak{g}$ the section $\alpha_{ \mathfrak{g} }(e)$ is $H$-invariant.

Let $K \subseteq E$ be the image of $\alpha_{\mathfrak{g}}$ and $E_{{\rm red}} = K^\perp/G$ the reduction of $E$ by the extended action. Then $E_{{\rm red}}$ is a Courant algebroid on $X$. Let $E/H$ denote the simple reduction of $E$ by the action of $H$, which is a Courant algebroid on $P_0$.
\begin{proposition}\label{commutingred}
The action of $H$ on $E$ determines an action of $H$ on $E_{{\rm red}}$. Let $E_{{\rm red}}/H$ be the simple reduction of $E_{{\rm red}}$ by this action. The extended action $\alpha_{\mathfrak{g}}$ on $E$ determines an extended action on $E/H$. Let $(E/H)_{{\rm red}}$ denote the reduction by this extended action. There is a canonical isomorphism of Courant algebroids on $M$:
\begin{equation*}
E_{{\rm red}}/H \simeq (E/H)_{{\rm red}}.
\end{equation*}
\end{proposition}
\begin{proof}
Since the extended action $\alpha_{\mathfrak{g}}$ maps to $H$-invariant sections of $E$ we have that $K$ and $K^\perp$ are preserved by $H$. Then since the actions of $G$ and $H$ commute we get a natural action of $H$ on $E_{{\rm red}} = K^\perp/G$. We can identify sections of $E_{{\rm red}}/H$ with $G \times H$-invariant sections of $K^\perp$.\\

Let $M$ be the vector bundle on $P_0$ obtained by the quotient of $L$ by the action of $H$, so $M = L/H$. It is clear that $M$ is a Lie algebroid on $P_0$ and transitive since we assume $L$ is transitive. Recall that $E$ is given by $E = L \oplus L^*$ with the $h$-twisted Dorfman bracket. Since $h$ is invariant under the action of $H$ it defines a class $h' \in H^3(M)$. Clearly the simple reduction $E/H$ is the Courant algebroid associated to the Lie algebroid $M$ with \v{S}evera class $h' \in H^3(M)$. Now since the extended action $\alpha_{\mathfrak{g}} \colon \mathfrak{g} \to \Gamma(E)$ maps to $H$-invariant sections it automatically defines an extended action $\alpha' \colon \mathfrak{g} \to \Gamma(E/H)$. The image of $\alpha'$ is given by the quotient $K/H$ and $(K/H)^\perp = K^\perp/H$. Thus sections of $(E/H)_{{\rm red}}$ are naturally identified with $G \times H$-invariant sections of $K^\perp$. These identifications give the isomorphism $E_{{\rm red}}/H \simeq (E/H)_{{\rm red}}$.
\end{proof}


\section{Heterotic Courant algebroids}\label{sechetca}


\subsection{Quadratic Lie algebroids}

Let $V \to M$ be a vector bundle equipped with a bundle map $[ \, \, , \, \, ] \colon V \otimes V \to V$ which is skew-symmetric and satisfies the Jacobi identity, so that the fibres of $V$ have the structure of Lie algebras. Following \cite{mac} we say that $V$ is a {\em Lie algebra bundle} if there is a fixed Lie algebra $\mathfrak{g}$, an open cover $\{ U_i \}$ of $M$ and trivialisations $V|_{U_i} \simeq U_i \times \mathfrak{g}$ under which the bracket $[ \, \, , \, \, ]$ is identified with the Lie bracket on $\mathfrak{g}$. If $A$ is any transitive Lie algebroid the kernel $V = {\rm Ker}(\rho)$ of the anchor $\rho \colon A \to TM$ is a Lie algebra bundle \cite[IV, Theorem 1.4]{mac}. Additionally we note that $V$ is a Lie algebroid module over $A$, where the action of $a \in \Gamma(A)$ on a section $v \in \Gamma(V)$ is $[a,v]$.\\

Let $V$ be a bundle of Lie algebras. We say that $V$ is a {\em quadratic Lie algebra bundle} if in addition there is a non-degenerate symmetric bilinear form $\langle \, \, , \, \, \rangle \colon V \otimes V \to \mathbb{R}$ which is ad-invariant in the sense that $\langle [a,b] , c \rangle + \langle b , [a,c] \rangle = 0$, for all $a,b,c \in \Gamma(V)$.

Following \cite{csx} we say that a transitive\footnote{More generally the notion of quadratic Lie algebroids can be defined for regular Lie algebroids, that is whenever the anchor has constant rank.} Lie algebroid $A$ is a {\em quadratic Lie algebroid} if $V = {\rm Ker}(\rho)$ is given a non-degenerate symmetric bilinear pairing $\langle \, \, , \, \, \rangle$ which is preserved by the module structure in the sense that for all $a\in \Gamma(A), b,c \in \Gamma(V)$ we have
\begin{equation}\label{compat}
\rho(a)\langle b , c \rangle = \langle [a,b] , c \rangle + \langle b , [a,c] \rangle.
\end{equation}
Note that by taking $a$ to be a section of $V$ we find that $\langle \, \, , \, \, \rangle$ must be ad-invariant so that $V$ is a quadratic Lie algebra bundle.


\subsection{Atiyah algebroids}

Let $G$ be a Lie group and $P \to M$ a principal $G$-bundle. Throughout we regard principal $G$-bundles as possessing a right action. The quotient of $TP$ by the action of $G$ defines a vector bundle $\mathcal{A} = TP/G$ on $M$ such that sections of $\mathcal{A}$ can be identified with $G$-invariant vector fields on $TP$. The Lie bracket of two invariant vector fields is again invariant, so there is a Lie bracket on the space of sections of $\mathcal{A}$ which makes $\mathcal{A}$ a Lie algebroid, the {\em Atiyah algebroid} of $P$.\\

Let $\mathfrak{g}$ be the Lie algebra of $G$. Throughout this paper we will assume that $G$ is a compact, connected, semisimple Lie group and equip $\mathfrak{g}$ with a G-invariant, non-degenerate bilinear form $c( \, \, , \, \, ) = \langle \, \, , \, \, \rangle$. For any principal $G$-bundle $P \to M$ with Atiyah algebroid $\mathcal{A}$, the kernel $V$ of the anchor $\rho \colon \mathcal{A} \to TM$ identifies with the adjoint bundle $\mathfrak{g}_P = P \times_G \mathfrak{g}$. The pairing $\langle \, \, , \, \, \rangle$ on $\mathfrak{g}$ is $G$-invariant, so defines a natural pairing on $\mathfrak{g}_P$ which we continue to denote by $\langle \, \, , \, \, \rangle$. This pairing makes $\mathcal{A}$ into a quadratic Lie algebroid, indeed let $A,B,C$ be invariant vector fields on $P$ such that $B,C$ are vertical. Then $B,C$ can be regarded as $\mathfrak{g}$-valued functions $B,C \colon P \to \mathfrak{g}$ such that $B(pg) = {\rm Ad}_{g^{-1}}(B(p))$ and similarly for $C$. It follows that $A\langle B , C \rangle = \langle 
A(B) , C \rangle + \langle B , A(C) \rangle$ from which we see that (\ref{compat}) holds.\\

For any $x \in \mathfrak{g}$ we write $\psi(x)$ for the associated vector field on $P$. We have that $\psi(  [x,y] ) = [ \psi(x) , \psi(y) ]$, where we use left invariant vector fields on $G$ to define the Lie bracket on $\mathfrak{g}$. Let $A$ be a connection on $P$. Thus $A \in \Omega^1(P,\mathfrak{g})$ is a $\mathfrak{g}$-valued $1$-form on $P$ such that $i_{\psi(x)} A = x$ for all $x \in \mathfrak{g}$ and $R_g^* A = {\rm Ad}_{g^{-1}} A$, where $R_g : P \to P$ denotes the right action of $g \in G$ on $P$. The curvature of $A$ is the $\mathfrak{g}$-valued $2$-form $F \in \Omega^2(P,\mathfrak{g})$ given by $F = dA + \frac{1}{2}[A,A]$. We can also regard $F$ as a $2$-form valued section of the adjoint bundle.\\

A connection $A$ determines a splitting of the exact sequence $0 \to \mathfrak{g}^P \to \mathcal{A} \to TM \to 0$. For any vector field $X$ on $M$ there is a unique vector field $X^H$ on $P$ such that $X^H$ projects to $X$ and $i_{X^H} A = 0$. We call $X^H$ the horizontal lift of $X$. By invariance of $A$ we have that $X^H$ is an invariant vector field and thus a section of $\mathcal{A}$. This gives the splitting of the exact sequence for $\mathcal{A}$. With respect to this splitting $\mathcal{A} = TM \oplus \mathfrak{g}_P$ one finds that the Lie bracket on sections of $\mathcal{A}$ is given by:
\begin{equation}\label{atiyah}
[ X + s , Y + t ]_{\mathcal{A}} = [X,Y] + \nabla_X t - \nabla_Y s - [s,t] - F(X,Y),
\end{equation}
where $[s,t]$ is defined using the algebraic Lie bracket on $\mathfrak{g}_P$ and $\nabla_X$ is the covariant derivative determined by the connection $A$.


\subsection{Heterotic Courant algebroids}

Let $\mathcal{H}$ be a transitive Courant algebroid on $M$. Then the anchor $\rho \colon \mathcal{H} \to TM$ dualises to an injective map $\rho^* \colon T^*M \to \mathcal{H}$, where we use the pairing on $\mathcal{H}$ to identify $\mathcal{H}$ with its dual. The quotient $\mathcal{A} = \mathcal{H}/T^*M$ inherits the structure of a transitive Lie algebroid. The anchor of $\mathcal{A}$ is induced by the anchor $\rho$ of $\mathcal{H}$ and shall also be denoted as $\rho$. The kernel of $\rho \colon \mathcal{A} \to TM$ naturally inherits a bilinear pairing from the pairing $\langle \, \, , \, \, \rangle$ on $\mathcal{H}$ and we will denote this also by $\langle \, \, , \, \, \rangle$. It is straightforward to see that this pairing makes $\mathcal{A}$ into a quadratic Lie algebroid.

\begin{definition}\label{defhetca}
We say that a transitive Courant algebroid $\mathcal{H}$ is a {\em heterotic Courant algebroid} if there exists a principal $G$-bundle $P$ such that $\mathcal{A} = \mathcal{H}/T^*M$ is isomorphic to the Atiyah algebroid of $P$ as a quadratic Lie algebroid.
\end{definition}

Let $\mathcal{H}$ be a heterotic Courant algebroid associated to the Atiyah algebroid $\mathcal{A} \simeq \mathcal{H}/T^*M$ and let $\mathfrak{g}_P$ denote the adjoint bundle. Define $\mathcal{K}$ to be the kernel of the anchor $\rho \colon \mathcal{H} \to TM$. Then $T^*M \subseteq \mathcal{K}$ and $\mathcal{K}/T^*M$ identifies with the kernel of the induced anchor $\mathcal{A} \to TM$, which is $\mathfrak{g}_P$. We have exact sequences
\begin{eqnarray}
&0  \to \mathcal{K} \to \mathcal{H} \to TM \to 0 \label{exseq1} \\ 
&0  \to T^*M \to \mathcal{K} \to \mathfrak{g}_P \to 0. \label{exseq2}
\end{eqnarray}
By a {\em splitting} of $\mathcal{H}$ we mean a section $s \colon TM \to \mathcal{H}$ of the anchor $\rho$ such that the image $s(TM) \subseteq \mathcal{H}$ is isotropic with respect to the pairing on $\mathcal{H}$. As with the exact case isotropic splittings always exist. Given a splitting $s$ we obtain a splitting of the exact sequence (\ref{exseq1}). Moreover the subspace $s(TM)^\perp \cap \mathcal{K}$ defines a lift of $\mathfrak{g}_P$ to $\mathcal{K}$ which is orthogonal to $T^*M$. Thus a splitting $s$ also determines a splitting of (\ref{exseq2}), giving a decomposition
\begin{equation}\label{equdecomh}
\mathcal{H} = TM \oplus \mathfrak{g}_P \oplus T^*M
\end{equation}
such that the anchor and pairing are given by
\begin{eqnarray}
\rho(X , s , \xi ) \! \! \! \! &=& \! \! \! \! X \label{equanchor} \\
\langle (X,s,\xi) , (Y,t,\eta) \rangle \! \! \! \! &=& \! \! \! \! \tfrac{1}{2}( i_X \eta + i_Y \xi) + \langle s , t \rangle. \label{equpairing}
\end{eqnarray}

In general there is an obstruction for a quadratic Lie algebroid $\mathcal{A}$ associated to a principal $G$-bundle $P \to M$ to arise from a transitive Courant algebroid $\mathcal{H}$ as the quotient $\mathcal{A} = \mathcal{H}/T^*M$. Let $A$ be a connection on $P$  with curvature $2$-form $F$. For suitable choice of $\langle \, \, , \, \, \rangle$, the characteristic class represented by the closed $4$-form $\langle F \wedge F \rangle$ is the first Pontryagin class of $P$. For convenience we will refer to the cohomology class represented by $\langle F \wedge F \rangle$ for any fixed choice of pairing $\langle \, \, , \, \, \rangle$ as the {\em first Pontryagin class} of $P$ (with respect to $\langle \, \, , \, \, \rangle$). It is well known that $\mathcal{A}$ comes from a transitive Courant algebroid $\mathcal{H}$ if and only if the first Pontryagin class vanishes \cite{bre}. In fact we have the following classification:
\begin{proposition}\label{hetca}
Let $P \to M$ be a principal $G$-bundle with Atiyah algebroid $\mathcal{A}$ and $A$ a connection on $P$ with curvature $F$. Let $H^0$ be a $3$-form on $M$ such that $dH^0 = \langle F,F \rangle$. Such a pair $(A,H^0)$ determines a transitive Courant algebroid $\mathcal{H}$ such that $\mathcal{H}/T^*M = \mathcal{A}$. The bundle $\mathcal{H}$ is given by (\ref{equdecomh}) with anchor given by (\ref{equanchor}), pairing by (\ref{equpairing}) and bracket given by:
\begin{equation}\label{bracket}
\begin{aligned}
\left[ X + s + \xi , Y + t + \eta \right]_{\mathcal{H}} &= [X,Y] + \nabla_X t - \nabla_Y s -[s,t] - F(X,Y) \\
& \; \; \; \; \; \; + \mathcal{L}_X \eta - i_Y d\xi + i_Y i_X H^0 \\
& \; \; \; \; \; \;  + 2\langle t , i_XF \rangle - 2\langle s ,i_Y F \rangle + 2\langle \nabla s , t \rangle,
\end{aligned}
\end{equation}
where $X,Y \in \Gamma(TM), s,t \in \Gamma(\mathfrak{g}_P), \xi,\eta \in \Gamma(T^*M)$. Conversely given a transitive Courant algebroid $\mathcal{H}$ and a splitting $s \colon TM \to \mathcal{H}$ there exists $(A,H^0)$ with $dH^0 = \langle F , F \rangle$, such that the bracket is given by (\ref{bracket}).
\end{proposition}
The proof is an immediate consequence of the well known classification of transitive Courant algebroids \cite{bre},\cite{vai},\cite{csx}.\\

Let $\mathcal{H}$ be the bundle $TM \oplus \mathfrak{g}_P \oplus T^*M$. Given a $2$-form $B$ and a $\mathfrak{g}_P$-valued 1-form $A$, we define endomorphisms $e^B : \mathcal{H} \to \mathcal{H}$, $e^A : \mathcal{H} \to \mathcal{H}$ as follows:
\begin{eqnarray*}
e^B( X + s + \xi) &=& X + s + \xi + i_X B \\
e^A(X+s+\xi) &=& X + s-AX + \xi + \langle 2s - AX , A \rangle.
\end{eqnarray*}
We call $e^B$ a {\em $B$-shift} and $e^A$ an {\em $A$-shift}. Note that $A$- and $B$-shifts preserve the anchor (\ref{equanchor}) and pairing (\ref{equpairing}). Suppose that $\nabla$ is a $G$-connection with curvature $F_\nabla$ and $H^0$ is a $3$-form such that $dH^0 = \langle F_\nabla , F_\nabla \rangle$. Then we may equip $\mathcal{H}$ with the corresponding Dorfman bracket given by (\ref{bracket}). We write this as $[ \, \, , \, \, ]_{\nabla,H^0}$ to show the dependence on $(\nabla,H^0)$. In general $A$- and $B$-shifts do not preserve this bracket. However we have the following identities:
\begin{eqnarray}
\left[ e^Bu , e^Bv \right]_{\nabla,H^0} &=& e^B \left[u,v \right]_{\nabla,H^0+dB} \label{equbshift} \\
\left[ e^Au , e^Av \right]_{\nabla,H^0} &=& e^A \left[u,v \right]_{\nabla + A , H^0 + 2\langle A \wedge F_\nabla \rangle + \langle A \wedge d_\nabla A \rangle + \tfrac{1}{3} \langle A \wedge [A \wedge A] \rangle }. \label{equashift}
\end{eqnarray}


\subsection{Heterotic Courant algebroids by reduction}

We will show that heterotic Courant algebroids are obtained by reduction of exact Courant algebroids. Let $\sigma \colon P \to M$ be a principal $G$-bundle. Every exact Courant algebroid on $P$ can be constructed by taking a closed $G$-invariant $3$-form $H \in \Omega^3(P)$ and taking the $H$-twisted Dorfman bracket on $E = TP \oplus T^*P$. By $G$-invariance  of $H$, the action of $G$ on $P$ extends immediately to an action on $E$ which moreover preserves the splitting $E = TP \oplus T^*P$. We are therefore in the situation described in Section 2.2 of \cite{bcg} in which trivially extended actions $\alpha \colon \mathfrak{g} \to \Gamma(E)$ up to equivalence correspond to solutions of $d_G( \Phi ) = c$, where $\Phi = H + \xi$ for $\xi \colon \mathfrak{g} \to \Gamma(T^*M)$ the $T^*M$-component of $\alpha$ regarded as a section of $\left( \mathfrak{g}^* \otimes \Omega^1(P) \right)^G$ and $c \in \left( S^2 \mathfrak{g}^* \otimes \Omega^0(P)  \right)^G$ is the quadratic form on $\mathfrak{g}$ given by $c(x,y) = -\langle \alpha(x) , \alpha(y) \rangle$. It is natural to consider the case where $c = \langle \, \, , \, \, \rangle$.

\begin{proposition}\label{extact}
Equivalence classes of solutions to $d_G( H+\xi ) = c = \langle \, \, , \, \, \rangle$ are represented by pairs $(H^0,A)$, where $H^0$ is a $3$-form on $M$ and $A$ is a connection on $P$ with curvature $F$ such that $dH^0 = \langle F,F \rangle$. The corresponding pair $(H,\xi)$ is given by
\begin{equation*}
\begin{aligned}
H &= \sigma^*(H^0) - CS_3(A) \\
\xi &= -cA.
\end{aligned}
\end{equation*}
Here $c$ is viewed as a map $c \colon \mathfrak{g} \to \mathfrak{g}^*$ and $CS_3(A)$ is the Chern-Simons $3$-form of $A$
\begin{equation*}
CS_3(A) = c(A,F) - \frac{1}{3!}c( A , [A,A] ).
\end{equation*}
\end{proposition}
\begin{proof}
We let $\psi \colon \mathfrak{g} \to \Gamma(TP)$ denote the map sending an element $x \in \mathfrak{g}$ to the corresponding vector field on $P$. Then an extended action $\alpha \colon \mathfrak{g} \to \Gamma(E)$ has the form $\alpha(x) = ( \psi(x) , \xi(x) )$, where $\xi \colon \mathfrak{g} \to \Gamma(T^*P)$. We can think of $\xi$ as a $\mathfrak{g}^*$-valued $1$-form on $P$ or using $c$ we find that there is a $\mathfrak{g}$-valued $1$-form $A' \in \Omega^1(P,\mathfrak{g})$ such that $\xi = -cA'$. We have that $A'$ must be invariant in the sense that $R_g^* A' = {\rm Ad}_{g^{-1}} A'$. The condition that $c(x,y) = -\langle \alpha(x) , \alpha(y) \rangle$ becomes $2c(x,y) = c( i_{\psi(x)} A', y ) + c( i_{\psi(y)} A' , x )$. To express the general solution let us fix a basis $e_1, \dots , e_m$ for $\mathfrak{g}$, a dual basis $e^1, \dots , e^m$ and set $c_{ij} = c(e_i,e_j)$. Now choose an arbitrary connection $A_0 = A_0^i e_i$ on $P$. Then we can decompose $A'$ as
\begin{equation*}
A' = a^i_j A_0^j e_i + B^i e_i
\end{equation*}
where the $B^i$ satisfy $i_{\psi(x)} B^i = 0$ for all $x \in \mathfrak{g}$. Thus $a_i^r c_{rj} + c_{ri} a^r_j = 2c_{ij}$. The general solution has the form $a_i^r c_{rj} = c_{ij} + \beta_{ij}$, where $\beta_{ij} = -\beta_{ji}$. Substituting we obtain
\begin{equation*}
A' = A_0 + B^i e_i + \beta_{im} c^{mj} A_0^i e_j = A + \beta_{im} c^{mj} A_0^i e_j
\end{equation*}
where $A$ is the connection $A = A_0 + B^i e_i$. From this we find that $\xi = -cA'$ is given by $\xi = -cA - \beta_{ij} A_0^i e^j$. We now argue that the term $\beta_{ij} A_0^i e^j$ can be eliminated. If $\Phi = H + \xi$ is a solution to $d_G (\Phi) = c$, then an equivalent solution is given by $\Phi' = \Phi - d_G( \frac{1}{2} \beta_{ij} A_0^i A_0^j ) = H' + \xi'$ where $H' = H - d( \frac{1}{2} \beta_{ij} A_0^i A_0^j )$ and $\xi' = \xi + \beta_{ij}A_0^i e^j = -cA$. We have shown that modulo equivalence every solution to $d_G( H + \xi ) = c$ is given by an extended action of the form $\xi = -cA$ for some connection $A$.\\

Given a pair $(H,A)$ consisting of a closed invariant $3$-form $H$ and a connection $A$ we must determine when the condition $d_G( H - cA ) = c$ is satisfied. Write $A = A^i e_i$ and observe that $\iota( A^i) = e^i$. Then $d_G( -cA ) = d_G( -c_{ij} A^i e^j) = -c_{ij} dA^i e^j + c_{ij} e^i e^j$, so we have $d_G(H) -c_{ij}dA^i e^j = 0$. To proceed let us decompose $H$ into the most general form
\begin{equation*}
H = H^0 + H^1_i A^i + \frac{1}{2!} H^2_{ij} A^{ij} + \frac{1}{3!} H^3_{ijk} A^{ijk},
\end{equation*}
where $A^{i_1 i_2 \dots i_k} = A^{i_1} \wedge \dots \wedge A^{i_k}$ and we omit pullback notation. Let us introduce structure constants $c_{ij}^k$ such that $[e_i , e_j] = c_{ij}^k e_k$ and set $c_{ijk} = c_{ij}^m c_{mj}$. Then $c_{ijk}$ is skew-symmetric in $i,j,k$. The curvature $F^i e_i$ is given by $F^k = dA^k + \frac{1}{2} c_{ij}^k A^{ij}$. The equation $d_G(H-cA) = c$ takes the form $d_G(H) -c_{ij}F^i e^j + \frac{1}{2} c_{ijk} A^{ij} e^k = 0$. Equating coefficients of $e^i$ we arrive at the conditions $H^3_{ijk} = c_{ijk}$, $H^2_{ij} = 0$, $H^1_i = -c_{ij}F^j$. Therefore $H$ is of the form
\begin{equation*}
H = H^0 -c_{ij} A^i F^j + \frac{1}{3!} c_{ijk} A^{ijk}.
\end{equation*}
We recognise $c_{ij} A^i F^j - \frac{1}{3!} c_{ijk} A^{ijk}$ to be the Chern-Simons $3$-form of $A$, which we denote by $CS_3(A)$. Thus $H = H^0 - CS_3(A)$. Finally the equation $d_G( H -cA) = c$ also requires that $H$ is a closed form. Since $d CS_3(A) = c(F,F)$ this is equivalent to the condition $dH^0 = c(F,F)$. Conversely a pair $(H^0,A)$ where $H^0$ is a $3$-form on $M$ satisfying $dH^0 = c(F,F)$ defines a solution to $d_G( H -cA) = c$ with $H = H^0 - CS_3(A)$.
\end{proof}

Let $\sigma \colon P \to M$ be a principal $G$-bundle and $(H^0,A)$ a pair satisfying the conditions of Proposition \ref{extact} so that $H = \sigma^*(H^0) - CS_3(A)$ is a closed $3$-form on $P$ and we have an extended action $\alpha \colon \mathfrak{g} \to \Gamma(E)$, where $E = TP \oplus T^*P$ with the $H$-twisted Dorfman bracket and $\alpha = (\psi , \xi)$ with $\psi \colon \mathfrak{g} \to TP$ the map sending an element of $\mathfrak{g}$ to the corresponding vector field and $\xi \colon \mathfrak{g} \to \Gamma(T^*P)$ is given by $\xi = -cA$. Since our extended action integrates to an action of $G$ on $E$ we obtain a Courant algebroid by reduction. Let $K \subset E$ be the image of $\alpha$ which in the present situation is a subbundle of $E$ and let $K^\perp$ be the annihilator of $K$ with respect to the pairing on $E$. According the Proposition \ref{redisca} we obtain a Courant algebroid $E_{{\rm red}}$ on $P/G = M$, the reduction of $E$. As a vector bundle $E_{{\rm red}} = K^\perp/G$. Sections of $E_{{\
rm red}}$ can be identified with $G$-invariant sections of $K$ and in this way $E_{{\rm red}}$ inherits the structure of a Courant algebroid on $M$. 

\begin{proposition}\label{reduct}
Given an extended action $\alpha \colon \mathfrak{g} \to \Gamma(E)$ corresponding to a pair $(H^0,A)$ there is an isomorphism of Courant algebroids $f \colon \mathcal{H} \to E_{{\rm red}}$ where $\mathcal{H} = TM \oplus \mathfrak{g}_P \oplus T^*M$ is the heterotic Courant algebroid associated to the pair $(H^0,A)$ as in Proposition \ref{hetca}. In fact we can take $f$ to be as follows:
\begin{equation*}
f( X + s + \xi ) = X^H + s + c(A,s) + \sigma^*(\xi)
\end{equation*}
where $X^H$ is the horizontal lift of $X$ determined by the connection $A$ and $s \in \Omega^0(M , \mathfrak{g}_P)$ is identified with an invariant vertical vector field on $P$.
\end{proposition}
\begin{proof}
Let us take the map $f \colon \mathcal{H} \to E_{{\rm red}}$ to be defined as above. We show that $f$ is an isomorphism of Courant algebroids. First we need to show that $f$ is well defined in the sense that it takes a section $X + s + \xi$ of $\mathcal{H}$ to a section of $E_{{\rm red}}$. Recall first that sections of $E_{{\rm red}}$ naturally identify with $G$-invariant sections of $K^\perp$. This is easily verified by noting that the bundle $K$ is spanned by sections of the form $t - c(A,t)$ where $t \in \Omega^0(M,\mathfrak{g}_P)$. It is trivial to verify that $f$ preserves anchors and pairings, so all that remains is to check that $f$ preserves Dorfman brackets. Let $X+s+\xi$ and $Y+t+\eta$ be sections of $\mathcal{H}$. We compute (omitting pullback notation)
\begin{equation*}
\begin{aligned}
\left[ f(X+s+\xi) , f(Y+t+\eta) \right]_{E} =& \; [ X^H + s +c(A,s) + \xi , Y^H + t +c(A,t) + \eta]_E \\
= & \; [X^H + s , Y^H + t ] + \mathcal{L}_{(X^H + s)}( c(A,t) + \eta) \\
& \; - i_{(Y^H + t)} d( c(A,s) + \xi) \\
& \; + i_{(Y^H + t)} i_{(X^H + s)} ( H^0 - CS_3(A)).
\end{aligned}
\end{equation*}
The term $[X^H + s , Y^H + t ]$ is the commutator of invariant vector fields on $P$, so it is given by the formula (\ref{atiyah}) for the Atiyah algebroid:
\begin{equation*}
[X^H + s , Y^H + t ] = [X,Y]^H + \nabla_X t - \nabla_Y s - [s,t] - F(X,Y).
\end{equation*}
Consider now the expression $i_{(X^H + s)} d( c(A,t) + \eta)$. After some simplification we find
\begin{equation*}
i_{(X^H + s)} d( c(A,t) + \eta) = i_X d\eta -c(s,\nabla t) +c(i_XF,t) -c([s,t] , A) +c(A, \nabla_X t).
\end{equation*}
Likewise we find that
\begin{equation*}
d  i_{(X^H + s)} (c(A,t) + \eta) = di_X \eta + d c(s,t).
\end{equation*}
Finally we also have
\begin{equation*}
\begin{aligned}
i_{(Y^H + t)} i_{(X^H + s)} ( H^0 - CS_3(A)) =&  \, \, i_Y i_X H^0 -c(s,i_YF) +c(t,i_XF) \\
& -c(A,F(X,Y)) +c(A,[s,t]).
\end{aligned}
\end{equation*}
Combining all of these calculations we arrive at
\begin{equation*}
\begin{aligned}
\left[ f(X+s+\xi) , f(Y+t+\eta) \right]_{E} =& [X,Y]^H +\nabla_X t - \nabla_Y s -[s,t] -F(X,Y) \\
& c(\nabla_X t , A) -c(\nabla_Y s , A) -c([s,t],A) \\
& -c(F(X,Y),A) -2c(i_Y F , s) +2c(i_XF,t) \\
& + 2c(\nabla s , t) + \mathcal{L}_X \eta - i_Y d\xi + i_Yi_X H^0.
\end{aligned}
\end{equation*}
The right hand side is $f( [X+s+\xi , Y + t + \eta]_{\mathcal{H}} )$ as required.
\end{proof}

Combining Propositions \ref{hetca} and \ref{reduct} we obtain the following classification of heterotic Courant algebroids:
\begin{proposition}\label{hetred}
Every heterotic Courant algebroid on $M$ is obtained by reduction of an exact Courant algebroid $E = TP \oplus T^*P$ on a principal $G$-bundle $\sigma \colon P \to M$ with flux a closed invariant $3$-form $H \in \Omega^3(P)$. The reduction is by a trivially extended action $\xi \colon \mathfrak{g} \to \Gamma(T^*P)$ of the form $\xi = -cA$, where $A$ is a connection such that $H = \sigma^*(H^0) - CS_3(A)$.
\end{proposition}

We saw in Section \ref{secrbea} that equivalent actions give rise to isomorphic reductions, so equivalence classes of extended actions give rise to isomorphic heterotic Courant algebroids.\\

Suppose that $\sigma \colon P \to M$ is a principal $G$-bundle with vanishing first Pontryagin class. Let $\mathcal{EA}(P)$ denote the set of equivalence classes of trivially extended actions on $P$. Here by an extended action we mean a pair $(H,\xi)$ such that $H$ is a closed $G$-invariant $3$-form on $P$ and $\xi$ is an equivariant map $\xi \colon \mathfrak{g} \to \Gamma(T^*P)$ satisfying $d_G( H + \xi ) = c$. Two such pairs $(H,\xi)$ and $(H',\xi')$ are equivalent if there is a degree $2$ class $\beta$ in the Cartan complex such that $H'+\xi' = H + \xi + d_G(\beta)$. Since we are assuming that $P$ has trivial first Pontryagin class we know that $\mathcal{EA}(P)$ is non-empty. Moreover it is clear that the set of pairs $(H,\xi)$ such that $d_G(H+\xi) = c$ is a torsor over the space of $d_G$-closed classes of degree $3$ in the Cartan complex. Factoring out by equivalence we find that $\mathcal{EA}(P)$ is a torsor over $H^3_G(P,\mathbb{R})$, degree $3$ equivariant cohomology of $P$. Since $P$ is a principal 
$G$-bundle we have that $H^3_G(P,\mathbb{R}) = H^3(M,\mathbb{R})$ and we conclude that $\mathcal{EA}(P)$ is a torsor over $H^3(M,\mathbb{R})$.\\

Let $(H,\xi)$ be an extended action so that $d_G(H + \xi) = c$. In particular $H$ is a closed $3$-form on $P$. Changing $(H,\xi)$ to an equivalent extended action changes $H$ by an exact form, so the cohomology class $[H] \in H^3(P,\mathbb{R})$ depends only on the equivalence class of the pair $(H,\xi)$. We know that any such $H$ has the form $H = \sigma^*(H^0) - CS_3(A)$, for some $3$-form $H^0$ on $M$ and a connection $A$ on $P$. The restriction of $H$ to any fibre of $P$ is given by
\begin{equation}\label{cart3}
\omega_3 = \frac{1}{3!} c( [\omega , \omega ] , \omega)
\end{equation}
where $\omega \in \Omega^1(G,\mathfrak{g})$ is the left Maurer-Cartan form on $G$. We call $\omega_3$ the {\em Cartan $3$-form}. Thus $[H] \in H^3(P,\mathbb{R})$ has the property that its restriction to any fibre yields the class $\omega_3 \in H^3(G,\mathbb{R})$.

\begin{definition}\label{defstrcl}
A {\em real string class} is a class $x \in H^3(P,\mathbb{R})$ such that the restriction of $x$ to any fibre of $P$ coincides with $\omega_3$. 
\end{definition}

Since $G$ is taken to be compact, connected and semisimple we have $H^0(G,\mathbb{R}) = \mathbb{R}$, $H^1(G,\mathbb{R}) = H^2(G,\mathbb{R}) = 0$ and $H^3(G,\mathbb{R}) \neq 0$, with $\omega_3$ representing a non-trivial class. Considering the Leray-Serre spectral sequence for $G \to EG \to BG$ with cohomology in real coefficients we see that the existence of a real string class on $P$ is equivalent to the vanishing of the first Pontryagin class in real coefficients. Let us denote by $\mathcal{SC}(P)$ the set of real string classes on $P$.

\begin{proposition}\label{scea}
Suppose $\sigma : P \to M$ is a principal $G$-bundle with vanishing real first Pontryagin class. Then the set $\mathcal{SC}(P)$ is a torsor for $H^3(M,\mathbb{R})$ where the action of $y \in H^3(M,\mathbb{R})$ on $x \in \mathcal{SC}(P)$ is given by $x + \sigma^*(y)$. The map $\mathcal{EA}(P) \to \mathcal{SC}(P)$ which sends a pair $(H,\xi)$ to the cohomology class $[H]$ is an isomorphism of $H^3(M,\mathbb{R})$-torsors.
\end{proposition}
\begin{proof}
The fact that $\mathcal{SC}(P)$ is a torsor over $H^3(M,\mathbb{R})$ is immediate from the Leray-Serre spectral sequence in real coefficients for $\sigma \colon P \to M$. The map which sends a pair $(H,\xi)$ to the class $[H]$ clearly respects the $H^3(M,\mathbb{R})$ actions, so is necessarily a bijection.
\end{proof}

We also need to consider string classes in integral cohomology. For this suppose the pairing $\langle \, \, , \, \, \rangle$ on $\mathfrak{g}$ is such that the Cartan $3$-form $\omega_3$ lies in the image $H^3(G,\mathbb{Z}) \to H^3(G,\mathbb{R})$. Fix a choice of lifting of $\omega_3$ to a class in integral cohomology. For most simple compact connected $G$ this is a trivial matter as $H^3(G,\mathbb{Z}) = \mathbb{Z}$. The exceptions are the groups $SO(4n)/\mathbb{Z}_2$, $n \ge 2$, which have $H^3(G,\mathbb{Z}) = \mathbb{Z} \oplus \mathbb{Z}_2$ \cite{fgk}. 

\begin{definition}\label{stringcl}
For a fixed choice of lift $\omega_3 \in H^3(G,\mathbb{Z})$ we will say that a class $h \in H^3(P,\mathbb{Z})$ on a principal $G$-bundle $P \to X$ is a {\em string class} if $h$ restricts to $\omega_3$ on the fibres of $P$. 
\end{definition}

For the spin groups $G = Spin(m)$, $m \ge 3$, it is shown in \cite{red} that string classes are in bijection with string structures as introduced by Killingback \cite{kil}.


\section{Heterotic T-duality}\label{sechettd}


\subsection{Review of T-duality}

We begin with a brief review of topological T-duality \cite{bem},\cite{bhm1},\cite{bunksch},\cite{brs},\cite{bar0},\cite{bar}. Let $T^n = \mathbb{R}^n/\mathbb{Z}^n$ be a rank $n$ torus, $\hat{T}^n = {\mathbb{R}^n}^*/{\mathbb{Z}^n}^*$ the dual torus. Let $\mathfrak{t}^n = \mathbb{R}^n$ be the Lie algebra of $T^n$ and ${\mathfrak{t}^n}^* = {\mathbb{R}^n}^*$ the Lie algebra of $\hat{T}^n$. Fix a basis $t_1, \dots , t_n$ for $\mathfrak{t}^n$ and dual basis $t^1, \dots , t^n$.\\

Let $\pi \colon X \to M$ be a principal $T^n$-bundle over $M$ and $\hat{\pi} \colon \hat{X} \to M$ a principal $\hat{T}^n$-bundle over $X$. Let $C = X \times_M \hat{X}$ be the fibre product, which is a principal $T^n \times \hat{T}^n$-bundle. Let $p \colon C \to X$, $\hat{p} \colon C \to \hat{X}$ be the projections and set $q = \pi \circ p = \hat{\pi} \circ \hat{p}$. Denote by $\mathcal{C}^*_C$ the sheaf of smooth $\mathbb{C}^*$-valued functions on $C$ and similarly define $\mathcal{C}^*_X, \mathcal{C}^*_{\hat{X}}$. Suppose now that we have classes $h \in H^3(X,\mathbb{Z}) \simeq H^2(X , \mathcal{C}^*_X)$, $\hat{h} \in H^3(\hat{X} , \mathbb{Z} ) \simeq H^2(\hat{X} , \mathcal{C}^*_{\hat{X}})$. There is a natural identification of $H^2(T^n \times \hat{T}^n , \mathbb{Z})$ with $\left(\wedge^2 \mathfrak{t}^n \right) \oplus \left( \mathfrak{t}^n \otimes {\mathfrak{t}^n}^* \right) \oplus \left( \wedge^2 {\mathfrak{t}^n}^* \right)$. Let $\mathcal{P} \in H^2( T^n \times \hat{T}^n , \mathbb{Z})$ be the class 
corresponding to the 
identity in $End(\mathfrak{t}) \simeq \mathfrak{t}^n \otimes {\mathfrak{t}^n}^*$. Then $\mathcal{P}$ may be thought of as a line bundle on $T^n \times \hat{T}^n$ which we call the {\em Poincar\'e line bundle}. Since translations act trivially on $H^2(T^n \times \hat{T}^n , \mathbb{Z})$, we find that $\mathcal{P}$ defines a class in $H^0( M , R^1q_*( \mathcal{C}^*_C))$ which we continue to denote by $\mathcal{P}$. Let $d_2 \colon H^1(M , R^1q_*(\mathcal{C}^*_C)) \to H^2(M , q_*(\mathcal{C}^*_C))$ denote the differential $d_2 \colon E_2^{0,1} \to E^{2,0}_2$ in the Leray-Serre spectral sequence for $q \colon C \to M$ using the sheaf $\mathcal{C}^*_C$. Observe that we have natural pullback maps $p^* \colon \pi_*( \mathcal{C}^*_X) \to q_*(\mathcal{C}^*_C)$, $\hat{p}^* \colon \hat{\pi}_*(\mathcal{C}^*_{\hat{X}}) \to q_*(\mathcal{C}^*_C)$. The following definition is easily shown to be equivalent to the definition of T-duality in \cite{bar}.
\begin{definition}\label{deftd}
The pairs $(X,h),(\hat{X},\hat{h})$ are {\em T-dual} if the following conditions hold:
\begin{itemize}
\item{$h$ is the image of a class $h' \in H^0( M , \pi_*(\mathcal{C}^*_X))$ under the natural map $H^2( M , \pi_*(\mathcal{C}^*_X)) \to H^2(X ,\mathcal{C}^*_X)$.}
\item{Similarly $\hat{h}$ is the image of a class $\hat{h}' \in H^2( M , \hat{\pi}_*(\mathcal{C}^*_{\hat{X}}))$.}
\item{The classes $h',\hat{h}'$ may be chosen such that $\hat{p}^*( \hat{h}') - p^*(h') = d_2 \mathcal{P}$.}
\end{itemize}
In this case we also say that $(\hat{X},\hat{h})$ is a T-dual of $(X,h)$.
\end{definition}

An immediate consequence of this definition is that if $(X,h),(\hat{X},\hat{h})$ are T-dual then $\hat{p}^*(\hat{h}) = p^*(h)$, so that $h,\hat{h}$ coincide on $C$. Let $(X,h)$ be a pair consisting of a principal $T^n$-bundle $\pi \colon X \to M$ and a class $h \in H^2( X , \mathcal{C}^*_X)$. We say that $(X,h)$ is {\em T-dualisable} if there exists a T-dual $(\hat{X},\hat{h})$ of $(X,h)$. It turns out that $(X,h)$ is T-dualisable if and only if $h$ is in the image of the natural map $H^2( M , \pi_*(\mathcal{C}^*_X)) \to H^2(X ,\mathcal{C}^*_X)$.\\

A key aspect of T-duality is that it determines an isomorphism of Courant algebroids. For this suppose that $(X,h),(\hat{X},\hat{h})$ are T-duals over $M$. Let $\theta \in \Omega^1(X , \mathfrak{t}^n)$ be a $T^n$-connection for  $\pi \colon X \to M$ and let $\hat{\theta} \in \Omega^1(\hat{X} , {\mathfrak{t}^n}^*)$ be a $\hat{T}^n$-connection for $\hat{\pi} \colon \hat{X} \to M$. Let $F \in \Omega^2(M , \mathfrak{t}^n)$ be the curvature of $\theta$ and $\hat{F} \in \Omega^2(M , {\mathfrak{t}^n}^*)$ the curvature of $\hat{X}$. Write $\theta = \theta^i t_i$, $F = F^i t_i$ and $\hat{\theta} = \hat{\theta}_i t^i$, $\hat{F} = \hat{F}_i t^i$.
\begin{proposition}[\cite{bar}]\label{proptdt}
For any choice of connections $\theta,\hat{\theta}$ there exists a $3$-form $\overline{H} \in \Omega^3(M)$ such that $d\overline{H} + F^i \wedge F_i = 0$. Define $H \in \Omega^3(X), \hat{H} \in \Omega^3(\hat{X})$ as follows:
\begin{equation*}
\begin{aligned}
H &= \pi^*(\overline{H}) + \pi^*(\hat{F}_i) \wedge \theta^i \\
\hat{H} &= \hat{\pi}^*(\overline{H}) + \hat{\pi}^*(F^i) \wedge \hat{\theta}_i.
\end{aligned}
\end{equation*}
Then $H,\hat{H}$ are closed forms, which by a suitable choice of $\overline{H}$ can be made to represent the cohomology classes $h,\hat{h}$. Moreover it follows that
\begin{equation}\label{equhhhat}
\hat{\pi}^*(\hat{H}) - \pi^*(H) = d ( p^*(\theta^i) \wedge \hat{p}^*(\hat{\theta}_i) ).
\end{equation}
\end{proposition}
Equation (\ref{equhhhat}) translates the T-duality condition $\hat{p}^*( \hat{h}') - p^*(h') = d_2 \mathcal{P}$ into a statement involving differential forms.\\

Given T-duals $(X,h),(\hat{H},\hat{h})$ and forms $(\theta,\hat{\theta},\overline{H})$ as in Proposition \ref{proptdt} we obtain an isomorphism of Courant algebroids as follows. Let $E = TX \oplus T^*X$ be the exact Courant algebroid on $X$ with $H$-twisted Dorfman bracket and $\hat{E} = T\hat{X} \oplus T^*\hat{X}$ the exact Courant algebroid on $\hat{X}$ with $\hat{H}$-twisted Dorfman bracket. Then since $H,\hat{H}$ are invariant under the torus actions we have that $T^n$ acts on $E$ and $\hat{T}^n$ acts on $\hat{E}$. Using the connections $\theta,\hat{\theta}$ we obtain identifications
\begin{equation}\label{equidents}
\begin{aligned}
E/T^n &= TM \oplus \mathfrak{t}^n \oplus (\mathfrak{t}^n)^* \oplus T^*M \\
\hat{E}/\hat{T}^n &=TM \oplus (\mathfrak{t}^n)^* \oplus \mathfrak{t}^n \oplus T^*M.
\end{aligned}
\end{equation}
Then we define a map $\phi \colon E/T^n \to \hat{E}/\hat{T}^n$ which up to sign is the map swapping the inner two factors:
\begin{equation}\label{equphi}
\phi( X , t , u , \xi ) = (X , -u , -t , \xi ).
\end{equation}
\begin{proposition}[\cite{guacav},\cite{bar}]\label{propphiiso}
The map $\phi \colon E/T^n \to \hat{E}/\hat{T}^n$ defined by Equation (\ref{equphi}) is an isomorphism of Courant algebroids on $M$.
\end{proposition}


\subsection{Overview of heterotic T-duality}\label{secover}

Since ordinary T-duality involves an isomorphism of exact Courant algebroids, a natural approach to developing heterotic T-duality is to construct in a similar manner an isomorphism between heterotic Courant algebroids. We have established that heterotic Courant algebroids are obtained from exact Courant algebroids by reduction. Instead of performing the reductions first and then constructing a T-duality isomorphism of heterotic Courant algebroids, we will reverse the process, using ordinary T-duality to construct an isomorphism of exact Courant algebroids and then use reductions to obtain an isomorphism of heterotic Courant algebroids.\\

Let $\sigma \colon P \to X$ be a principal $G$-bundle over $X$ and $h \in H^3(P,\mathbb{R})$ be a real string class. Let $\mathcal{H}$ be the heterotic Courant algebroid on $X$ associated to $h$. Suppose now that $X$ is a principal $T^n$-bundle $\pi_0 \colon X \to M$ over a manifold $M$, where $T^n$ is the rank $n$ torus $T^n = \mathbb{R}^n / \mathbb{Z}^n$. We seek to promote the torus action on $X$ to an action on $\mathcal{H}$. For this we assume that the $T^n$-action on $X$ lifts to a $T^n$-action on $P$ by principal bundle automorphisms. Thus the $T^n$ and $G$-actions on $P$ commute and we can view $P$ as a principal $T^n \times G$-bundle over $M$. Set $P_0 = P/T^n$, so that $P_0$ is a principal $G$-bundle over $M$. Define $\pi,\sigma_0$ to be the projections $\pi \colon P \to P_0$ and $\sigma_0 \colon P_0 \to M$. Recall the commutative diagram (\ref{comdiag}). Let $E$ be the exact Courant algebroid on $P$ with \v{S}evera class $h$. In order to lift $T^n$ to an action on $\mathcal{H}$ we will consider 
lifting 
$T^n$ to an action on $E$ preserving an extended action. Choose a representative $H \in \Omega^3(P)$ for $h$ invariant under $G \times T^n$. This is certainly possible as $G \times T^n$ is compact. We can take $E = TP \oplus T^*P$ with the $H$-twisted Dorfman bracket, then since $H$ is invariant we obtain an action of $G \times T^n$ on $E$. In order to apply the results of Section \ref{seccomred} we need to be able to choose an extended action $\alpha \colon \mathfrak{g} \to \Gamma(E)$ which maps to $T^n$-invariant sections. This will be made clear in Section \ref{sectdcwr}.\\

To implement T-duality we must assume the \v{S}evera class $h \in H^3(P,\mathbb{R})$ lifts to an integral cohomology class $h \in H^3(P,\mathbb{Z})$, thus $h$ is an integral string class on $P$. Since $P \to P_0$ is a principal $T^n$-bundle, we may consider T-dualising the pair $(P,h)$ to obtain a $\hat{T}^n$-bundle $\hat{\pi} \colon \hat{P} \to P_0$ and class $\hat{h} \in H^3(P,\mathbb{Z})$. The existence of such a T-dual requires the vanishing of an obstruction derived from $h$. If $\hat{h}$ is a real string class on $\hat{P}$ then we can construct a corresponding T-dual heterotic Courant algebroid. There is a problem as it can happen that the $G$-action on $P_0$ does not lift to an action on $\hat{P}$ commuting with the $\hat{T}^n$-action. In Section \ref{sectdsc} we will see that there is a cohomological obstruction in finding such a lift of the $G$-action. We will show that when the obstruction vanishes, we have a $G \times \hat{T}^n$-action on $\hat{P}$ and on setting $\hat{X} = \hat{P}/G$ we get a 
commutative diagram
\begin{equation*}
\xymatrix{
P \ar[r]^{\pi} \ar[d]^{\sigma} & P_0 \ar[d]^{\sigma_0} & \hat{P} \ar[l]_{\hat{\pi}} \ar[d]^{\hat{\sigma}} \\
X \ar[r]^{\pi_0} & M & \hat{X} \ar[l]_{\hat{\pi}_0}
}
\end{equation*}
Let $C = P \times_{P_0} \hat{P}$ be the fibre product and $p \colon C \to P$, $\hat{p} \colon C \to \hat{P}$ the projections. One of the requirements of T-duality is that $\hat{p}^*(\hat{h}) = p^*(h)$. We will use this to show that $\hat{h}$ is a string class. In this way we obtain a T-duality for string classes. 

Let $\mathcal{H}$, $\hat{\mathcal{H}}$ be the heterotic Courant algebroids on $X,\hat{X}$ corresponding to the string classes $h,\hat{h}$. The existence of torus actions on $P$ and $\hat{P}$ determines torus actions on $\mathcal{H}$ and $\hat{\mathcal{H}}$. We will proceed to show that there is an isomorphism $\mathcal{H}/T^n \simeq \hat{\mathcal{H}}/\hat{T}^n$. Let $E,\hat{E}$ be the exact Courant algebroids on $P,\hat{P}$ corresponding to $h,\hat{h}$. By T-duality of $(P,h)$ and $(\hat{P},\hat{h})$ there is an isomorphism $\phi \colon E/T^n \to \hat{E}/\hat{T}^n$. Now if we can construct the isomorphism $\phi$ in such a way that it interchanges the extended actions on $E$ and $\hat{E}$ then it will induce an isomorphism $(E/T^n)_{{\rm red}} \simeq (\hat{E}/\hat{T}^n)_{{\rm red}}$. From this and Proposition \ref{commutingred} we have a series of isomorphisms
\begin{equation*}
\mathcal{H}/ T^n \simeq (E/T^n)_{{\rm red}} \simeq (\hat{E}/\hat{T}^n)_{{\rm red}} \simeq \hat{\mathcal{H}} / \hat{T}^n
\end{equation*}
which will establish the desired isomorphism $\mathcal{H}/T^n \simeq \hat{\mathcal{H}}/\hat{T}^n$.


\subsection{T-duality of string classes}\label{sectdsc}

As described in Section \ref{secover} the starting point for heterotic T-duality is a principal $G \times T^n$-bundle $P$ over $M$. Recall again the commutative diagram (\ref{comdiag}), noting that $\pi \colon P \to P_0$ is a principal $T^n$-bundle. Let $h \in H^3(P,\mathbb{Z})$ be a string class. Suppose that $h$ is T-dualisable and let $(\hat{P},\hat{h})$ be a T-dual. Thus $\hat{P} \to P_0$ is a principal $\hat{T}^n$-bundle, where $\hat{T}^n$ is the dual torus to $T^n$ and $\hat{h} \in H^3(\hat{P},\mathbb{Z})$. We are interested in the case that $\hat{P} \to P_0$ is the pullback of a principal $\hat{T}^n$-bundle on $M$. In general there is an obstruction to doing so.
\begin{definition}\label{deftdsc}
Let $P \to M$ be a principal $G \times T^n$-bundle and $h \in H^3(P,\mathbb{Z})$. Set $P_0 = P/T^n$. We say that $h$ is a {\em T-dualisable string class} if
\begin{itemize}
\item{$h$ is T-dualisable with respect to the $T^n$-bundle $P \to P_0$}
\item{The restriction of $h$ to any fibre $G \times T^n$ of $P$ agrees with the class $\omega_3 \in H^3(G,\mathbb{Z}) \subset H^3(G \times T^n , \mathbb{Z})$.}
\end{itemize}
\end{definition}
If $h$ is T-dualisable then the restriction $h|_{G \times T^n}$ of $h$ to a fibre $G \times T^n$ must be a T-dualisable class on the trivial $T^n$-bundle $G \times T^n \to G$. This means that $h|_{G \times T^n}$ must lie in the subgroup $H^3(G,\mathbb{Z}) \oplus H^2(G, H^1(T^n,\mathbb{Z}))$ of $H^3(G \times T^n , \mathbb{Z})$. If $h$ is also a string class then the component of $h$ along the fibre $G$ equals the class $\omega_3$. Thus the obstruction for a string class which is T-dualisable to satisfy Definition \ref{deftdsc} is the component of $h|_{G \times T^n}$ which lies in $H^2(G, H^1(T^n,\mathbb{Z}))$. 
\begin{proposition}\label{proptopcond}
Let $h \in H^3(P,\mathbb{Z})$ be a string class which is T-dualisable with respect to the $T^n$-bundle $P \to P_0$. Restriction of $h$ to the $G \times T^n$-fibres of $P \to M$ defines a class $\kappa \in H^2(G, H^1(T^n,\mathbb{Z}))$. Then $\kappa = 0$ if and only if for every T-dual $(\hat{P},\hat{h})$ of $(P,h)$ over $P_0$ the $\hat{T}^n$-bundle $\hat{P} \to P_0$ is the pullback under $\sigma_0 \colon P_0 \to M$ of a $\hat{T}^n$-bundle $\hat{X} \to M$. Moreover in the case that $\kappa = 0$ the bundle $\hat{X}$ such that $\hat{P} = \sigma_0^*(\hat{X})$ is unique and $\hat{h}$ is a T-dualisable string class.
\end{proposition}
\begin{proof}
Let $(\hat{P},\hat{h})$ be a T-dual. Restricting the base $P_0$ to a fibre $G$ we get a pair $(\hat{P}|_G , \hat{h}|_G)$ which is a T-dual to the restriction of $(P,h)$ over $G$. In particular this means that $\kappa \in H^2(G,H^1(T^n,\mathbb{Z}))$ is the Chern class of $\hat{P}|_G \to G$. If $\kappa \neq 0$ then any $\hat{P}|_G \to G$ is a non-trivial bundle and $\hat{P}$ can not be the pullback of some bundle $\hat{X} \to M$.\\

Conversely assume $\kappa = 0$. Let $(\hat{P},\hat{h})$ be a T-dual over $P_0$. Restricting $\hat{P}$ over $G$ yields a trivial bundle. By considering the Leray-Serre spectral sequence for $P_0 \to M$ with coefficients in $H^1(T^n,\mathbb{Z})$ we see that the Chern class $\hat{c}$ of $\hat{P}$ must be a pullback from $M$, hence $\hat{P}$ itself is a pullback. From the spectral sequence and the fact that $H^1(G,H^1(T^n,\mathbb{Z})) = 0$ we get that $\hat{X}$ is unique. Finally since $\kappa = 0$ we see that the restriction of $h$ to $G \times T^n$ is equal to the pullback of $\omega_3$ under the projection $G \times T^n \to G$. By T-duality we similarly have that the restriction of $\hat{h}$ to a fibre $G \times \hat{T}^n$ equals the pullback of $\omega_3$ under the projection $G \times \hat{T}^n \to G$, hence $\hat{h}$ is a T-dualisable string class.
\end{proof}

Let $\sigma \colon P \to X$ be a principal $G$-bundle. Let $A$ be a connection on $P$ with curvature $F$. The $4$-form $c(F \wedge F)$ defines a characteristic class $p_1 \in H^4(BG,\mathbb{R})$ which we refer to as the first Pontryagin class. Given a faithful representation $V$ of $G$ we can normalise $k$ so that $p_1$ is the usual first Pontryagin class of the associated bundle $P \times_G V$ but it is not necessary for us to do so. We take $\omega_3 \in H^3(G,\mathbb{R})$ to be the transgression of $p_1$, which we recall is given by the Cartan $3$-form (\ref{cart3}).\\

Recall that $\hat{T}^n$ is defined to be the dual torus to $T^n$. One way to formalise this is to take $\hat{T}^n = H^1( T^n , \mathbb{R} ) / H^1( \hat{T}^n , \mathbb{Z})$. In particular this determines a dual pairing between the lattices $H^1(T^n , \mathbb{Z} )$ and $H^1( \hat{T}^n , \mathbb{Z})$, which we denote by $\langle \, \, , \, \, \rangle$.

\begin{proposition}\label{topcond}
Suppose that $h \in H^3(P,\mathbb{Z})$ is a T-dualisable string class. Let $(\hat{P},\hat{h})$ be any T-dual and let $\hat{\pi}_0 \colon \hat{X} \to M$ be the $\hat{T}^n$-bundle over $M$ such that $\hat{P} \to P_0$ is the pullback of $X$ under $\sigma_0 \colon P_0 \to M$. Let $c \in H^2(M , H^1(\hat{T}^n,\mathbb{Z}))$ and $\hat{c} \in H^2(M , H^1(T^n,\mathbb{Z}))$ be the Chern classes of $X \to M$ and $\hat{X} \to M$. The following holds in $H^4(M,\mathbb{R})$:
\begin{equation*}
-p_1(P_0) + \langle c , \hat{c} \rangle = 0.
\end{equation*}
\end{proposition}
\begin{proof}
Choose a connection $A$ on $P_0$ and pull it back to a connection on $P$. We can then represent $h$ by a closed $G \times T^n$-invariant $3$-form $H \in \Omega^3(P)$ such that $H = \sigma^*(H^0) - CS_3(A)$ for some $H^0 \in \Omega^3(X)$. Now choose a connection $\theta \in \Omega^1(X,\mathfrak{t}^n)$ for the principal $T^n$-bundle $X \to M$, where $\mathfrak{t}^n$ is the Lie algebra of $T^n$. Choose a basis $t_1, \dots , t_n$ for $\mathfrak{t}^n$ with dual basis $t^1, \dots , t^n$. We can identify $H^1(T^n,\mathbb{Z})$ with the lattice spanned by $t^1, \dots, t^n$ in $(\mathfrak{t}^n)^*$ and $H^1(\hat{T}^n,\mathbb{Z})$ with the lattice spanned by $t_1, \dots , t_n$ in $\mathfrak{t}^n$. We write $\theta = \theta^i t_i$. Now we may decompose $H^0$ as
\begin{equation*}
H^0 = \pi_0^*(K^0) + \pi^*_0(K^1_i) \theta^i + \frac{1}{2} \pi^*_0(K^2_{ij} )\theta^{ij} + \frac{1}{3!} \pi^*_0(K^3_{ijk}) \theta^{ijk} 
\end{equation*}
where $K^0, K^1_i , \dots , K^3_{ijk}$ are forms on $M$. Now since $H = \sigma^*(H^0) - CS_3(A)$ is T-dualisable the components $K^3_{ijk}$ must vanish and the components $K^2_{ij}$, must be exact, say $K^2_{ij} = dB^2_{ij}$. On replacing $H^0$ by $H^0 - dB$ and $H$ by $\sigma^*(H^0 - dB) - CS_3(A)$, where $B = \frac{1}{2}B^2_{ij} \theta^{ij}$ we may assume that $K^3_{ijk} = 0$ and $K^2_{ij} = 0$. Then $K^1_i$ is closed, so $K_i t^i$ represents a class in $H^2(M,H^1( T^n , \mathbb{Z} ))$. According to T-duality the dual Chern class $\sigma_0^*(\hat{c}) \in H^2(P_0 , H^1(T^n , \mathbb{Z}))$ agrees with $\sigma_0^*( [K^1_i t^i] )$ up to the image of the differential $d_2 \colon H^0(P_0 , H^2(T^n , \mathbb{Z})) \to H^2(P_0 , H^1(T^n , \mathbb{Z}))$. Now since $\sigma_0^* \colon H^0(M, H^2(T^n,\mathbb{Z})) \to H^0(P_0 , H^2(T^n , \mathbb{Z}))$ is an isomorphism it likewise follows that $\hat{c} \in H^2(M , H^1(T^n , \mathbb{Z}))$ agrees with $[K^1_i t^i]$ up to the image of the differential $d_2 \colon H^0(M, 
H^2(T^n 
, \mathbb{Z})) \to H^2(M , H^1(T^n , \mathbb{Z}))$. If $c = c^i t_i$ is the Chern class of $X \to M$ then the differential $d_2 \colon H^0(M, H^2(T^n , \mathbb{Z})) \to H^2(M , H^1(T^n , \mathbb{Z}))$ sends $t^i \wedge t^j$ to $c^i t^j - c^j t^i$. Write $\hat{c} = c_i t^i$. It follows that we have an equality
\begin{equation*}
\langle c , \hat{c} \rangle = c^i \hat{c}_i = c^i [ K^1_i].
\end{equation*}
Let $F$ be the curvature of $A$ and set $F^i = d\theta^i$, so that $F^i$ is a closed $2$-form on $M$ representing $c^i$. The equation $d( \sigma^*(H^0) - CS_3(A) ) = 0$ gives
\begin{equation*}
dK^0 + K^1_i F^i - c(F_A,F_A) = 0
\end{equation*}
where $F_A$ is the curvature of $A$. Passing to cohomology we obtain $-p_1(P) + [K^1_i] c^i = 0$ and the result follows.
\end{proof}


\subsection{T-duality and fluxes}\label{sectdflux}

Let $h \in H^3(P,\mathbb{Z})$ be a T-dualisable string class. Let $(\hat{P},\hat{h})$ be a T-dual and $\hat{\pi}_0 \colon \hat{X} \to M$ a torus bundle on $M$ which pulls back to $\hat{P}$. Let $C$ be the fibre product $C = P \times_{P_0} \hat{P}$ and $p \colon C \to P$, $\hat{p} \colon C \to \hat{P}$ the projections. We have a commutative diagram
\begin{equation}\label{tdiagram}
\xymatrix{
P \ar[d]^{\sigma} \ar[dr]^{\pi} & C \ar[l]_p \ar[r]^{\hat{p}} & \hat{P} \ar[d]^{\hat{\sigma}} \ar[dl]_{\hat{\pi}} \\
X \ar[dr]_{\pi_0} & P_0 \ar[d]^{\sigma_0} & \hat{X} \ar[dl]^{\hat{\pi}_0} \\ 
& M &
}
\end{equation}
Let $E$ be the exact Courant algebroid on $P$ with \v{S}evera class $h$ and $\hat{E}$ the exact Courant algebroid on $\hat{P}$ with \v{S}evera class $\hat{h}$. If we choose invariant representatives for $h,\hat{h}$ we get torus actions on $E,\hat{E}$. By T-duality there is an isomorphism $E/T^n \simeq \hat{E}/\hat{T}^n$. We construct  heterotic Courant algebroids $\mathcal{H}, \hat{\mathcal{H}}$ on $X$ and $\hat{X}$ from extended actions $\alpha \colon \mathfrak{g} \to \Gamma(E)$, $\hat{\alpha} \colon\mathfrak{g} \to \Gamma(\hat{E})$. If we knew that the isomorphism $E/T^n \simeq \hat{E}/\hat{T}^n$ exchanges the extended actions $\alpha,\hat{\alpha}$ we could obtain an isomorphism of the form $\mathcal{H}/T^n \simeq \mathcal{H}/\hat{T}^n$. In this section we prove that this is indeed possible and gives rise to the desired isomorphism of heterotic Courant algebroids.\\

Let $A \in \Omega^1(P,\mathfrak{g}))$ be any $T^n$-invariant connection for the $G$-bundle $P \to X$ and $\hat{A} \in \Omega^1(\hat{P},\mathfrak{g})$ any $\hat{T}^n$-invariant connection for the $G$-bundle $\hat{P} \to \hat{X}$. Similarly let $\theta \in \Omega^1(P , \mathfrak{t}^n)$ be a $G$-invariant connection for the $T^n$-bundle $P \to P_0$ and $\hat{\theta} \in \Omega^1(P,(\mathfrak{t}^n)^*)$ a $G$-invariant connection for $\hat{P} \to P_0$. We write $\theta = \theta^i t_i$ and $\hat{\theta} = \hat{\theta}_i t^i$. Choose a basis $e_1, \dots , e_m$ for $\mathfrak{g}$ and dual basis $e^1, \dots , e^m$ for $\mathfrak{g}^*$. Let $\psi \colon \mathfrak{g} \to \Gamma(TP)$ be the map sending an element of $\mathfrak{g}$ to the corresponding vector field on $P$ and $\hat{\psi} \colon \mathfrak{g} \to \Gamma(T\hat{P})$ the corresponding map for $\hat{P}$. Recall the operator $\iota$ used to define the differential $d_G = d - \iota$ in the Cartan model for equivariant cohomology. On $P$ we have $\iota \omega = 
e^j i_{\psi(e_j)} \omega$ and similarly for $\hat{P}$ with $\hat{\psi}$ replacing $\psi$. Define sections $v \in \Omega^0(P, \mathfrak{g} \otimes \mathfrak{t}^n)$ and $\hat{v} \in \Omega^0( \hat{P} , \mathfrak{g} \otimes (\mathfrak{t}^n)^*)$ by $\iota \theta = -cv$, $\iota \hat{\theta} = -c\hat{v}$, where we view $c$ as a map $c \colon \mathfrak{g} \to \mathfrak{g}^*$. More explicitly $v = v^i t_i$ and $\hat{v} = \hat{v}_i t^i$, where $i_{\psi(e_j)} \theta^i = -c(e_j, v^i), i_{\psi(e_j)} \hat{\theta}_i = -c(e_j,\hat{v}_i)$.

\begin{definition}
A $4$-tuple of connections $\theta,\hat{\theta},A,\hat{A}$ are {\em compatible} if on $C$ we have an equality (omitting pullback notation):
\begin{equation}\label{compatible}
\hat{A} - A = \langle \theta , \hat{v} \rangle - \langle \hat{\theta} , v \rangle.
\end{equation}
\end{definition}
We note that compatible $4$-tuples of connections $(\theta,\hat{\theta},A,\hat{A})$ always exist as we can take $A = \pi^*(A_0), \hat{A} = \pi^*(A_0)$ for some $G$-connection $A_0$ on $P_0$ and let $\theta,\hat{\theta}$ be pullbacks from $X,\hat{X}$. Then $v,\hat{v}$ vanish, so the connections are compatible.
\begin{proposition}\label{propcompat}
Let $(\theta,\hat{\theta},A,\hat{A})$ be a compatible $4$-tuple of connections. There exists a $G$-connection $A_0$ on $P_0$, a $T^n$-connection $\theta_0$ on $X$, a $\hat{T}^n$-connection on $\hat{X}$ and a $1$-form valued section $\lambda$ of the adjoint bundle $\mathfrak{g}_{P_0}$ on $M$ such that the $4$-tuple $(\theta,\hat{\theta},A,\hat{A})$ is given by (omitting pullback notation):
\begin{equation}\label{4tuple}
\begin{aligned}
\theta &= \theta_0 -c(v,A_0) -c(v,\lambda) \\
\hat{\theta} &= \hat{\theta}_0 -c(\hat{v},A_0) -c(\hat{v},\lambda) \\
A &= A_0 - \langle \theta_0 , \hat{v} \rangle \\
\hat{A} &= A_0 - \langle \hat{\theta}_0 , v \rangle.
\end{aligned}
\end{equation}
In addition $v,\hat{v}$ must satisfy a compatibility condition:
\begin{equation}\label{compatcond}
\langle cv \wedge c\hat{v} \rangle = 0
\end{equation}
where $cv \wedge c\hat{v}$ is the section of $\Omega^0(C , \wedge^2 \mathfrak{g})$ obtained by pulling back $cv$ and $c\hat{v}$ to $C$, pairing $\mathfrak{t}^n$ and $(\mathfrak{t}^n)^*$ factors and wedging $\mathfrak{g}$ factors.

Conversely for any choice of connections $(\theta_0,\hat{\theta}_0,A_0)$, $1$-form valued section of the adjoint bundle $\lambda$ and pair $v,\hat{v}$ satisfying (\ref{compatcond}) we obtain a compatible $4$-tuple of connections by (\ref{4tuple}).
\end{proposition}
\begin{proof}
Starting with Equation (\ref{compatible}) and contracting with the vector field associated to $e_j$ we obtain $0 = c(e_j , v^i) \hat{v}_i - c(e_j , \hat{v}_i)v^i$. Hence for all $j,k$ we have $c(e_j , v^i)c(e_k, \hat{v}_i) - c(e_j , \hat{v}_i )c(e_k, v^i) = 0$ which is Equation (\ref{compatcond}).\\

Let us re-write the compatibility condition (\ref{compatible}) in the form
\begin{equation*}
A + \langle \theta , \hat{v} \rangle = \hat{A} + \langle \hat{\theta} , v \rangle.
\end{equation*}
Since the left hand side is a form on $P$ and the right hand side a form on $\hat{P}$, they must both equal a form $\beta \in \Omega^1(P_0,\mathfrak{g})$ on $P_0$, thus $A = \beta - \langle \theta , \hat{v} \rangle$, $\hat{A} = \beta - \langle \hat{\theta} , v \rangle$. Choose a $G$-connection $B_0$ on $P_0$. Then by the definition of $v,\hat{v}$ we find that $\theta,\hat{\theta}$ admit decompositions $\theta = \theta_0 -c(v , B_0)$, $\hat{\theta} = \hat{\theta}_0 -c(\hat{v} , B_0)$, where $\theta_0$ is a $T^n$-connection on $X$ and $\hat{\theta}_0$ a $\hat{T}^n$-connection on $\hat{X}$.

By (\ref{compatcond}) we have an equality $\langle c(v,B_0) , \hat{v} \rangle = \langle c(\hat{v} , B_0 ) , v \rangle$. Define an equivariant $1$-form $A_0 \in \Omega^1(P_0,\mathfrak{g})$ by $A_0 = \beta + \langle c(v,B_0) , \hat{v} \rangle$. We then have
\begin{equation*}
\begin{aligned}
A &= A_0 - \langle \theta_0 , \hat{v} \rangle \\
\hat{A} &= A_0 - \langle \hat{\theta}_0 , v \rangle.
\end{aligned}
\end{equation*}
This implies that $A_0$ is a $G$-connection. Define $\lambda$ to be the $1$-form valued section of the adjoint bundle on $M$ such that $B_0 = A_0 + \lambda$. Substituting we obtain (\ref{4tuple}). Conversely given $(A_0,\theta_0,\hat{\theta}_0)$, a section $\lambda$ and a pair $v,\hat{v}$ satisfying (\ref{compatcond}), a direct computation shows that the $4$-tuple of connections given by Equations (\ref{4tuple}) is compatible.
\end{proof}

The following proposition shows that a T-dualisable string class $h$ can simultaneously be represented in two ways, one reflecting that $h$ is a string class, the other reflecting that $h$ is T-dualisable:
\begin{proposition}\label{tdrelation0}
Let $(\theta,\hat{\theta},A,\hat{A})$ be a compatible $4$-tuple of connections. Then the class $h$ is represented by a closed $3$-form $H \in \Omega^3(P)$ such that there exists a $T^n$-invariant form $H^0 \in \Omega^3(X)$ and a $G$-invariant form $K^0 \in \Omega^3(P_0)$ such that
\begin{equation*}
\begin{aligned}
H &= H^0 - CS_3(A) \\
& = K^0 + \langle \hat{F} \wedge \theta \rangle
\end{aligned}
\end{equation*}
where $F = d\hat{\theta} \in \Omega^2(P_0 , (\mathfrak{t}^n)^*)$ is the curvature of $\hat{\theta}$.
\end{proposition}
\begin{proof}
First of all note that it suffices to show that we can find a $T^n$-invariant form $H^0 \in \Omega^3(X)$ and a $G$-invariant form $K^0 \in \Omega^3(P_0)$ such that $H^0 - CS_3(A)$ is closed and $H^0 - CS_3(A) = K^0 + \langle \hat{F} \wedge \theta \rangle$. For if we set $H' = H^0 - CS_3(A)$ then by considering Leray-Serre spectral sequences we see that the class $\delta = h - [H'] \in H^3(P,\mathbb{R})$ is a pullback from $M$. Let $H^M$ be a closed $3$-form on $M$ representing $\delta$. Then $h$ is represented by $H'+H^M = (H^0+H^M) - CS_3(A) = (K^0 + H^M) + \langle \hat{F} \wedge \theta \rangle$.\\

By Proposition \ref{propcompat} we may write $(\theta,\hat{\theta},A,\hat{A})$ in the form (\ref{4tuple}) for some $A_0,\theta_0,\hat{\theta}_0,\lambda$. We proceed first under the assumption that $\lambda = 0$. Choose a basis $t_1, \dots, t_n$ for $\mathfrak{t}^n$ with dual basis $t^1 , \dots , t^n$ and write $\theta = \theta^i t_i$, $\theta_0 = \theta_0^i t_i$, $v = v^i t_i$. We also let $F^i = d\theta^i$, $F^i_0 = d\theta^i_0$. Similarly define $\hat{\theta}_i$, $(\hat{\theta}_0)_i$, $\hat{v}_i$, $\hat{F}^i$, $(\hat{F}_0)_i$. The Chern-Simons $3$-form $CS_3(A) = CS_3(A_0 - \langle \theta_0 , \hat{v} \rangle)$ can be expanded in terms of the $\theta^i_0$ to obtain:
\begin{equation*}
\begin{aligned}
CS_3(A) &= CS_3(A_0) - c(A_0 , \hat{v}_i )F_0^i \\
& +\left( -c(dA_0 , \hat{v}_i) + c(d_{A_0} \hat{v}_i , A_0) + c(\hat{v}_jF_0^j , \hat{v}_i)  \right) \wedge \theta^i_0 \\
& +c(d_{A_0} \hat{v}_i , \hat{v}_j ) \theta^{ij}_0 - \frac{1}{3} c(\hat{v}_i , [\hat{v}_j , \hat{v}_k]) \theta_0^{ijk}
\end{aligned}
\end{equation*}
Now we seek to find a $T^n$-invariant $3$-form $H^0$ on $X$ such that $H^0 - CS_3(A)$ is closed and $H^0 - CS_3(A) = K^0 + \langle \hat{F} \wedge \theta \rangle$ for some $T^n$-invariant form $K^0$ on $P_0$. First we write down the most general invariant $3$-form on $X$ decomposed with respect to $\theta_0$:
\begin{equation*}
H^0 = \widetilde{H}^0 + H^1_i \theta^i_0 + \frac{1}{2} H^2_{ij} \theta_0^{ij} + \frac{1}{3!} H^3_{ijk} \theta_0^{ijk}.
\end{equation*}
The components $H^3_{ijk},H^2_{ij},H^1_i$ are uniquely determined by the condition that $H^0+ CS_3(A) = K^0 + \langle \hat{F} \wedge \theta \rangle$, resulting in:
\begin{equation}\label{constrh}
\begin{aligned}
H^3_{ijk} &= -2 c(\hat{v}_i , [\hat{v}_j , \hat{v}_k])\\
H^2_{ij} &= -2 c(d_{A_0} \hat{v}_i , \hat{v}_j ) \\
H^1_i &= -c(dA_0 , \hat{v}_i) + c(d_{A_0} \hat{v}_i , A_0) + c(\hat{v}_jF_0^j , \hat{v}_i) + \hat{F}_i.
\end{aligned}
\end{equation}
The expressions in (\ref{constrh}) for $H^3_{ijk}$, $H^2_{ij}$ are manifestly differential forms on $M$, however the expression for $H^1_i$ a priori appears to only be a form on $P_0$ because of its dependence on $A$. However using $\hat{F}_i = (\hat{F}_0)_i  - d c(\hat{v}_i , A_0)$ we find
\begin{equation*}
H^1_i = -2c(\hat{v}_i , F_{A_0} ) + c(\hat{v}_j F_0^j , \hat{v}_i ) + (\hat{F}_0)_i,
\end{equation*}
where $F_{A_0}$ is the curvature of $A_0$. This shows that $H^1_i$ is in fact a form on $M$. Having made these choices it certainly is the case that $H^0 - CS_3(A) = K^0 + \langle \hat{F} \wedge \theta \rangle$, for some $G$-invariant form $K^0$ on $P_0$. It remains to show that we can choose $\tilde{H}^0 \in \Omega^3(M)$ such that $H^0 - CS_3(A)$ is closed. Note first that
\begin{equation*}
H^0 - CS_3(A) = \tilde{H}^0 - CS_3(A_0) + c(A_0 , \hat{v}_i )F_0^i  + \langle \hat{F} \wedge \theta_0 \rangle.
\end{equation*}
On differentiating we find that
\begin{equation*}
d(H^0 - CS_3(A) ) = d\tilde{H}^0 - c(F_{A_0} , F_{A_0}) + F_0^i \wedge (\hat{F}_0)_i.
\end{equation*}
From Proposition \ref{topcond} we have the topological condition $-p_1(P_0) + [F_0]^i \wedge [ (\hat{F}_0)_i] = 0 \in H^4(M,\mathbb{R})$. This shows that we can indeed find $\tilde{H}^0$ so that $H^0 - CS_3(A)$ is closed. This completes the proof in the case $\lambda = 0$. For the more general case with non-zero $\lambda$ set $\pi = -c(v , \lambda)$, $\hat{\pi} = -c(\hat{v} , \lambda)$. Then $(\theta - \pi , \hat{\theta} -\hat{\pi} , A , \hat{A})$ is a $4$-tuple of compatible connections for which there is no $\lambda$ term. Thus we can find $H^0,K^0$ so that $H^0 - CS_3(A)$ is closed and $H^0 - CS_3(A) = K^0 + \langle d (\hat{\theta} - \hat{\pi}) \wedge (\theta - \pi) \rangle$. Thus $H + d \langle \hat{\pi} \wedge \theta \rangle = (H^0 - d \langle \hat{\pi} \wedge \theta \rangle) - CS_3(A) = (K^0 - \langle \hat{F} - d\hat{\pi} \wedge \pi \rangle - \langle \hat{\pi} \wedge F \rangle ) + \hat{F} \wedge \theta$ has the required form.
\end{proof}

\begin{proposition}\label{tdrelation}
Let $(\theta,\hat{\theta},A,\hat{A})$ be a compatible $4$-tuple of connections. The classes $h,\hat{h}$ admit representatives $H \in \Omega^3(P)$, $\hat{H} \in \Omega^3(P)$ such that $H$ has the form $H = H^0 - CS_3(A)$, where $H^0 \in \Omega^3(X)$ is $T^n$-invariant and such that on $C$
\begin{equation*}
\hat{H} = H + d \langle \theta \wedge \hat{\theta} \rangle.
\end{equation*}
\end{proposition}
\begin{proof}
We may choose $H^0,\hat{H}^0,K^0,\hat{K^0}$ as in Proposition \ref{tdrelation0} such that $H = H^0 - CS_3(A) = K^0 + \langle \hat{F} \wedge \theta \rangle$ is a representative for $h$ and $\hat{H} = \hat{H}^0 - CS_3(\hat{A}) = \hat{K}^0 + \langle F \wedge \hat{\theta} \rangle$ is a representative for $\hat{h}$. As in the proof of \cite[Theorem 4.6]{bar} we can find a $T^n \times \hat{T}^n$-invariant $2$-form $B$ on $C$ such that $\hat{K}^0 - K^0 = dB$ and such that when decomposed into forms on $P_0$ using $\theta, \hat{\theta}$, there are no $\theta^i \wedge \hat{\theta}_j$-terms. On averaging we can assume that $B$ is actually $G \times T^n \times \hat{T}^n$-invariant.\\

Choose $A_0$ to be a connection on $P_0$ such that Equation (\ref{4tuple}) holds and define $\pi$, $\hat{\pi}$ as in the proof of Proposition \ref{tdrelation0}. Then from the proof of Proposition \ref{tdrelation0} we find that $\hat{H} - H$ has the form
\begin{equation*}
\hat{H} - H = J^0 + d \langle \theta \wedge \hat{\theta} \rangle + d \langle \pi \wedge \hat{\theta} \rangle - d \langle \hat{\pi} \wedge \theta \rangle 
\end{equation*}
where $J^0$ is a $3$-form on $M$. On the other hand we clearly have
\begin{equation*}
\hat{H} - H = \hat{K}^0 - K^0 + d \langle \theta \wedge \hat{\theta} \rangle.
\end{equation*}
Combining these with $\hat{K}^0 - K^0 = dB$ we get
\begin{equation*}
d(B - \langle \pi \wedge \hat{\theta} \rangle + \langle \hat{\pi} \wedge \theta \rangle ) = J^0.
\end{equation*}
Let $B' = B - \langle \pi \wedge \hat{\theta} \rangle + \langle \hat{\pi} \wedge \theta \rangle$, so that $B'$ is a $G \times T^n \times \hat{T}^n$-invariant form on $C$ which has no $\theta^i \wedge \hat{\theta}_j$-terms and $dB' = J^0$, where $J^0$ is a form on $M$. We can use $\theta,\hat{\theta},A_0$ to decompose $B'$ in terms of forms on the base $M$:
\begin{equation*}
\begin{aligned}
B' &= B^0 + B^1_i \theta^i + (\hat{B}^1)^i \hat{\theta}_i + \frac{1}{2} B^2_{ij} \theta^{ij} + \frac{1}{2} (\hat{B}^2)^{ij} \hat{\theta}_{ij}\\
& + (C^0 , A_0) + ( C^1_i , A_0)\theta^i + ( (\hat{C}^1)^i , A_0) \hat{\theta}_i + ( C^2 , A_0 \wedge A_0 ).
\end{aligned}
\end{equation*}
In this expression the coefficients $B^0,B^1_i, \dots , (\hat{B}^2)^{ij}$ are forms on $M$. The coefficients $C^0, C^1_i , (\hat{C}^1)^i$ are forms on $M$ valued in the dual $\mathfrak{g}_{P_0}^*$ of adjoint bundle $\mathfrak{g}_{P_0}$ for $P_0 \to M$ and $C^2$ is valued in $\wedge^2 \mathfrak{g}_{P_0}^*$. We use $( \, \, , \, \, )$ to denote the dual pairing of $\mathfrak{g}_{P_0}$ and $\mathfrak{g}_{P_0}^*$ and extend this to a pairing of $\wedge^2 \mathfrak{g}_{P_0}$ and $\wedge^2 \mathfrak{g}_{P_0}^*$.

Since $dB'$ is equal to a form on $M$ it must be that the $C^2$-term determines a class in $H^0(M , H^2(G,\mathbb{R}))$. However $H^2(G, \mathbb{R}) = 0$, so after adding to $B'$ an exact term of the form $d( D , A_0 )$, where $D$ is a $1$-form valued section of $\mathfrak{g}_{P_0}^*$ we may assume $C^2 = 0$. We also have $H^1(G,\mathbb{R}) = 0$, so a similar argument allows us to eliminate the terms $C^0, C^1_i , (\hat{C}^1)^i$ by adding exact terms to $B'$. Since exact terms do not change $dB'$ we still have $dB' = J^0$. Now we may write $B' = \hat{b} - b$, where
\begin{equation*}
\begin{aligned}
b &= -B^0 - B^1_i \theta^i - \frac{1}{2} B^2_{ij} \theta^{ij}\\
\hat{b} & = (\hat{B}^1)^i \hat{\theta}_i + \frac{1}{2} (\hat{B}^2)^{ij} \hat{\theta}_{ij}
\end{aligned}
\end{equation*}
Note that $b$ is an invariant $2$-form on $X$ and $\hat{b}$ is an invariant $2$-form on $\hat{X}$. Then since $dB' = J^0$ we have
\begin{equation*}
d\hat{b} = J^0 + db. 
\end{equation*}
The left hand side is a form on $X$ and the right hand side a form on $\hat{X}$, so $db$ and $d\hat{b}$ must be forms on $M$. We have that $\hat{K}^0 - K^0 = dB$ and $B = \hat{b} - b + \langle \pi \wedge \hat{\theta} \rangle - \langle \hat{\pi} \wedge \theta \rangle$, so
\begin{equation*}
K^0 - d(b + \langle \hat{\pi} \wedge \theta \rangle) = \hat{K}^0 - d( \hat{b} + \langle \pi \wedge \hat{\theta} \rangle ).
\end{equation*}
The left hand side is a form on $P$ while the right hand side is a form on $\hat{P}$, so they must both equal forms on $P_0$. Thus $d(b + \langle \hat{\pi} \wedge \theta \rangle)$ and $d( \hat{b} + \langle \pi \wedge \hat{\theta} \rangle )$ are forms on $P_0$. However, $b + \langle \hat{\pi} \wedge \theta \rangle$ is a form on $X$ so $d(b + \langle \hat{\pi} \wedge \theta \rangle)$ is a form on both $X$ and $P_0$, so it is actually a form on $M$. Similarly $d( \hat{b} + \langle \pi \wedge \hat{\theta} \rangle )$ must actually be a form on $M$. Replacing $H$ by $H - d(b + \langle \hat{\pi} \wedge \theta \rangle)$ and $\hat{H}$ by $\hat{H} - d( \hat{b} + \langle \pi \wedge \hat{\theta} \rangle )$ completes the proof of the proposition.
\end{proof}

The compatibility condition (\ref{compatible}) can be conveniently rewritten in the form:
\begin{equation*}
\hat{A} - A = - \iota \langle \theta \wedge \hat{\theta} \rangle.
\end{equation*}
Let $H,\hat{H},H^0,\hat{H}^0$ be as in Proposition \ref{tdrelation}, so that in particular $\hat{H} = H + d \langle \theta \wedge \hat{\theta} \rangle$. Set $\Phi = H + \xi$, $\hat{\Phi} = \hat{H} + \hat{\xi}$, where $\xi = -cA$, $\hat{\xi} = -c\hat{A}$. Then combining the above results we have:
\begin{equation*}
\hat{\Phi} = \Phi + d_G \langle \theta \wedge \hat{\theta} \rangle.
\end{equation*}
By assumption $H$ has the form $H = H^0 - CS_3(A)$, so from Proposition \ref{extact} we get that $\Phi$ is a solution of $d_G \Phi = c$. From the above we obtain $d_G \hat{\Phi} = c$, so that $\hat{\Phi}$ determines an extended action and $\hat{H}$ has the form $\hat{H} = \hat{H}^0 - CS_3(\hat{A})$ for some $\hat{T}^n$-invariant form $\hat{H}^0 \in \Omega^3(\hat{X})$. We now observe the following: let $(\theta,\hat{\theta},A,\hat{A})$ be any compatible $4$-tuple of connections, $H \in \Omega^3(P)$, $\hat{H} \in \Omega^3(\hat{P})$ any pair of $3$-forms such that $\hat{H} + \hat{\xi} = H + \xi + d_G \langle \theta \wedge \hat{\theta} \rangle $, then necessarily $H,\hat{H}$ are real string classes and satisfy the conditions of Propositions \ref{tdrelation0}, \ref{tdrelation}.
\begin{definition}\label{hettddata}
We say that the data $(\theta,\hat{\theta},A,\hat{A},H,\hat{H})$ satisfies heterotic T-duality if $(\theta,\hat{\theta},A,\hat{A})$ is a compatible $4$-tuple of connections and
\begin{equation*}
\hat{H} + \hat{\xi} = H + \xi + d_G \langle \theta \wedge \hat{\theta} \rangle .
\end{equation*}
\end{definition}
We have established the following existence result:
\begin{proposition}
Let $P \to M$ be a principal $G \times T^n$-bundle and $h \in H^3(P,\mathbb{Z})$ a T-dualisable string class. Let $(\hat{P},\hat{h})$ be a T-dual to $(P,h)$ over $P_0$ and $\hat{\pi}_0 \colon \hat{X} \to M$ a torus bundle on $M$ which pulls back to $\hat{P}$. For any compatible $4$-tuple of connections $(\theta,\hat{\theta},A,\hat{A})$ there exists forms $H,\hat{H}$ representing $h,\hat{h}$ and such that $(\theta,\hat{\theta},A,\hat{A},H,\hat{H})$ satisfies heterotic T-duality.
\end{proposition}
 

\subsection{T-duality commutes with reduction}\label{sectdcwr}

Suppose that we have a principal $G \times T^n$-bundle $P \to M$ and T-dualisable string class $h \in H^3(P,\mathbb{Z})$. As usual let $(\hat{P},\hat{h})$ be a T-dual and suppose the data $(\theta,\hat{\theta},A,\hat{A},H,\hat{H})$ satisfies heterotic T-duality according to Definition \ref{hettddata}. In particular we have $\hat{H} = H + d \langle \theta \wedge \hat{\theta} \rangle$, which we recall determines an isomorphism of exact Courant algebroids.

Let $E = TP \oplus T^*P$ be the exact Courant algebroid on $P$ with $H$-twisted Dorfman bracket and $\hat{E} = T\hat{P} \oplus T^*\hat{P}$ be the exact Courant algebroid on $\hat{P}$ with $\hat{H}$-twisted Dorfman bracket. Then since $H,\hat{H}$ are invariant under the torus actions we have that $T^n$ acts on $E$ and $\hat{T}^n$ acts on $\hat{E}$. Recall that using the connections $\theta,\hat{\theta}$ we obtain the identifications given in Equation (\ref{equidents}). According to Proposition \ref{propphiiso} we have an isomorphism of Courant algebroids $\phi \colon E/T^n \to \hat{E}/\hat{T}^n$ given by (\ref{equphi}).\\

Let $\psi \colon \mathfrak{g} \to \Gamma(TP)$ be the map which associates to an element of $\mathfrak{g}$ the corresponding vector field on $P$ and similarly define $\hat{\psi}$. Note that since the actions of $G$ on $P$ and $\hat{P}$ cover the action of $G$ on $P_0$ we have a commutative diagram
\begin{equation}\label{exarel1}
\xymatrix{
\Gamma(TP) \ar[r] & \Gamma(TP_0) & \Gamma(T\hat{P}) \ar[l] \\
& \mathfrak{g} \ar[ul]^{\psi} \ar[u]^{\psi_0} \ar[ur]_{\hat{\psi}} &
}
\end{equation}
where $\psi_0 \colon \mathfrak{g} \to \Gamma(TP_0)$ is the map corresponding to the action of $G$ on $P_0$. Let $\alpha \colon \mathfrak{g} \to \Gamma(E)$, $\hat{\alpha} \colon \mathfrak{g} \to \Gamma(\hat{E})$ be the extended actions determined by the connections $A,\hat{A}$. Thus $\alpha = (\psi , \xi)$ , $\hat{\alpha} = (\hat{\psi} , \hat{\xi})$ where $\xi = -cA$, $\hat{\xi} = -c\hat{A}$. Since $A,\hat{A}$ are $T^n$-invariant we have that $\alpha$ maps to $T^n$-invariant sections of $E$ and $\hat{\alpha}$ maps to $\hat{T}^n$-invariant sections of $\hat{E}$. Let $\mathcal{H} = E_{{\rm red}}$ be the reduction of $E$ by the extended action $\alpha$ and $\hat{\mathcal{H}} = \hat{E}_{{\rm red}}$ be the reduction by the extended action $\hat{\alpha}$. Thus $\mathcal{H}$ is a heterotic Courant algebroid on $X$ and $\hat{\mathcal{H}}$ is a heterotic Courant algebroid on $\hat{X}$. According to Proposition \ref{commutingred} the $T^n$-action on $E$ naturally defines a $T^n$-action on $\mathcal{H}$ and similarly we 
obtain an action of $\hat{T}^n$ on $\hat{\mathcal{H}}$. We have from Proposition \ref{commutingred} that the extended actions $\alpha,\hat{\alpha}$ naturally define extended actions on $E/T^n$, $\hat{E}/\hat{T}^n$ and that $\mathcal{H}/T^n \simeq (E/T^n)_{{\rm red}}$, $\hat{\mathcal{H}}/\hat{T}^n \simeq (\hat{E}/\hat{T}^n)_{{\rm red}}$.
\begin{proposition}\label{propisohet}
The isomorphism $\phi$ exchanges extended actions $\alpha,\hat{\alpha}$ in the sense that we have a commutative diagram
\begin{equation*}\xymatrix{
\mathfrak{g} \ar[r]^-{\alpha} \ar[dr]_-{\hat{\alpha}} & \Gamma(E) \ar[d]^-\phi \\
& \Gamma(\hat{E})
}
\end{equation*}
It follows that $\phi$ descends to an isomorphism $\phi \colon (E/T^n)_{{\rm red}} \to (\hat{E}/\hat{T}^n)_{{\rm red}}$. Composing with the natural identifications $\mathcal{H}/T^n \simeq (E/T^n)_{{\rm red}}$, $\hat{\mathcal{H}}/\hat{T}^n \simeq (\hat{E}/\hat{T}^n)_{{\rm red}}$ given in Proposition \ref{commutingred} we obtain an isomorphism of heterotic Courant algebroids $\mathcal{H}/T^n \simeq \hat{\mathcal{H}}/\hat{T}^n$.
\end{proposition}
\begin{proof}
Recall the compatibility condition for the connections $(\theta,\hat{\theta} ,A,\hat{A})$, namely
\begin{equation*}
\hat{A} - A = \langle \theta , \hat{v} \rangle - \langle \hat{\theta} , v \rangle.
\end{equation*}
As in the proof of Proposition \ref{propcompat} we find there is a form $\beta \in \Omega^1(P_0 , \mathfrak{g})$ defined on $P_0$ such that:
\begin{equation}\label{exarel2}
\begin{aligned}
A &= \beta - \langle \theta , \hat{v} \rangle \\
\hat{A} &= \beta - \langle \hat{\theta} , v \rangle.
\end{aligned}
\end{equation}
Using the decompositions of (\ref{equidents}) together with (\ref{exarel1}),(\ref{exarel2}) we find that the extended actions $\alpha,\hat{\alpha}$ take the form
\begin{equation*}
\begin{aligned}
\alpha & = ( \psi_0 , -cv , c\hat{v} , -c\beta) \\
\hat{\alpha} & = ( \psi_0 , -c\hat{v} , cv , -c\beta ).
\end{aligned}
\end{equation*}
The relation $\hat{\alpha} = \phi \circ \alpha$ directly follows. From here the proposition is immediate.
\end{proof}


\section{Examples}\label{secex}


\subsection{A general class of examples}\label{secgenex}

Suppose that $M$ is simply connected and that $H^2(M,\mathbb{Z})$ is torsion free and finitely generated, for example if $M$ is a simply connected finite CW complex. Choose $c \in H^2(M,\mathbb{Z})$ and let $\pi_0 \colon X_c \to M$ be the corresponding principal circle bundle. Let $G$ be a compact, connected, simply connected simple Lie group and $\sigma_0: P_a \to M$ a principal $G$-bundle with characteristic class $a \in H^4(M,\mathbb{Z})$ corresponding to the generator of $H^4(BG,\mathbb{Z}) = \mathbb{Z}$. Let $Q_{a,c}$ be the fibre product of $X_c$ and $P_a$. We have a commutative square
\begin{equation*}\xymatrix{
Q_{a,c} \ar[r]^{\pi} \ar[d]^{\sigma} & P_a \ar[d]^{\sigma_0} \\
X_c \ar[r]^{\pi_0} & M.
}
\end{equation*}
Assume $a$ is not a torsion class. By the Leray-Serre spectral sequence for $P_a \to M$ we get $H^1(P_a,\mathbb{Z}) = 0, H^2(P_a,\mathbb{Z}) = H^2(M,\mathbb{Z}), H^3(P_a,\mathbb{Z}) = H^3(M,\mathbb{Z}), H^4(P_a,\mathbb{Z}) = H^4(M,\mathbb{Z})/\langle a \rangle$. Moreover we find that $\sigma_0^* \colon H^i(M,\mathbb{Z}) \to H^i(P_a,\mathbb{Z})$ is an isomorphism for $i = 2,3$ and a surjection for $i=4$.

Define
\begin{equation*}
\begin{aligned}
Ann(c) &= \{ \hat{c} \in H^2(M,\mathbb{Z}) | c \smallsmile \hat{c} = 0 \}, \\
Ann(c \; {\rm mod} \; a) &= \{ \hat{c} \in H^2(M,\mathbb{Z}) | c \smallsmile \hat{c} = 0 ({\rm mod} \; a )\}.
\end{aligned}
\end{equation*}
Note that since $H^2(M,\mathbb{Z})$ is finitely generated and torsion free, the same is true for the above two groups. Then by the Gysin sequence for $Q_{a,c} \to P_a$ we find:
\begin{equation*}
H^3(Q_{a,c} , \mathbb{Z}) = H^3(M,\mathbb{Z}) \oplus Ann(c \; {\rm mod} \; a).
\end{equation*}
Classes in $H^3(M,\mathbb{Z}) \subseteq H^3(Q_{a,c}, \mathbb{Z})$ are pulled back from $M$ and the fibre integration $\pi_* \colon H^3(Q_{a,c} , \mathbb{Z}) \to H^2(P_a,\mathbb{Z})$ is projection to $Ann(c \; {\rm mod} \; a)$ followed by the inclusion $Ann(c \; {\rm mod} \; a) \subseteq H^2(M,\mathbb{Z}) = H^2(P_a,\mathbb{Z})$. Thus a class $h \in H^3(Q_{a,c} , \mathbb{Z})$ consists of a pair $h = (h_M , \hat{c})$ where $h_M \in H^3(M,\mathbb{Z})$ and $\hat{c} \in H^2(M,\mathbb{Z})$ with $c \smallsmile \hat{c} = 0 ({\rm mod} \; a)$. For heterotic T-duality we need $h$ to be a string structure.

\begin{lemma}
If $a$ is not in the image of $c \! \smallsmile \, \colon H^2(M,\mathbb{Z}) \to H^4(M,\mathbb{Z})$ then there are no string structures on $Q_{a,c}$. If $a$ is in the image of $c \! \smallsmile \, \colon H^2(M,\mathbb{Z}) \to H^4(M,\mathbb{Z})$ then $h = (h_M , \hat{c})$ is a string structure if and only if $c \smallsmile \hat{c} = a$.
\end{lemma}
\begin{proof}
Using the Gysin sequence for $X_c \to M$ followed by the Leray-Serre spectral sequence for $Q_{a,c} \to X_c$ we obtain an exact sequence
\begin{equation*}
0 \to H^3(M,\mathbb{Z}) \oplus Ann(c) \to H^3(Q_{a,c} , \mathbb{Z}) \buildrel i^* \over \to H^3(G,\mathbb{Z}) = \mathbb{Z} \buildrel a \over \to H^4(M,\mathbb{Z})/\langle c \rangle
\end{equation*}
where $i \colon G \to Q_{a,c}$ is the inclusion of a fibre. If $a$ is not in the image of $c \! \smallsmile\, \colon H^2(M,\mathbb{Z}) \to H^4(M,\mathbb{Z})$ then by exactness $1 \in \mathbb{Z} = H^3(G,\mathbb{Z})$ is not in the image of $i^*$ and there are no string structures. Suppose now that $a$ is in the image of $c \! \smallsmile \, \colon H^2(M,\mathbb{Z}) \to H^4(M,\mathbb{Z})$. In this case the map $i^* \colon H^3(Q_{a,c} , \mathbb{Z}) \to H^3(G,\mathbb{Z}) = \mathbb{Z}$ may be described as follows. Represent a class $h \in H^3(Q_{a,c} , \mathbb{Z})$ as a pair $h = (h_M , \hat{c})$ where $h_M \in H^3(M,\mathbb{Z})$ and $c \smallsmile \hat{c} = k a$ for some $k \in \mathbb{Z}$. Then $i^*(h) = k$. The result follows on noting that string structures are the classes such that $i^*(h) = 1$.
\end{proof}

Fix a choice of principal $G$-bundle $P_a \to X$ with non-torsion characteristic class $a \in H^4(M,\mathbb{Z})$. Then a circle bundle $Q_{a,c} \to P_a$ with T-dualisable string class corresponds to a triple
\begin{equation*}
(h_M , c , \hat{c}) \in H^3(M,\mathbb{Z}) \oplus H^2(M,\mathbb{Z}) \oplus H^2(M,\mathbb{Z})
\end{equation*}
such that $c \smallsmile \hat{c} = a$. We have established that in this instance heterotic T-duality is given by the involution $(h_M , c , \hat{c}) \mapsto (h_M , \hat{c} , c)$.


\subsection{Case of a trivial $G$-bundle}\label{sectrivial}

Consider the case of a principal $G$-bundle $P_0 = M \times G \to M$ over a space $M$. Let $X_c \to M$ be a principal circle bundle with Chern class $c \in H^2(M,\mathbb{Z})$ and $Q_{0,c} = X_c \times_M P_0 = X_c \times G$ the fibre product. Assume that $G$ is compact, simple and simply connected. Then $H^3(Q_{0,c},\mathbb{Z}) = H^3(X_c,\mathbb{Z}) \oplus H^3(G,\mathbb{Z})$. Let $\omega_3$ denote the generator of $H^3(G,\mathbb{Z}) = \mathbb{Z}$, let $\pi_G \colon Q_{0,c} \to G$ be the projection to $G$ and $\sigma_c \colon Q_{0,c} \to X_c$ the projection to $X_c$. Then a class $h \in H^3(Q_{0,c},\mathbb{Z})$ is a string class if and only if it is of the form $h = \sigma_c^*(h_0) + \pi_G^*(\omega_3)$ for a class $h_0 \in H^3(X_c , \mathbb{Z})$ which is uniquely determined from $h$. It is then easy to see that heterotic T-duality for the pair $(X_c , h)$ reduces to ordinary T-duality for the pair $(X_c , h_0)$. In other words, if $(X_{\hat{c}} , \hat{h})$ is a second pair consisting of a principal circle bundle $X_{\hat{c}}$ with Chern class $\hat{c} \in H^2(M,\mathbb{Z})$ and string class $\hat{h} \in H^3( Q_{0,\hat{c}} , \mathbb{Z})$ with $\hat{h} = \sigma_{\hat{c}}^* (\hat{h}_0) + \pi_G^*(\omega_3)$, then $(X_c,h),(X_{\hat{c}} , \hat{h})$ are heterotic T-duals if and only if $(X_c , h_0),(X_{\hat{c}} , \hat{h}_0)$ are ordinary T-duals. This argument extends easily to the case of higher rank torus bundles.


\subsection{$4$-manifolds}

Suppose that $M$ is a compact simply connected $4$-manifold. Then $H^4(M,\mathbb{Z}) = \mathbb{Z}$ and the cup product of degree $2$ cohomology classes defines an intersection form $\langle \, , \, \rangle$ on $H^2(M,\mathbb{Z})$. Let $G$ be a compact, simple and simply connected Lie group and fix a principal $G$-bundle $P_0 \to M$ and let $a \in \mathbb{Z} = H^4(M,\mathbb{Z})$ be the characteristic class corresponding to the generator of $H^4(BG,\mathbb{Z})$. If $a \neq 0$ then from Section \ref{secgenex} we have that pairs consisting of a principal circle bundle $X \to M$ and T-dualisable string class $h$ on $P = P_0 \times_M X$ correspond to pairs $(c,\hat{c}) \in H^2(M,\mathbb{Z}) \oplus H^2(M,\mathbb{Z})$ such that $\langle c , \hat{c} \rangle = a$. Repeating the analysis of Section \ref{secgenex} in the case of a trivial $G$-bundle, where $a=0$ we find that T-dualisable pairs $(X,h)$ correspond to pairs $(c,\hat{c}) \in H^2(M,\mathbb{Z}) \oplus H^2(M,\mathbb{Z})$ such that $\langle c , \hat{c} \rangle = 0$. As explained in Section \ref{sectrivial}, this is precisely the condition for ordinary T-duality.

In summary ordinary T-duality on $M$ corresponds to solutions of $\langle c , \hat{c} \rangle = 0$, while heterotic T-duality gives the more general equation $\langle c ,\hat{c} \rangle = a$. This extends the applicability of T-duality to a larger class of examples and allows for more flexibility in the possible changes in topology under T-duality. For example in the case $M = \mathbb{CP}^2$ any solution to $\langle c , \hat{c} \rangle = 0$ must have $c = 0$ or $\hat{c} = 0$, while if $a \neq 0$ we obtain many solutions to $\langle c , \hat{c} \rangle = a$ for which $c,\hat{c}$ are both non-zero.\\

Suppose in addition that $M$ is spin and let $\mathcal{F} \to M$ be the principal spin frame bundle. We can extend the analysis of Section \ref{secgenex} to the case of a principal $G \times Spin(4)$-bundle. If $P \to M$ is a principal $G$-bundle with characteristic class $a \in H^4(M,\mathbb{Z})$ then consider the fibre product $P \times_M \mathcal{F} \to M$. As explained in Section \ref{sechetd} the replacement of $P$ with $P \times_M \mathcal{F}$ is necessary for the anomaly cancellation condition (\ref{anom}). Now the equation for a T-dual pair becomes $\langle c , \hat{c} \rangle = a - \frac{1}{2}p_1(TM)$, or using the signature theorem this becomes $\langle c, \hat{c} \rangle = a - \frac{3}{2} \tau$, where $\tau$ is the signature of $M$.
\vspace{-0.1cm}
\subsection{Higher dimensional lens spaces}\label{sechdl}

Consider the case $M = \mathbb{CP}^n$. Then $H^2(M,\mathbb{Z}) = \mathbb{Z}$, $H^3(M,\mathbb{Z}) = 0$, $H^4(M,\mathbb{Z}) = \mathbb{Z}$ and the cup product $H^2(M,\mathbb{Z}) \otimes H^2(M,\mathbb{Z}) \to H^2(M,\mathbb{Z})$ is multiplication. Thus for a given $0 \neq a \in \mathbb{Z}$ we are looking for integers $c,\hat{c}$ such that $c \hat{c} = a$. Let $S^{2n+1} \to \mathbb{CP}^n$ be the Hopf fibration. This has Chern class $1$, so the lens space $\mathbb{Z}_c \backslash S^{2n+1}$ has Chern class $c$, where $\mathbb{Z}_c$ acts as a finite subgroup of the circle action.

The tangent bundle of $S^{2n+1}$ descends to a rank $2n+1$ bundle $V$ over $\mathbb{CP}^n$. Since $V = T\mathbb{CP}^{n} \oplus \mathbb{R}$, the first Pontryagin class of $V$ coincides with the first Pontryagin class of $T\mathbb{CP}^n$ which is $(n+1)\omega^2$, where $\omega$ is the K\"ahler form on $\mathbb{CP}^n$. Suppose that $n$ is odd, so that $\mathbb{CP}^n$ is a spin manifold. Then $V$ admits a spin structure. Recall that the characteristic class generating $H^4( BSpin(k) , \mathbb{Z})$ for $k \ge 4$ is the fractional Pontryagin class $\frac{1}{2}p_1$. For $n \ge 3$ odd we consider the spin bundle $P \to \mathbb{CP}^n$ of $V$ which has fractional Pontryagin class $\frac{1}{2}p_1(V) = \frac{1}{2}(n+1) \omega^2$. The pullback of $P$ to $S^{2n+1}$ is the spin bundle of $S^{2n+1}$, the total space of which is isomorphic to $Spin(2n+1)$. Thus if $c,\hat{c}$ are integers with $c \hat{c} = \frac{1}{2}(n+1)$, then with respect to the principal $Spin(2n+1)$-bundle $P$, the lens spaces $\mathbb{Z}_c \backslash S^{2n+1}$, $\mathbb{Z}_{\hat{c}} \backslash S^{2n+1}$ are heterotic T-duals. 

We examine this example in more detail showing that the lens spaces $\mathbb{Z}_c \backslash S^{2n+1}$ are naturally solutions to the heterotic equations of motion (see Section \ref{secheof}). The action of scalar multiplication by unit complex numbers on $\mathbb{C}^{n+1}$ defines a homomorphism $\phi \colon U(1) \to SO(2n+2)$. If $n$ is odd $\phi$ induces a trivial homomorphism on fundamental groups, so that we may lift to a homomorphism $\tilde{\phi} \colon U(1) \to Spin(2n+2)$. Using $\tilde{\phi}$ the actions of the cyclic groups $\mathbb{Z}_c,\mathbb{Z}_{\hat{c}}$ on $S^{2n+2}$ lift to the spin bundle. From this we can argue that the pullback of $P$ to $\mathbb{Z}_c \backslash S^{2n+1}$ is the spin bundle of $\mathbb{Z}_c \backslash S^{2n+1}$. Since the group $\mathbb{Z}_c$ lifts to a subgroup of $Spin(2n+2)$ we have that the pullback of $P$ to $\mathbb{Z}_c \backslash S^{2n+1}$ can be identified with $\mathbb{Z}_c \backslash Spin(2n+2)$. 

Consider the affine connection $\nabla^L$ on $Spin(2n+2)$ given by left translation. This is a flat connection which preserves the metric on $Spin(2n+2)$ defined by the Killing form and has torsion given by the Cartan $3$-form. By left invariance $\nabla^L$ descends to a flat connection with torsion on the quotient $\mathbb{Z}_c \backslash Spin(2n+2)$. Since this connection is flat it determines a solution to the type II equations of motion (\ref{eint}). Furthermore from Proposition \ref{einstlift} it follows that the lens space $\mathbb{Z}_c \backslash S^{2n+1}$ equipped with the principal $Spin(2n+1)$-bundle $\mathbb{Z}_c \backslash Spin(2n+2)$ is a solution to the heterotic equations of motion (\ref{heint}). In this example we see that heterotic T-duality sends a solution of the heterotic equations of motion on the lens space $\mathbb{Z}_c \backslash Spin(2n+2)$ to another solution on a different lens space $\mathbb{Z}_{\hat{c}} \backslash Spin(2n+2)$. In Section \ref{sechetd} we will show this is a general feature of heterotic T-duality.

\subsection{Further homogeneous examples}

Let $G$ be a compact, simple, simply connected Lie group and $\omega_3(G)$ the generator of $H^3(G , \mathbb{Z}) = \mathbb{Z}$. Given a closed subgroup $\phi \colon H \to G$, where $H$ is also compact, simple and simply connected, we may view $G \to G/H$ as a principal $H$-bundle. Let $\omega_3(H)$ be the generator of $H^3(H,\mathbb{Z})$. Then $i^*(\omega_3(G)) = j_\phi \omega_3(H)$, where $j_\phi$ is the Dynkin index of $\phi$. Thus $\omega_3(G)$ is a string class for the principal $H$-bundle $G \to G/H$ precisely when the subgroup $H \subset G$ has Dynkin index $1$. For example the following inclusions have Dynkin index $1$: $SU(n) \subset SU(n+1)$, $Spin(n) \subset Spin(n+1)$, $Sp(n) \subset Sp(n+1)$, $SU(3) \subset G_2$, $G_2 \subset Spin(7)$, $Spin(9) \subset F_4$, $F_4 \subset E_6$, $E_6 \subset E_7$, $E_7 \subset E_8$. Many more examples can be obtained by composing these inclusions. We will consider here the case $Sp(n) \subset Sp(n+m)$.

Let $\mathbb{H}^k$ denote $k$-dimensional quaternionic space. Under the decomposition $\mathbb{H}^{n+m} = \mathbb{H}^n \oplus \mathbb{H}^m$ we obtain an inclusion $Sp(n) \times Sp(m) \subset Sp(n+m)$. Let $U(1) \subset Sp(n+m)$ be the subgroup corresponding to the diagonal inclusion $U(1) \to T^m$ of $U(1)$ into a maximal torus in the $Sp(m)$ factor. This corresponds to a decomposition $\mathbb{H}^{n+m} = \mathbb{H}^n \oplus \mathbb{C}^m \oplus {\mathbb{C}^m}^*$. Set $M = Sp(n+m)/Sp(n) \times U(1)$ and $X = Sp(n+m)/Sp(n)$. Then $X$ is a principal circle bundle over $M$ with Chern class $\omega$, which by the Leray-Serre spectral sequence is found to be a generator of $H^2(M,\mathbb{Z}) = \mathbb{Z}$. Associated to the fundamental representation of the $Sp(n)$-factor is a rank $2n$ complex vector bundle $E$ on $M$ and associated to the $U(1)$-factor is a complex line bundle $L$. We then have $\mathbb{C}^{2n+2m} = E \oplus (L \oplus L^*) \otimes \mathbb{C}^m$. Moreover $L$ is the line bundle associated to $X$, so $c_1(L) = \omega$ and it follows that $c_2(E) = m\omega^2$. Let $P_0 \to M$ be the principal $Sp(n)$-bundle corresponding to $E \to M$. So $P_0 = Sp(n+m)/U(1)$. Since the characteristic class $c_2$ corresponds to the generator of $H^4(BSp(n),\mathbb{Z})$, we are looking for solutions to $c \hat{c} = m$. Let $X_c \to M$ be the principal circle bundle over $M$ with Chern class $c \omega$. Then it follows that $X_c = Sp(n+m)/ Sp(n) \times \mathbb{Z}_c$ and the pullback of $P_0$ to $X_c$ is given by $Sp(n+m)/ \mathbb{Z}_c$. Arguing as in Section \ref{sechdl}, we obtain on $X_c$ a solution of the heterotic equations of motion. This gives many examples of T-dual pairs satisfying the heterotic equations.

\subsection{A universal construction}

We again take $G$ to be a compact, simple, simply connected Lie group. Let $X$ be the principal $K(\mathbb{Z},3)$-fibration over $K(\mathbb{Z},2) \times K(\mathbb{Z},2) \times BG$ classified by $c \smallsmile \hat{c} - p$, where $c,\hat{c}$ are the universal Chern classes for the two $K(\mathbb{Z},2)$-factors and $p$ is the characteristic class generating $H^4(BG,\mathbb{Z}) = \mathbb{Z}$. The universal $G$-bundle $EG \to BG$ pulls back to a principal $G$-bundle $P_0 \to X$ with characteristic class the pullback of $p$ to $H^4(X,\mathbb{Z})$. Then since $X$ is simply connected and $p \in H^4(X,\mathbb{Z})$ is non-torsion, we are in the setting of Section \ref{secgenex}. On $X$ we have $c \smallsmile \hat{c} = p$, so this gives a pair of heterotic T-dual spaces which may be considered as a universal example. Given any homotopy class of maps $Y \to X$ we may pullback to produce a heterotic T-dual pair over $Y$. Conversely given a space $Y$, a principal $G$-bundle $P$ over $Y$ with characteristic class $p \in H^4(Y,\mathbb{Z})$ and Chern classes $c,\hat{c} \in H^2(Y,\mathbb{Z})$ such that $c \smallsmile \hat{c} = p$, then the classifying map $Y \to K(\mathbb{Z},2) \times K(\mathbb{Z},2) \times BG$ classifying $(c,\hat{c},P)$ lifts to a map $Y \to X$.

\newpage
\section{Generalised metrics and reduction}\label{secgmr}


\subsection{Generalised metrics}\label{secgenmet}

\begin{definition}
Let $E$ be a transitive Courant algebroid on $M$ where $M$ has dimension $n$. We define a {\em generalised metric} on $E$ to be a rank $n$ negative definite subbundle $E_- \subseteq E$. Following \cite{garc}, we say the generalised metric is {\em admissible} if the restriction $\rho |_{E_-}$ of the anchor of $E$ to $E_-$ is a bundle isomorphism $\rho |_{E_-} \colon E_- \to TM$.
\end{definition}
Given a generalised metric $E_- \subseteq E$ we let $E_+ = E_-^\perp$ be the orthogonal complement with respect to the pairing on $E$, so we have an orthogonal decomposition $E = E_- \oplus E_+$. Admissibility is equivalent to the condition $\rho^*(T^*M) \cap E_+ = \{0\}$.\\

An admissible generalised metric determines a section $s = (\rho |_{E_-})^{-1} \colon TM \to E$ of the anchor such that the image $s(TM) = E_-$ is negative definite with respect to the pairing $\langle \, \, , \, \, \rangle$ on $E$. An admissible generalised metric then defines a Riemannian metric $g$ on $M$ by the relation 
\begin{equation}\label{induced}
g(X,Y) = -\langle s(X) , s(Y) \rangle.
\end{equation}
It is possible to define an indefinite signature notion of generalised metrics. For this one takes $E_- \subseteq E$ to be a rank $n$ subbundle which is non-degenerate with respect to the pairing on $E$. We may again say that $E_-$ is admissible if $\rho |_{E_-} \colon E_- \to TM$ is an isomorphism and in this case we obtain an indefinite signature metric $g$ on $M$ by (\ref{induced}).\\

In the case where $E$ is an exact Courant algebroid on $M$ a generalised metric on $E$ is equivalent to an orthogonal decomposition $E = E_- \oplus E_+$ into maximal positive and negative  definite subspaces and is automatically admissible. Identifying $E$ with $TM \oplus T^*M$ we have that there exists a Riemannian metric $g$ and a $2$-form $B$ on $M$ such that
\begin{equation*}
\begin{aligned}
E_+ &= \{ X + gX + BX \; | \; X \in TM \} \\
E_- &=  \{ X - gX + BX \; | \; X \in TM \} .
\end{aligned}
\end{equation*}
Moreover the metric $g$ coincides with the induced metric (\ref{induced}).\\

Next we consider generalised metrics on heterotic Courant algebroids. Let $\mathfrak{g}$ be the Lie algebra of $G$ and suppose that $c = \langle \, \, , \, \, \rangle$ is a non-degenerate invariant bilinear form on $\mathfrak{g}$. Let $\sigma \colon P \to X$ be a principal $G$-bundle on $X$ with Atiyah algebroid $\mathcal{A}$. The invariant pairing $c$ on $\mathfrak{g}$ gives $\mathcal{A}$ the structure of a quadratic Lie algebroid. Let $\mathcal{H}$ be a heterotic Courant algebroid associated to $\mathcal{A}$. Thus $\mathcal{H}$ is a transitive Courant algebroid on $X$ such that the associated quadratic Lie algebroid $\mathcal{H}/T^*X$ is $\mathcal{A}$. Let $\rho \colon \mathcal{H} \to TX$ be the anchor and set $\mathcal{K} = {\rm Ker}(\rho)$. Then we have exact sequences (\ref{exseq1}), (\ref{exseq2}). Recall that we may split the sequences so that as a vector bundle $\mathcal{H} = TX \oplus \mathfrak{g}_P \oplus T^*X$ with anchor given by (\ref{equanchor}) and pairing by (\ref{equpairing}). Suppose $\mathcal{H}_- \subseteq \mathcal{H}$ is an admissible generalised metric. By admissibillity $\mathcal{H}_-$ can be expressed as the graph of a map $TX \to \mathfrak{g}_P \oplus T^*X$ and it follows that $\mathcal{H}_+,\mathcal{H}_-$ have the form
\begin{equation}\label{hetgenmet1}
\begin{aligned}
\mathcal{H}_+ &= \{ X + a + gX + BX + \langle 2a + AX , A \rangle  \; | \; X \in TX, \; a \in \mathfrak{g}_P \} \\
\mathcal{H}_- &=  \{ X -AX - gX + BX - \langle AX , A \rangle \; | \; X \in TX \}.
\end{aligned}
\end{equation}
where $A \in \Omega^1(X, \mathfrak{g}_P)$, $B \in \Omega^2(X)$ and $g$ is the induced Riemannian metric on $X$. Conversely such a triple $(g,A,B)$ defines a generalised metric by (\ref{hetgenmet1}). If the pairing $c( \, \, , \, \, )$ on $\mathfrak{g}$ is positive definite then $\mathcal{H}_+$ is positive definite, so in this case a generalised metric is a decomposition into maximal positive and negative definite subspaces and is automatically admissible.\\

Given a generalised metric $\mathcal{H}_- \subseteq \mathcal{H}$ define an endomorphism $\mathcal{G} \colon \mathcal{H} \to \mathcal{H}$ to be multiplication by $\pm 1$ on $\mathcal{H}_{\pm}$. We obtain a symmetric non-degenerate pairing $\mathcal{G}( \, \, , \, \, )$ on $\mathcal{H}$ by setting $\mathcal{G}( a , b ) = \langle \mathcal{G}a , b \rangle$. Using $\mathcal{G}( \, \, , \, \, )$ we may split exact sequences (\ref{exseq1}) and (\ref{exseq2}) to obtain a decomposition $\mathcal{H} = TX \oplus \mathfrak{g}_P \oplus T^*X$ with pairing and anchor as usual. Define $\pi_{\mathfrak{g}_P} \colon \mathcal{H} \to \mathfrak{g}_P$ and $\pi_{T^*X} \colon \mathcal{H} \to T^*X$ as the projections to the second and third factors. With respect to this splitting we find that $\mathcal{H}_-$ is given as in (\ref{hetgenmet1}) with $A = 0$, $B=0$. The splitting of $\mathcal{H}$ induced by $\mathcal{G}$ in this manner defines a $G$-connection $\nabla$ and a $3$-form $H$ on $X$ as follows:
\begin{equation}\label{nablah}
\begin{aligned}
\nabla_X a &= \pi_{\mathfrak{g}_P} \left( [ X , a ]_{\mathcal{H}} \right), \\
i_Y i_X H &= \pi_{T^*X}([X,Y]_{\mathcal{H}}).
\end{aligned}
\end{equation}
Let $F$ denote the curvature of $\nabla$. As usual we have the relation $dH = c(F,F)$. We have thus shown:
\begin{proposition}\label{genmetclass}
Let $\mathcal{H}$ be a heterotic Courant algebroid associated to the Atiyah algebroid $\mathcal{A}$ and $\mathcal{H}_- \subseteq \mathcal{H}$ an admissible generalised metric. Let $g$ be the Riemannian metric induced by $\mathcal{H}_-$. There exists a $G$-connection $\nabla$ with curvature $F$, a $3$-form $H$ with $dH = \langle F,F \rangle$ and a splitting $s$ of $\mathcal{H}$ such that in the decomposition (\ref{equdecomh}) given by $s$, the Dorfman bracket is given by (\ref{bracket}) and $\mathcal{H}_-$ is given by
\begin{equation*}
\mathcal{H}_- =  \{ X - gX \; | \; X \in TX \}.
\end{equation*}
Moreover the triple $(g,\nabla,H)$ is uniquely determined by the pair $(\mathcal{H},\mathcal{H}_-)$.
\end{proposition}

Proposition \ref{genmetclass} shows that a generalised metric on $\mathcal{H}$ determines a triple $(g,\nabla,H)$ consisting of a Riemannian metric $g$, a $G$-connection $\nabla$ with curvature $F$ and a $3$-form $H$ with $dH = \langle F,F \rangle$. This is precisely the bosonic field content of the low energy limit of heterotic string theory, or of gauged supergravity. Conversely, such a triple $(g,\nabla,H)$ determines a Courant algebroid structure on $\mathcal{H} = TX \oplus \mathfrak{g}_P \oplus T^*X$.\\

We have established that with respect to a fixed decomposition $\mathcal{H} = TX \oplus \mathfrak{g}_P \oplus T^*X$, a generalised metric corresponds to a triple $(g,A,B)$ according to (\ref{hetgenmet1}). From Proposition \ref{hetca} such a splitting determines a $G$-connection $\nabla^0$ and $3$-form $H^0$ according to (\ref{bracket}). On the other hand we have just seen in Proposition \ref{genmetclass} that the generalised metric corresponding to $(g,A,B)$ determines a different splitting of $\mathcal{H}$ and hence a different $G$-connection $\nabla$ and $3$-form $H$ as in (\ref{nablah}). Using (\ref{equbshift}),(\ref{equashift}) we find $\nabla = \nabla^0 + A$, $H = H^0 + dB + 2 \langle A \wedge F^0 \rangle + \langle A , d_{\nabla^0} A \rangle + \frac{1}{3} \langle A \wedge [ A \wedge A ] \rangle$, where $F^0$ is the curvature of $\nabla^0$.


\subsection{Generalised metrics and reduction}\label{genmetred}

Let $\sigma \colon P \to X$ be a principal $G$ bundle and $E$ an exact Courant algebroid on $P$. Suppose that $E_- \subseteq E$ is a $G$-invariant generalised metric on $E$. Then we obtain a $3$-form $H$ on $P$ and a Riemannian metric $g$ such that $E = TP \oplus T^*P$ with $H$-twisted bracket and $E_- = \{ X + gX \; | \; X \in TP \}$. Now suppose that $\xi$ is an extended action which we take to be of the form $\xi = -cA$ for a connection $A$ on $P$. Suppose at first that $c$ is positive definite. We have seen that such a $\xi$ is an extended action if and only if $H$ is of the form $H = \sigma^*(H^0) - CS_3(A)$ with $H^0$ a $3$-form on $X$. The extended action $\xi$ gives a map $\xi \colon\mathfrak{g} \to \Gamma( T^*P )$ such that the image of $\xi$ defines a subbundle $K \subset E$. Moreover the generalised metric defines a map $s \colon TP \to E$ given by the inverse of $\rho |_{E_-} \colon E_- \to TP$. We say that $(E_- , \xi)$ are compatible if the following diagram commutes:
\begin{equation*}
\xymatrix{
\mathfrak{g} \ar[r]^\xi \ar[d] & K \ar[d] \\
TP \ar[r]^s & E.
}
\end{equation*}
Using the metric $g$ on $TP$ we obtain an orthogonal splitting $TP = \sigma^*(TX) \oplus \mathfrak{g}$ and on composing with $s \colon TP \to E$ we obtain a map $s' \colon \sigma^*(TX) \to E$. If $(E_-,\xi)$ are compatible it follows that the image of $s'$ is contained in the complement $K^\perp$. Factoring by the $G$-action we obtain a map $s' \colon TX \to K^\perp/G = \mathcal{H}$, where $\mathcal{H}$ is the heterotic Courant algebroid obtained as the reduction of the extended action. It follows that $s' \colon TX \to \mathcal{H}$ defines an admissible generalised metric $\mathcal{H}_-$ on $\mathcal{H}$. It is clear also that $(E_-,\xi)$ are compatible if and only if $g$ has the form $g = \sigma^*(g^0) + \langle A , A \rangle$ for a Riemannian metric $g^0$ on $TX$. Then $(\mathcal{H}, \mathcal{H}_-)$ correspond to the triple $(g^0 , A , H^0)$. Conversely such a triple $(g^0,A,H^0)$ can be lifted to a generalised metric on $TP$ corresponding to the metric $g = \sigma^*(g^0) + \langle A , A \rangle$ and $3$-
form $H = \sigma^*(H^0) - CS_3(A)$. In this way we see that admissible generalised metrics on $\mathcal{H}$ are obtained by reduction of generalised metrics on the corresponding exact Courant algebroid on $P$.

The above results extend to the case where $c$ is indefinite. The only modification is to note that the metric $g = \sigma^*(g^0) + c( A , A )$ has indefinite signature. As discussed in Section \ref{secgenmet} the notion of generalised metric extends to the indefinite signature case without difficulty. Thus we may continue to interpret generalised metrics on heterotic Courant algebroids as being obtained through reduction, even when $c$ has indefinite signature.


\subsection{The global Buscher rules}

In order to understand T-duality of generalised metrics on heterotic Courant algebroids, we first need to consider the exact case. Thus suppose we have rank $n$ torus bundles $\pi \colon X \to M$, $\hat{\pi} \colon \hat{X} \to M$ which are T-dual in the ordinary sense. In particular there are classes $h \in H^3(X , \mathbb{Z})$, $\hat{h} \in H^3(\hat{X} , \mathbb{Z})$ which are the Dixmier-Douady classes for bundle gerbes on $X,\hat{X}$. Let $h_{\mathbb{R}}, \hat{h}_{\mathbb{R}}$ be the images of $h,\hat{h}$ in real cohomology and $E,\hat{E}$ the exact Courant algebroids associated to $h_{\mathbb{R}}, \hat{h}_{\mathbb{R}}$. Define $T^n, \hat{T}^n, \mathfrak{t}^n, (\mathfrak{t}^n)^*$ as in Section \ref{sechettd}. We consider $X$ as a principal $T^n$-bundle and $\hat{X}$ as a principal $\hat{T}^n$-bundle. Also fix a basis $t_1 , \dots , t_n$ for $\mathfrak{t}^n$ and dual basis $t^1 , \dots , t^n$.\\

Let $E_- \subseteq E$ be a $T^n$-invariant generalised metric on $E$ and suppose $(E,E_-)$ corresponds to the pair $(g,H)$. Then $g$ determines a $T^n$-connection $\theta = \theta^i t_i$ and a metric $\overline{g}$ on $M$ such that $g,H$ have the form $g = \overline{g} + g_{ij} \theta^i \theta^j$, $H = \overline{H} + H_i \theta^i + \tfrac{1}{2} H_{ij} \theta^{ij}$ where we have used the fact that $H$ is a closed $T^n$-invariant form representing the trivial cohomology class on the fibres of $X \to M$ to deduce that there are no $\theta^{ijk}$-terms in the expansion of $H$.

\begin{proposition}\label{globalb}
There exists an isomorphism $\phi \colon E/T^n \to \hat{E}/\hat{T}^n$ of Courant algebroids such that $\hat{E}_- = \phi(E_-)$ is an invariant generalised metric on $\hat{E}$ corresponding to a pair $(\hat{g},\hat{H})$ such that $(g,H),(\hat{g},\hat{H})$ are related as follows. There exists a $\hat{T}^n$-connection $\hat{\theta} = \hat{\theta}_i t^i$ on $\hat{X}$, with curvature $\hat{F}^i t_i$ such that $\hat{g} = \overline{g} + \hat{g}^{ij} \hat{\theta}_i \hat{\theta}_j$, $\hat{H} = \overline{H} + \hat{H}^i \hat{\theta}_i + \tfrac{1}{2} \hat{H}^{ij} \hat{\theta}_{ij}$ and there exists sections $B = \tfrac{1}{2}B_{ij} t^i \wedge t^j \in \Gamma( M , \wedge^2 (\mathfrak{t}^n)^*), \hat{B} = \tfrac{1}{2} \hat{B}^{ij} t_i \wedge t_j \in \Gamma( M , \wedge^2 \mathfrak{t}^n)$ satisfying the relations
\begin{equation*}
\begin{aligned}
\hat{g}^{ij} + \hat{B}^{ij} &= (g_{ij} + B_{ij} )^{-1} \\
H_i &= \hat{F}_i - B_{ij}\hat{F}^j \\
\hat{H}^i &= F^i - \hat{B}^{ij} \hat{F}_j \\
H_{ij} &= dB_{ij} \\
\hat{H}^{ij} &= d\hat{B}^{ij}.
\end{aligned}
\end{equation*}
\end{proposition}
\begin{proof}
Choose a connection $\hat{\theta}^0_i t^i$ on $\hat{X}$ with curvature $\hat{F}^0_i t^i$. From \cite{bar} we have that there exists a $3$-form $\overline{H}^0$ such that $h_\mathbb{R}$ is represented by $H' = \overline{H}^0 + \hat{F}^0_i \wedge \theta^i$ and $\hat{h}_\mathbb{R}$ is represented by $\hat{H}' = \overline{H}^0 + F^i \hat{\theta}^0_i$. Thus $E$ is isomorphic to $TX \oplus T^*X$ with $H'$-twisted bracket. With respect to this decomposition of $E$ we have that $E_- = \{ X - gX + B'X \; | \; X \in TX \}$ for some invariant $2$-form $B'$. We decompose $B'$ with respect to $\theta$ to obtain $B' = \overline{B}' + B_i \wedge \theta^i + \tfrac{1}{2} B_{ij} \theta^{ij}$. On performing a $B$-shift by $e^{-\overline{B}'}$ and replacing $H^0$ by $H^0 + d\overline{B}'$, we may assume $\overline{B}' = 0$. Let $\hat{\theta}$ be given by $\hat{\theta}_i = \hat{\theta}^0_i + B_i$ and let $\hat{F}_i t^i$ be the curvature of $\hat{\theta}$. Then $\hat{h}_\mathbb{R}$ is represented by $\overline{H}^0 + F_i \hat{\theta}^0_i = (\overline{H}^0 - B_i \wedge F^i) + F^i \wedge \hat{\theta}_i = \overline{H} + F^i \wedge \hat{\theta}_i$ where $\overline{H} = \overline{H}^0 - B_i \wedge F^i$. Composing with $e^{-B_i \theta^i}$ we may replace $H'$ by $H' + d( B_i \theta^i) = \overline{H} + \hat{F}^i \theta_i$ and $E_-$ is then given as the graph of $-g + \tfrac{1}{2}B_{ij} \theta^{ij}$. 

Observe that $d \overline{H} + F^i \wedge \hat{F}_i = 0$. It follows from \cite{bar} that there is an isomorphism $\phi \colon E/T^n \to \hat{E}/ \hat{T}^n$ defined as follows. We write $E = TX \oplus T^*X$ with $(\overline{H} + \hat{F}_i \wedge \theta^i)$-twisted bracket and $\hat{E} = T\hat{X} \oplus T^*\hat{X}$ with $(\overline{H} + F^i \wedge \hat{\theta}_i)$-twisted bracket. Then using $\theta,\hat{\theta}$ to split $TX, T\hat{X}$ we have $E/T^n = TM \oplus \mathfrak{t}^n \oplus (\mathfrak{t}^n)^* \oplus T^*M$, $\hat{E}/\hat{T}^n = TM \oplus (\mathfrak{t}^n)^* \oplus \mathfrak{t}^n \oplus T^*M$ and $\phi$ is given by $\phi( X , u , v , \xi ) = (X , -v , -u , \xi)$. Let $\hat{E}_- = \phi(E_-)$. It follows that $\hat{E}_-$ is given by the graph of $-\hat{g} + \tfrac{1}{2} \hat{B}^{ij} \hat{\theta}_{ij}$ where $\hat{g} = \overline{g} + \hat{g}^{ij} \hat{\theta}_i \hat{\theta}_j$ and $(\hat{g}^{ij} + \hat{B}^{ij})$ is the inverse matrix of $(g_{ij} + B_{ij})$. The proposition follows by noting that $H = \overline{H} + \hat{F}_i \wedge \theta^i + d( \tfrac{1}{2} B_{ij} \theta^{ij} )$, $\hat{H} = \overline{H} + F^i \wedge \hat{\theta}_i + d( \tfrac{1}{2} \hat{B}^{ij} \hat{\theta}_{ij} )$.
\end{proof}
We say that pairs $(g,H) , (\hat{g} , \hat{H})$ related as in Proposition \ref{globalb} satisfy the {\em global Buscher rules}.


\subsection{The global heterotic Buscher rules}

Suppose we are in the setting of heterotic T-duality so we have spaces forming the commutative diagram (\ref{tdiagram}) and heterotic Courant algebroids $\mathcal{H},\hat{\mathcal{H}}$ on $X,\hat{X}$. Let $\mathcal{H}_- \subseteq \mathcal{H}$ be a $T^n$-invariant admissible generalised metric on $\mathcal{H}$ and suppose $(\mathcal{H},\mathcal{H}_-)$ corresponds to the triple $(g,A,H)$. Then $g$ determines a $T^n$-connection $\theta = \theta^i t_i$ a metric $\overline{g}$ on $M$ and a $G$-connection $A^0$ on $P_0$ such that $(g,A,H)$ have the form $g = \overline{g} + g_{ij} \theta^i \theta^j$, $A = A^0 + A_i \theta^i$, $H = \overline{H} + H_i \theta^i + \tfrac{1}{2} H_{ij} \theta^{ij} + \tfrac{1}{3!}H_{ijk}$. Consider the symmetric matrix $h_{ij} = g_{ij} + \langle A_i , A_j \rangle$. We say that the generalised metric is {\em T-dualisable} if $h_{ij}$ is positive definite, in particular this implies that for any skew-symmetric matrix $B_{ij}$, the matrix $h_{ij} + B_{ij}$ is invertible.

\begin{proposition}\label{globalhb}
Suppose $\mathcal{H}_-$ is T-dualisable. There exists an isomorphism $\phi \colon \mathcal{H}/T^n \to \hat{\mathcal{H}}/\hat{T}^n$ of Courant algebroids such that $\hat{\mathcal{H}}_- = \phi(\mathcal{H}_-)$ is an invariant generalised metric on $\hat{\mathcal{H}}$ corresponding to a triple $(\hat{g},\hat{A},\hat{H})$ related to $(g,A,H)$ as follows. There exists a $\hat{T}^n$-connection $\hat{\theta} = \hat{\theta}_i t^i$ on $\hat{X}$, with curvature $\hat{F}^i t_i$ such that $\hat{g} = \overline{g} + \hat{g}^{ij} \hat{\theta}_i \hat{\theta}_j$, $\hat{A} = A^0 + \hat{A}^i \hat{\theta}_i$, $\hat{H} = \overline{H} + \hat{H}^i \hat{\theta}_i + \tfrac{1}{2} \hat{H}^{ij} \hat{\theta}_{ij} + \tfrac{1}{3!}H^{ijk} \hat{\theta}_{ijk}$ and there exists sections $B = \tfrac{1}{2}B_{ij} t^i \wedge t^j \in \Gamma( M , \wedge^2 (\mathfrak{t}^n)^*), \hat{B} = \tfrac{1}{2} \hat{B}^{ij} t_i \wedge t_j \in \Gamma( M , \wedge^2 \mathfrak{t}^n)$ satisfying the relations
\begin{equation*}
\begin{aligned}
\hat{g}^{ij} + \hat{B}^{ij} + \langle \hat{A}^i , \hat{A}^j \rangle & = (g_{ij} + B_{ij} + \langle A_i , A_j \rangle )^{-1} \\
\hat{A}^i &= A_j ( \hat{g}^{ji} + \hat{B}^{ji} + \langle \hat{A}^j , \hat{A}^i \rangle ) \\
H_i &= \hat{F}_i - B_{ij}\hat{F}^j + \langle A_i , A_j F^j + 2F_{A^0} \rangle \\
\hat{H}^i &= F^i - \hat{B}^{ij} \hat{F}_j + \langle \hat{A}^i , \hat{A}^j \hat{F}_j + 2F_{A^0} \rangle \\
H_{ij} &= dB_{ij} + \langle \nabla^0 A_i , A_j \rangle - \langle \nabla^0 A_j , A_i \rangle \\
\hat{H}^{ij} &= d\hat{B}^{ij} + \langle \nabla^0 \hat{A}^i , \hat{A}^j \rangle - \langle \nabla^0 \hat{A}^j , \hat{A}^i \rangle \\
H_{ijk} &= 2\langle A_i , [A_j , A_k] \rangle \\
\hat{H}^{ijk} &= 2 \langle \hat{A}^i , [ \hat{A}^j , \hat{A}^k ] \rangle,
\end{aligned}
\end{equation*}
where $\nabla^0$ is the covariant derivative associated to $A^0$ and $F_{A^0}$ the curvature of $A^0$.
\end{proposition}
\begin{proof}
The proof is almost the same as that of Proposition \ref{globalb}. Using the same arguments we can choose connections $\theta,\hat{\theta},A^0$ such that $\mathcal{H}$ is isomorphic to $TX \oplus \mathfrak{g}_P \oplus T^*X$ using connection $A^0$ and $3$-form $\overline{H} + \hat{F}^i \wedge \theta^i$. Taking advantage of $A$- and $B$-shifts we can choose the connections so that in addition $\mathcal{H}_-$ corresponds to the graph in $TX \oplus \mathfrak{g}_P \oplus T^*X$ of a triple $(g , A_i \theta^i , \tfrac{1}{2} B_{ij} \theta^{ij} )$.
\end{proof}
Triples $(g,A, H),(\hat{g} , \hat{A} , \hat{H} )$ related in this manner will be said to satisfy the {\em global heterotic Buscher rules}.

\begin{remark}
The requirement that $g_{ij} + \langle A_i , A_j \rangle$ be positive definite is equivalent to requiring that the lifted metric $g + \langle A , A \rangle$ is positive definite along the fibres of $P \to P_0$.
\end{remark}

Next we observe that the heterotic Buscher rules are compatible with the ordinary Buscher rules in the following sense. Suppose that we are in the setting of Proposition \ref{globalhb} so that in particular we have generalised metrics $\mathcal{H}_- \subseteq \mathcal{H}$, $\hat{\mathcal{H}}_- = \phi(\mathcal{H}_-) \subseteq \hat{\mathcal{H}}$ corresponding to triples $(g,A,H),(\hat{g},\hat{A},\hat{H})$ satisfying the global heterotic Buscher rules. Let $E$ be the exact Courant algebroid on $P$ such that $\mathcal{H}$ is obtained from $E$ by reduction and similarly let $\hat{E}$ be the exact Courant algebroid on $\hat{P}$ for which $\hat{\mathcal{H}}$ is obtained by reduction. According to Section \ref{genmetred} we have that the generalised metric $\mathcal{H}_-$ is obtained by reduction of a corresponding generalised metric $E_- \subset E$ and likewise $\hat{\mathcal{H}_-}$ is obtained from a generalised metric $\hat{E}_- \subset \hat{E}$. Recall that $E_-$ corresponds to the pair $(g',H')$ where $g' = g + 
\langle A , A \rangle$, $H' = H - CS_3(A)$ and similarly let $\hat{E}_-$ correspond to $(\hat{g}',\hat{H}')$.
\begin{proposition}\label{buscherlift}
The pairs $(g',H'),(\hat{g}',\hat{H}')$ satisfy the global Buscher rules.
\end{proposition}
\begin{proof}
Let $\theta' = \theta'^i t_i$ be the $T^n$-connection for $P \to P_0$ such that $g' = \overline{g}' + h_{ij} \theta'^i \theta'^j$ for some $h_{ij}$. Write $H'$ as $H' = \overline{H}' + H'_i \theta'^i + \tfrac{1}{2} H'_{ij} \theta'^{ij} + \tfrac{1}{3!} H'_{ijk} \theta'^{ijk}$. We have that $h_{ij} = g_{ij} + \langle A_i , A_j \rangle$. Similarly we obtain a $\hat{T}^n$-connection $\hat{\theta}'$ such that $\hat{g}' = \overline{g}'' + \hat{h}^{ij} \hat{\theta}'_i \hat{\theta}'_j$, $\hat{H}' = \overline{H}'' + \hat{H}'^i \hat{\theta}'_i + \tfrac{1}{2} \hat{H}'^{ij} \hat{\theta}'_{ij} + \tfrac{1}{3!} \hat{H}'^{ijk} \hat{\theta}'_{ijk}$. Since $(g,A,H)$ was assumed T-dualisable we have that $h_{ij}$ is positive definite. Let $h^{ij}$ be the inverse matrix and similarly let $\hat{h}_{ij}$ be the inverse of $\hat{h}^{ij}$. Then $\theta'^i = \theta^i + \delta^i$, where $\delta^i = h^{ij}\langle A_j , A^0 \rangle$ and $\hat{\theta}'_i = \hat{\theta}_i + \hat{\delta}_i$, where $\hat{\delta}_i = \hat{h}_{ij} \langle \
hat{A}^j , A^0 \rangle$. We find
\begin{equation*}
\begin{aligned}
\overline{g}' &= \overline{g} + \langle A^0 , A^0 \rangle - h_{ij} \delta^i \delta^j \\
\overline{H}' &= \overline{H} - CS_3(A^0) - F^i \wedge \hat{\delta}_i - \hat{F}_i \wedge \delta^i - h_{ij} d \delta^i \wedge \delta^j + \tfrac{1}{2} dB_{ij} \wedge \delta^{ij}\\
H'_i &= ( \hat{F}_i + d \hat{\delta}_i ) - B_{ij} ( F^j + d \delta^j) \\
H'_{ij} &= dB_{ij} \\
H'_{ijk} &= 0,
\end{aligned}
\end{equation*}
with similar identities for $\overline{g}'', \overline{H}'', \hat{H}'^i,\hat{H}'^{ij},\hat{H}'_{ijk}$. In order to verify the Buscher rules for the pairs $(g',H'),(\hat{g}',\hat{H}')$ it remains only to show that $\overline{g}' = \overline{g}''$ and $\overline{H}' = \overline{H}''$. These identities follow by straightforward computation.
\end{proof}


\section{Heterotic equations of motion and T-duality}\label{sechetd}


\subsection{Heterotic equations by reduction}

Given a metric $g$ (not necessarily positive definite) and $3$-form $H$ on $X$ we let $\nabla^g$ denote the Levi-Civita connection and define $\nabla^{g,H}$ by $\nabla_X^{g,H} Y = \nabla_X^g Y + \frac{1}{2} g^{-1}( i_Y i_X H)$. Let $R^{g,H}$ denote the curvature of $\nabla^{g,H}$ and $Ric^{g,H}$ the Ricci curvature defined by $Ric^{g,H}(X,Y) = Tr( Z \mapsto R^{g,H}(Z,X)Y )$. If $\varphi$ is a function on $X$ we consider the following equation for the triple $g,H,\varphi$
\begin{equation}\label{eint}
Ric^{g,H} + 2 \nabla^{g,H} d\varphi = 0.
\end{equation}
These equations occur in the low energy limit of type II string theories \cite{cfmp}. We refer to $\varphi$ as the dilaton and (\ref{eint}) as the {\em type II equations of motion}. Solutions are called {\em strong} if in addition the $3$-form $H$ is closed. One also considers the following equation arising from variation of the dilaton \cite{cfmp}:
\begin{equation}\label{eint2}
s^g + 4\Delta^g \varphi - 4 | d\varphi |^2 + \tfrac{1}{2} |H |^2 = 0,
\end{equation}
where $s^g$ is the scalar curvature of $g$, $\Delta^g f = g^{\mu \nu} \nabla^g_{\mu} \nabla^g_{\nu} f$ the Laplacian and $|H|^2 = \tfrac{1}{3!} H_{\alpha \beta \gamma} H^{\alpha \beta \gamma}$.\\

The heterotic equivalent of these equations are as follows. Let $P \to X$ be a principal $G$-bundle and fix an invariant symmetric non-degenerate pairing $c( \, \, , \, \, )$ on $\mathfrak{g}$. Consider a $4$-tuple $(g,A,H,\varphi)$ consisting of a metric $g$ on $X$ a $G$-connection $A$, a $3$-form $H$ on $X$ and a function $\varphi$ on $X$. Let $F$ be the curvature of $A$. Consider the following equations:
\begin{equation}\label{heint}
\begin{aligned}
Ric^{g,H}_{\mu \nu} + 2\nabla_\mu^{g,H} (d\varphi)_{\nu} - c( F_{\mu \alpha} , {F_{\nu}}^\alpha) &= 0 \\
-d_A^*(F)_\mu + 2 {F_\mu}^\alpha \partial_\alpha \varphi - \tfrac{1}{2} H_{\mu \alpha \beta} F^{\alpha \beta} & = 0.
\end{aligned}
\end{equation}
We refer to (\ref{heint}) as the {\em heterotic equations of motion}. Solutions are called {\em strong} if in addition we have $dH = c(F \wedge F)$. Once again we may consider an additional equation arising from the dilaton:
\begin{equation}\label{heint2}
s^g + 4\Delta^g \varphi - 4 | d\varphi |^2 + \tfrac{1}{2} |H|^2 - \frac{1}{2} \langle F_{\alpha \beta} , F^{\alpha \beta} \rangle = 0.
\end{equation}
The relation between Equations (\ref{heint}),(\ref{heint2}) and heterotic string theory deserves further clarification. In the low energy limit of heterotic string theory one has a $4$-tuple $(g,A^G,H,\varphi)$ and additionally a metric connection $A^{TX}$ on the tangent bundle. Let $F^G$ denote the curvature of $A^G$ and $R^{TX}$ the curvature of $A^{TX}$. Suppose $X$ has dimension $m$. Let $c_\mathfrak{g}$ denote an invariant pairing on $\mathfrak{g}$, $c_{\mathfrak{so}(m)}$ an invariant pairing on $\mathfrak{so}(m)$ and let $\alpha'$ be a positive constant. The equations of motion in the low energy limit of heterotic string theory are \cite{gpt},\cite{huto}
\begin{equation*}
\begin{aligned}
Ric^{g,H}_{\mu \nu} + 2\nabla_\mu^{g,H} (d\varphi)_{\nu} + \alpha' c_\mathfrak{g}({F^G}_{\mu \alpha} , {{F^G}_{\nu}}^\alpha) -  \alpha' c_{\mathfrak{so}(m)}( {R^{TX}}_{\mu \alpha} , {{R^{TX}}_{\nu}}^\alpha) &= 0 \\
-d_{A^G}^*(F^G)_\mu + 2 {(F^G)_\mu}^\alpha \partial_\alpha \varphi - \tfrac{1}{2} H_{\mu \alpha \beta} (F^G)^{\alpha \beta} & = 0 \\
-d_{A^{TX}}^*(R^{TX})_\mu + 2 {(R^{TX})_\mu}^\alpha \partial_\alpha \varphi - \tfrac{1}{2} H_{\mu \alpha \beta} (R^{TX})^{\alpha \beta} & = 0.
\end{aligned}
\end{equation*}
In addition solutions are required to satisfy the anomaly cancellation condition
\begin{equation}\label{anom}
dH = -\alpha' c_\mathfrak{g}( F^G \wedge F^G) + \alpha' c_{\mathfrak{so}(m)}( R^{TX} \wedge R^{TX} ).
\end{equation}
Lastly there is the equation corresponding to dilaton variation:
\begin{equation*}
s^g + 4\Delta^g \varphi - 4 | d\varphi |^2 + \tfrac{1}{2} |H|^2 + \frac{\alpha'}{2} c_\mathfrak{g}( {F^G}_{\alpha \beta} , {F^G}^{\alpha \beta} ) - \frac{\alpha'}{2} c_{\mathfrak{so}(m)}( {R^{TX}}_{\alpha \beta} , {R^{TX}}^{\alpha \beta} ) = 0.
\end{equation*}
We see that these equations are just a special case of (\ref{heint}) with gauge group $G \times O(m)$ and pairing $c = \alpha' (-c_\mathfrak{g} + c_{\mathfrak{so}(m)})$. Condition (\ref{anom}) is simply the requirement that solutions are strong. The pairings $c_\mathfrak{g},c_\mathfrak{so}(m)$ should be normalised so that at the level of cohomology classes (\ref{anom}) becomes the topological condition $p_1(P) - p_1(TX) = 0$.

\begin{proposition}\label{einstlift}
Let $\sigma \colon P \to X$ be a principal $G$-bundle and $c( \, \, , \, \, ) = \langle \, \, , \, \, \rangle$ an invariant pairing on $\mathfrak{g}$. A $4$-tuple $(g',A,H',\varphi)$ satisfies the heterotic equations of motion (\ref{heint}) if and only if the lifted $3$-tuple $(g,H,\varphi)$ satisfies the type II equations of motion (\ref{eint}), where
\begin{equation*}
\begin{aligned}
g &= g' + \langle A , A \rangle \\
H & = H' - CS_3(A).
\end{aligned}
\end{equation*}
Moreover $(g',A,H',\varphi)$ is a strong solution if and only if $(g,H,\varphi)$ is a strong solution. Let $e_1, \dots , e_k$ denote a basis $\mathfrak{g}$ with structure constants ${c_{ij}}^k$ given by $[e_i , e_j] = {c_{ij}}^k e_k$. Use the metric $a_{ij} = \langle e_i , e_j \rangle$ to raise and lower indices. Let $\beta^\varphi$ denote the left hand side of (\ref{eint2}) using $(g,H,\varphi)$ and $\beta^\varphi_{het}$ the left hand side of (\ref{heint2}) using $(g',A,H',\varphi)$. We have $\beta^\varphi_{het} = \beta^\varphi + \tfrac{1}{12} c_{ijk}c^{ijk}$.
\end{proposition}
\begin{proof}
Let $F$ denote the curvature of $A$. Using the connection $A$ we have a splitting $TP = \sigma^*(TX) \oplus \mathfrak{g}$. With respect to this splitting one finds that the Levi-Civita connections for $g$ and $g'$ are related as follows:
\begin{equation*}
\begin{aligned}
\nabla^g_\mu (\partial_\nu) &=\nabla^{g'}_\mu(\partial_\nu) - \tfrac{1}{2} F_{\mu \nu} \\
\nabla^g_\mu(a) &=\nabla_\mu^A(a) + \tfrac{1}{2} \langle {F_\mu}^\nu , a \rangle \partial_\nu \\
\nabla^g_a ( \partial_\mu ) &= \tfrac{1}{2} \langle {F_\mu}^\nu , a \rangle \partial_\nu \\
\nabla^g_a(b) &= -\tfrac{1}{2} [a , b ].
\end{aligned}
\end{equation*}
where $a,b$ are sections of the adjoint bundle $\mathfrak{g}_P$ and $\nabla^A$ is the covariant derivative associated to $A$. Considering now connections with torsion one finds
\begin{equation}\label{levic}
\begin{aligned}
\nabla^{g,H}_\mu (\partial_\nu) &= \nabla^{g',H'}_\mu (\partial_\nu) - F_{\mu \nu} \\
\nabla^{g,H}_\mu(a) &=\nabla_\mu^A(a) + \langle {F_\mu}^\nu , a \rangle \partial_\nu \\
\nabla^{g,H}_a ( \partial_\mu ) &= 0 \\
\nabla^{g,h}_a(b) &= 0.
\end{aligned}
\end{equation}
From this it follows that
\begin{equation*}
\begin{aligned}
Ric^{g,H}_{\mu \nu} + 2 \nabla^{g,H}_\mu (d\varphi)_\nu &= Ric^{g',H'}_{\mu \nu} +2 \nabla_\mu^{g',H'}(d\varphi)_\nu - c( F_{\mu \alpha} , {F_\nu}^{\alpha} ) \\
Ric^{g,H}(\partial_\mu , a) + 2\nabla^{g,H}_\mu( d\varphi)(a)  &= -c( a , -d^*_A(F)_\mu + 2{F_\mu}^\nu \partial_\nu \varphi - \tfrac{1}{2}H'_{\mu \alpha \beta} F^{\alpha \beta} ) \\
Ric^{g,H}(a , \partial_\mu) + 2\nabla^{g,H}_a( d\varphi)_\mu &=0\\
Ric^{g,H}(a,b) + 2 \nabla^{g,H}_a( d\varphi)(b) &= 0.
\end{aligned}
\end{equation*}
Thus $(g',A,H',\varphi)$ satisfies (\ref{heint}) if and only if $(g,H,\varphi)$ satisfies (\ref{eint}). The claim about strong solutions follows since $dH = dH' - c(F \wedge F)$. Making use of (\ref{levic}) we obtain the identity $\beta^\varphi_{het} = \beta^\varphi + \tfrac{1}{12} c_{ijk}c^{ijk}$.
\end{proof}


\subsection{Type II equations and T-duality}

Suppose $(g,H),(\hat{g} , \hat{H})$ satisfy the global Buscher rules. Using the connections $\theta,\hat{\theta}$ we have splittings $TX = TM \oplus \mathfrak{t}^n$, $T\hat{X} = TM \oplus (\mathfrak{t}^n)^*$. Define isomorphisms $\psi_{\pm} \colon TX \to T\hat{X}$ by setting $\psi_{\pm}( X , 0 ) = X$, $\psi_+( (0,t_i)) = -(g_{ij} + B_{ij}) t^j$, $\psi_-( (0, t_i)) = (g_{ji} + B_{ji}) t^j$.
\begin{proposition}\label{eetd}
Let $\varphi$ be a function on $M$ and set $\hat{\varphi} = \varphi - \frac{1}{2} log (det( g_{ij} + B_{ij} ))$. Then
\begin{equation}\label{ricci}
( \psi_-^* \otimes \psi_+^* ) (  Ric^{\hat{g},\hat{H}} + 2 \nabla^{\hat{g},\hat{H}} d\hat{\varphi} ) = Ric^{g,H} + 2\nabla^{g,H} d\varphi.
\end{equation}
In particular $(g,H,\varphi)$ is a solution of the type II equations of motion if and only if $(\hat{g},\hat{H},\hat{\varphi})$ is a solution. Let $\beta^\varphi$ denote the left hand side of (\ref{eint2}) using $(g,H,\varphi)$ and $\hat{\beta}^{\hat{\varphi}}$ the left hand side of (\ref{eint2}) using $(\hat{g},\hat{H},\hat{\varphi})$. Then $\beta^\varphi = \hat{\beta}^{\hat{\varphi}}$.
\end{proposition}
\begin{proof}
The proof is by direct calculation of $Ric^{g,H}$ and $Ric^{\hat{g},\hat{H}}$. First write $H$ as $H = \overline{H} + H_i \wedge \theta^i + \frac{1}{2} H_{ij} \theta^{ij}$, so $H_i = \hat{F}_i - B_{ij}F^j$, $H_{ij} = dB_{ij}$ and define $f^i_+ = \frac{1}{2}( F^i - g^{ij}H_j)$, $f^i_- = \frac{1}{2}( F^i + g^{ij}H_j)$ and $t_{\mu ij} = \frac{1}{2} \partial_\mu ( g_{ij} + B_{ij} )$. We use $\overline{g}$ and $g_{ij}$ to raise and lower indices. With respect to the splitting $TX = TM \oplus \mathfrak{t}^n$ one finds:
\begin{equation*}
\begin{aligned}
Ric^{g,H}_{\mu \nu} + 2\nabla^{g,H}_\mu (d\varphi)_\nu &= Ric^{\overline{g},\overline{H}}_{\mu \nu} + 2\nabla_\mu^{\overline{g},\overline{H}}( d\varphi - \frac{1}{2} {{t_\lambda}^k}_kdx^\lambda)_\nu \\ & \;\; \; \; \; \; \; \; -g_{ij}({f^i_+}_{\mu \alpha} {{f^j_+}_{\nu}}^\alpha + {f^i_-}_{\mu \alpha} {{f^j_-}_{\nu}}^\alpha) - t_{\mu ij} {t_\nu}^{ij}\\
Ric^{g,H}_{\mu i} +2\nabla^{g,H}_\mu(d\varphi)_i &= -2g_{ij}{{f^j_+}_\mu}^\lambda ( \partial_\lambda \varphi - \frac{1}{2} {{t_\lambda}^k}_k ) \\
& \; \; \; \; \; \; \; \; + d^*( g_{ij}f^j_+)_\mu + \frac{1}{2} H_{\mu \alpha \beta} g_{ij}{f^j_+}^{\alpha \beta} + {t^\alpha}_{ij}( {f^j_-}_{\mu \alpha} - {f^j_+}_{\mu \alpha} ) \\
Ric^{g,H}_{i \mu} +2\nabla^{g,H}_i(d\varphi)_\mu &= -2g_{ij}{{f^j_-}_\mu}^\lambda ( \partial_\lambda \varphi - \frac{1}{2} {{t_\lambda}^k}_k ) \\
& \; \; \; \;\; \; \; \; + d^*( g_{ij}f^j_-)_\mu + \frac{1}{2} H_{\mu \alpha \beta} g_{ij}{f^j_-}^{\alpha \beta} + {t^\alpha}_{ji}( {f^j_+}_{\mu \alpha} - {f^j_-}_{\mu \alpha} )\\
Ric^{g,H}_{ij} + 2\nabla_i^{g,H}(d\varphi)_j &= 2 {t^\lambda}_{ji}( \partial_\lambda \varphi - \frac{1}{2} {{t_\lambda}^k}_k ) \\
& \; \; \; \; \; \; \; \; + g_{ik}g_{jl}{f^k_-}_{\alpha \beta} {f^l_+}^{\alpha \beta} - \nabla^{\overline{g}}_\alpha( {t^\mu}_{ji} \partial_\mu )^\alpha + 2{t_\alpha}_{ki} {{t^\alpha}_j}^k,
\end{aligned}
\end{equation*}
with similar expressions for $Ric^{\hat{g} , \hat{H}} + 2 \nabla^{\hat{g},\hat{H}} \hat{\varphi}$. From here it is simply a matter of direct computation to verify the validity of (\ref{ricci}). Similarly a direct computation yields the identity $\beta^\varphi = \hat{\beta}^{\hat{\varphi}}$.
\end{proof}


\subsection{T-duality invariance of the heterotic equations} \label{secheof}

Putting together the results of the previous sections we are at last able to prove that heterotic T-duality preserves the heterotic equations of motion. Suppose we are again in the setting of heterotic T-duality and moreover that we have generalised metrics $\mathcal{H}_-$, $\hat{\mathcal{H}}_-$ related as in Proposition \ref{globalhb}. Let $\mathcal{H}_-$ correspond to the triple $(g,A,H)$ and $\hat{\mathcal{H}}_-$ to $(\hat{g},\hat{A},\hat{H})$, so that $(g,A,H),(\hat{g},\hat{A},\hat{H})$ satisfy the global heterotic Buscher rules.
\begin{proposition}\label{prophetequpres}
Let $\varphi$ be a function on $M$ and set $\hat{\varphi} = \varphi -\tfrac{1}{2}( log (det(g_{ji} + B_{ij} + \langle A_i , A_j \rangle )) )$. Then $(g,A,H,\varphi)$ is a solution of the heterotic equations of motion if and only if $(\hat{g} , \hat{A} , \hat{H} , \hat{\varphi} )$ is a solution. Moreover $(g,A,H,\varphi)$ satisfies (\ref{heint2}) if and only if $(\hat{g},\hat{A},\hat{H},\hat{\varphi})$ does.
\end{proposition}
\begin{proof}
By Proposition \ref{einstlift} we have that $(g,A,H,\varphi)$ satisfies the heterotic equations if and only if the lift $(g+ \langle A,A \rangle , H-CS_3(A) , \varphi)$ satisfies the type II equations. Similarly we lift $(\hat{g},\hat{A},\hat{H},\hat{\varphi})$ to $(\hat{g} + \langle \hat{A} , \hat{A} \rangle , \hat{H} - CS_3(\hat{A}) , \hat{\varphi})$. From Proposition \ref{buscherlift} we have that $(g+ \langle A,A \rangle , H-CS_3(A) , \varphi)$ and $(\hat{g} + \langle \hat{A} , \hat{A} \rangle , \hat{H} - CS_3(\hat{A}) , \hat{\varphi})$ are related by the global Buscher rules, hence by Proposition \ref{eetd} one satisfies the type II equations if and only if the other does. The last statement of the proposition holds by the same argument.
\end{proof}


\bibliographystyle{amsplain}

\end{document}